\documentclass[11pt,reqno,a4paper]{amsart}

\usepackage{amsmath,amssymb,amsfonts}
\usepackage{tikz}
\usepackage{hyperref}
\usepackage{appendix}
\usepackage{scalerel}
\usepackage{graphicx}
\usepackage{subcaption}
\usepackage{amsthm}
\usepackage{multirow}
\usepackage{pdflscape}
\usepackage{afterpage}
\usepackage{capt-of} 
\usepackage{lipsum}
\usepackage{booktabs}
\usepackage{siunitx}
\usepackage{esvect}
\usepackage{textcomp}
\usepackage[normalem]{ulem} 

\makeatletter
\def\set@curr@file#1{%
  \begingroup
    \escapechar\m@ne
    \xdef\@curr@file{\expandafter\string\csname #1\endcsname}%
  \endgroup
}
\def\quote@name#1{"\quote@@name#1\@gobble""}
\def\quote@@name#1"{#1\quote@@name}
\def\unquote@name#1{\quote@@name#1\@gobble"}
\makeatother
\usepackage{graphics}

\theoremstyle{plain}
\newtheorem{thm}{Theorem}
\newtheorem{lem}[thm]{Lemma}
\newtheorem{prop}[thm]{Proposition}

\newtheorem{maintheorem}{Theorem}

\theoremstyle{definition}

\theoremstyle{remark}
\newtheorem*{rem}{Remark}

\theoremstyle{plain}
\newtheorem{fact}{Fact}

\newcommand{\Rr}{\mathbb{R}}

\newcommand{\NN}{{\mathbb{N}}}

\newcommand{\RR}{\mathbb{R}}

\newcommand{\FF}{\mathcal{F}}

\newcommand{\etal}{\textit{et al}. }

\newcommand{\inter}{{\rm int}}

\newcommand{\RN}[1]{%
  \textup{\uppercase\expandafter{\romannumeral#1}}%
}

\author[Telmo Peixe]{Telmo Peixe}
\address{\small{ISEG-Lisbon School of Economics \& Management, Universidade de Lisboa, REM-Research in Economics and Mathematics, CEMAPRE-Centro de Matem\'atica Aplicada \`a Previs\~ao e Decis\~ao Econ\'omica.}}

\author[Alexandre A. Rodrigues]{Alexandre A. Rodrigues}
\address{\small{Centro de Matem\'atica and Faculdade de Ci\^encias, Universidade do Porto.}}

\email{telmop@iseg.ulisboa.pt, alexandre.rodrigues@fc.up.pt}

\begin{document}


\title[Persistent Strange attractors in 3D Polymatrix Replicators]
{Persistent strange attractors \\ in 3D Polymatrix Replicators}

\date{\today}

 \begin{abstract}
We introduce a one-parameter family of polymatrix replicators defined in a three-dimensional cube and study its bifurcations.
For a given interval of parameters, each element of the family can be $C^2$-approximated by a vector field whose flow exhibits suspended horseshoes and persistent strange attractors.
The proof relies on the  existence of a Shilnikov homoclinic cycle to the interior equilibrium.
We also describe the phenomenological steps responsible for the transition from regular to chaotic dynamics in our system (route to chaos).
 \end{abstract}

\maketitle

\section{Introduction}\label{sec:intro}

The \textit{polymatrix replicator}, introduced by Alishah, Duarte, and Peixe~\cite{alishah2015hamiltonian,alishah2015conservative}, is a system of ordinary differential equations developed to study the dynamics of the \textit{polymatrix game.}
It models the time evolution of the strategies that individuals from a stratified population choose to interact with each other.
These systems extend the class of \emph{bimatrix replicator} equations studied in~\cite{schuster1981coyness,schuster1981selfregulation,aguiar2011there}
to the study of the replicator dynamics in a population divided in a finite number of groups.

The polymatrix replicator induces a flow in a polytope defined by a finite product of simplices.
Alishah \etal~\cite{alishah2019asymptotic} presented a new method to study the asymptotic dynamics of flows defined on polytopes;
polymatrix replicators are a class examples of these flows.
Such dynamical systems arise naturally in the context of Evolutionary Game Theory (EGT) developed by Smith and Price~\cite{smith1973logic}.
We address the reader to Section 8 of Skyrms \cite{skyrms1992chaos} where a historical overview about  evolutionary game dynamics is given, including relations with the Lotka-Volterra  and the May-Leonard systems.

Smale \cite{smale1976differential} proved that strange attractors may be found in
ecological systems of $n\geq 5$ species in competition governed by Volterra equations.
Arneodo \etal \cite{arneodo1982strange} and Vano \etal \cite{vano2006chaos}, suggested  that chaos may be possible in Lotka-Volterra systems of $n=4$ species in competition.
Aiming a general setting where strange attractors may be observable, Arneodo \etal \cite{arneodo1980occurrence}  suggested the occurrence of chaos for $n=3$ species, not necessarily in competition.
This value of $n$ corresponds to the dimension of the phase space of the associated Lotka-Volterra system.
In all these references, the existence of chaos  has been achieved via  a homoclinic cycle to a saddle-focus.
In all cases, either the proof is mostly numerical or is the consequence of the existence of a Lorenz or a H\'enon attractor.

A strange attractor  is an invariant set with at least one positive Lyapunov exponent whose basin of attraction has non-empty interior.
Nowadays, at least for families of dissipative systems, chaotic dynamics is mostly understood as the persistence of strange attractors (occurring within a positive Lebesgue measure set of parameters) \cite{barrientos2019emergence}.
Persistence of dynamics is physically relevant because it means that the phenomenon is ``observable'' with positive probability.
The rigorous proof of the existence of a strange attractor is a great challenge. 

Finding explicit examples of three-dimensional vector fields  in the context of evolutionary games, whose flows exhibit chaos is of significant interest -- at this point, it is worth to see the system developed in ~\cite{accinelli2020power} to model social corruption.

In the present paper, by combining numerical and theoretical techniques, we construct a one-parameter family of polymatrix replicators containing  elements that can be $C^2$-approximated by vector fields exhibiting persistent strange attractors. This article is organised as follows. In Section \ref{sec:desc} we introduce the one-parameter family of polymatrix replicators that will be the focus of our work.
In Section \ref{terminology} we define the main concepts used throughout the article and state the main result.
Then in Section \ref{s: bifurcation analysis} we concentrate our analysis on  a parameter interval  where a single interior equilibrium exists. We enumerate all equilibria that appear on the  boundary of the phase space  and we study their Lyapunov stability.
Moreover, we numerically find the parameter values where the interior equilibrium undergoes local and global bifurcations.

We present in Section~\ref{s: numerical analysis} a numerical analysis which supports the description of the \textit{route to chaos} as well as the proof of the existence of strange attractors.
We compute  the Lyapunov exponents and characterise the maximal attracting    set as the parameter evolves.

In Section \ref{s:route}, we study the dynamics of the system in the interior of the phase space, stressing seven different topological scenarios: 
Cases I -- VII. We emphasize that the dynamics in the phase space's interior  is highly governed by the dynamics of the equilibria  on the faces.

We revive the arguments by Shilnikov and Ovsyannikov  \cite{shilnikov1965case, ovsyannikov1986systems}  and Mora and Viana \cite{mora1993abundance}  in Section  \ref{s: second part} to prove the existence of persistent strange attractors for a family of vector fields close to the  one-parameter family, concluding the proof of our main result.

Finally, in Section \ref{s: discussion} we relate our main results  with others in the literature, emphasising  the phenomenological scenario responsible for the emergence of strange attractors.

In order to facilitate the stability analysis, in Appendix \ref{tables_appendix} we exhibit tables with the explicit expression of the eigenvalues for the equilibria on the boundary, as well as their signs for different values of the parameter. 
In Appendices \ref{bd_appendix} and ~\ref{int_appendix}, we present a set of frames collecting the main   metamorphoses of the non-wandering set from a global attracting equilibrium to chaos.
We have endeavoured to make a self contained exposition bringing together all topics related to the proofs.

\section{Model description}\label{sec:desc}

Consider a population divided in three groups where individuals of each group have exactly two strategies to interact with other members of the population.
Based on~\cite{alishah2015hamiltonian,alishah2015conservative}, the model that we will consider to study the time evolution of the chosen strategies is the \emph{polymatrix game} and may be formalised as:
\begin{equation}\label{eq:poly_rep}
\dot{x}_i^\alpha(t) = x_i^\alpha (t) \left( (P x(t))_i^\alpha-\sum_{j=1}^2 (x_j^{\alpha}(t)) (P x(t))_j^\alpha \right), \alpha\in\{1,2,3\}, 
i\in\{1,2\},
\end{equation}
where $\dot{x}_i^\alpha(t)$ represents the time derivative of $x_i^\alpha(t)$, $P\in M_{6\times 6}(\Rr)$ is the payoff matrix,
$$
x(t)=\left(x_1^1(t),x_2^1(t),x_1^2(t),x_2^2(t),x_1^3(t),x_2^3(t)\right)
$$
and
$$
x_1^1(t)+x_2^1(t)=x_1^2(t)+x_2^2(t)=x_1^3(t)+x_2^3(t)=1.
$$
For simplicity of notation   will write $x$ instead of $x(t)$.
Since we are considering a population divided in three groups, each one with two possible strategies, the payoff matrix $P$ can be represented as a matrix,  
$$
P=\left(
\begin{array}{c|c|c}
P^{1,1}  & P^{1,2} & P^{1,3} \\
\hline\\[-4mm]
P^{2,1}  & P^{2,2} & P^{2,3} \\
\hline\\[-4mm]
P^{3,1}  & P^{3,2} & P^{3,3}
\end{array}
\right) = \left(
\begin{array}{cc|cc|cc}
p_{1,1}^{1,1}  & p_{1,2}^{1,1} & p_{1,1}^{1,2}  & p_{1,2}^{1,2} & p_{1,1}^{1,3}  & p_{1,2}^{1,3} \\[2mm]
p_{2,1}^{1,1}  & p_{2,2}^{1,1} & p_{2,1}^{1,2}  & p_{2,2}^{1,2} & p_{2,1}^{1,3}  & p_{2,2}^{1,3} \\[1mm]
\hline\\[-3mm]
p_{1,1}^{2,1}  & p_{1,2}^{2,1} & p_{1,1}^{2,2}  & p_{1,2}^{2,2} & p_{1,1}^{2,3}  & p_{1,2}^{2,3} \\[2mm]
p_{2,1}^{2,1}  & p_{2,2}^{2,1} & p_{2,1}^{2,2}  & p_{2,2}^{2,2} & p_{2,1}^{2,3}  & p_{2,2}^{2,3} \\[1mm]
\hline\\[-3mm]
p_{1,1}^{3,1}  & p_{1,2}^{3,1} & p_{1,1}^{3,2}  & p_{1,2}^{3,2} & p_{1,1}^{3,3}  & p_{1,2}^{3,3} \\[2mm]
p_{2,1}^{3,1}  & p_{2,2}^{3,1} & p_{2,1}^{3,2}  & p_{2,2}^{3,2} & p_{2,1}^{3,3}  & p_{2,2}^{3,3}
\end{array}
\right)\,,
$$
where each block $P^{\alpha, \beta}$, $\alpha, \beta \in \{1,2,3\}$, represents the payoff of the individuals of the group $\alpha$ when interacting with individuals of the group $\beta$, and where each entry $p_{i,j}^{\alpha, \beta}$ represents the average payoff of an individual of the group $\alpha$ using strategy $i$ when interacting with an individual of the group $\beta$ using strategy $j$.

In this setting, we can interpret equation~\eqref{eq:poly_rep} in the following way:
assuming random encounters between individuals of the population, for each group $\alpha\in\{1,2,3\}$,
the average payoff for a strategy $i\in\{1,2\}$, is given by
$$
(Px)_i^\alpha=\sum_{\beta=1}^3 \left( P^{\alpha,\beta} \right)_i^\alpha x^\beta
= \sum_{\beta=1}^3 \sum_{k=1}^2 p_{i,k}^{\alpha,\beta}x_k^\beta\,,
$$
the average payoff of all strategies in $\alpha$ is given by
$$
\sum_{i=1}^2 x_i^{\alpha} \left( Px \right)_i^\alpha = \sum_{\beta=1}^3 (x^{\alpha})^T P^{\alpha,\beta} x^{\beta}\,,
$$
and the growth rate $ \dfrac{\dot{x}_i^{\alpha}}{x_i^{\alpha}}$ of the frequency of each strategy $i\in\{1,2\}$
is equal
to the payoff difference
$$ (Px)_i^\alpha - \sum_{\beta=1}^3 (x^{\alpha})^T P^{\alpha,\beta} x^{\beta}. $$
For simplicity of notation, we consider $x=(x_1,x_2,x_3,x_4,x_5,x_6)$, where 
\begin{equation}
\label{sum1}
x_1+x_2=x_3+x_4=x_5+x_6=1.
\end{equation}
Then, system~\eqref{eq:poly_rep} may be written as
\begin{equation}\label{eq:poly_rep_2}
\left \{ \begin{array}{l}
\dot{x}_i=x_i\left((P x)_i-x_i(P x)_i-x_{i+1}(P x)_{i+1}\right) \\[2mm]
\dot{x}_{i+1}=x_{i+1}\left((P x)_{i+1}-x_i(P x)_i-x_{i+1}(Px)_{i+1}\right)
\end{array} \right. , \quad i\in\{1,3,5\}.
\end{equation}

\begin{lem}
System~\eqref{eq:poly_rep_2} is equivalent to
\begin{equation} \label{eq:poly_rep_3}
\left \{ \begin{array}{l}
\dot{x}_1=x_1(1-x_1)\left((P x)_1-(Px)_2\right) \\[2mm]
\dot{x}_3=x_3(1-x_3)\left((P x)_3-(Px)_4\right) \\[2mm]
\dot{x}_5=x_5(1-x_5)\left((P x)_5-(Px)_6\right) 
\end{array} \right. ,
\end{equation}
where $\dot{x}_2=-\dot{x}_1$, $\dot{x}_4=-\dot{x}_3$, and $\dot{x}_6=-\dot{x}_5$.
\end{lem}

\begin{proof}
Let $ i\in\{1,3,5\}$. Since $x_i+x_{i+1}=1$, from~\eqref{eq:poly_rep_2} we deduce that
\begin{align*}
\dot{x}_i &= x_i\left((P x)_i-x_i(P x)_i-x_{i+1}(Px)_{i+1}\right) \\
 		&= x_i\left((1-x_i)(P x)_i-(1-x_i)(Px)_{i+1}\right) \\
 		&= x_i(1-x_i)\left((P x)_i-(Px)_{i+1}\right) ,
\end{align*}
and $$\dot{x}_{i+1}=-\dot{x}_i.$$
\end{proof}

\begin{center}
\footnotesize{
\begin{tabular}{|c|c|c|} \toprule
Vertex  									&  $\Rr^3$ 					 	&  $\Rr^6$  		\\ \bottomrule \toprule
$\quad v_1\quad$ 	& $\quad (0,0,0)\quad$ 	& $\quad (1,0,1,0,1,0)\quad$	\\ \midrule
$v_2$ 					& $(0,0,1)$ 						& $(1,0,1,0,0,1)$  				\\ \midrule
$v_3$ 					& $(0,1,0)$ 						& $(1,0,0,1,1,0)$  				\\ \midrule
$v_4$ 					& $(0,1,1)$ 						& $(1,0,0,1,0,1)$  				\\ \midrule
$v_5$ 					& $(1,0,0)$ 						& $(0,1,1,0,1,0)$  				\\ \midrule
$v_6$ 					& $(1,0,1)$ 						& $(0,1,1,0,0,1)$  				\\ \midrule
$v_7$ 					& $(1,1,0)$ 						& $(0,1,0,1,1,0)$  				\\ \midrule
$v_8$ 					& $(1,1,1)$ 						& $(0,1,0,1,0,1)$  				\\ \bottomrule
\end{tabular}
\quad 
\begin{tabular}{|c|c|} \toprule
Face									& Vertices					 	   	\\ \bottomrule \toprule
$\quad  \sigma_1 \quad$ 	& $\{ v_5, v_6, v_7, v_8 \}$	\\ \midrule
$\quad  \sigma_2 \quad$ 	& $\{ v_1, v_2, v_3, v_5 \}$	\\ \midrule
$\quad  \sigma_3 \quad$ 	& $\{ v_3, v_4, v_7, v_8 \}$	\\ \midrule
$\quad  \sigma_4 \quad$ 	& $\{ v_1, v_2, v_5, v_6 \}$	\\ \midrule
$\quad  \sigma_5 \quad$  	& $\{ v_2, v_4, v_6, v_8 \}$	\\ \midrule
$\quad  \sigma_6\quad$ 	& $\{ v_1, v_3, v_5, v_7 \}$ \\ \bottomrule
\end{tabular}
        \captionof{table}{\small{Representation of the eight vertices of $[0,1]^3$ in $\Rr^3$ and  $\Gamma_{(2,2,2)}$ in $\Rr^6$, and the identification of the six faces according to vertices they contain.}}
        \label{tbl:Representation_of_vs}
}
\end{center}

  The phase space of system \eqref{eq:poly_rep_3} is the prism  $\Gamma_{(2,2,2)} :=\Delta^1\times \Delta^1 \times \Delta^1 \subset\Rr^6, $ where
$$
\Delta^1 = \{ (x_i,x_{i+1})\in\Rr^2 \,|\, x_i+x_{i+1}=1,\, x_i,x_{i+1}\geq 0 \}, \qquad  i\in\{1,3,5\}.
$$

Fixing a referential on $\Rr^3$, by~\eqref{sum1} we can define a bijection between $ \Gamma_{(2,2,2)} \subset\Rr^6$ and $[0,1]^3\subset \Rr^3$.
We  identify $(1,0,1,0,1,0)\in\Gamma_{(2,2,2)}$ with $(0,0,0)\in [0,1]^3$.
In Table~\ref{tbl:Representation_of_vs} (left) we identify each vertex of the cube $[0,1]^3$ with a vertex on the prism $\Gamma_{(2,2,2)}$.

Given the polymatrix replicator~\eqref{eq:poly_rep}, by~\cite[Proposition 1]{alishah2015conservative}, we may obtain an equivalent game (in the sense that the corresponding vector fields are the same) with a payoff matrix whose second row of each group has $0$'s in all of its entries.
From now on, we will analyse system~\eqref{eq:poly_rep_3} with payoff matrix
$$
P_{\boldsymbol{\mu}}=\left(
\begin{array}{cccccc}
 \boldsymbol{\mu}  & 14 & -10 & 10 & -2 & 2 \\
 0 & 0 & 0 & 0 & 0 & 0 \\
 10 & -10 & 2 & -2 & -2 & 2 \\
 0 & 0 & 0 & 0 & 0 & 0 \\
 -25 & 29 & 0 & -11 & -2 & 2 \\
 0 & 0 & 0 & 0 & 0 & 0 \\
\end{array}
\right)\,.
$$
This defines a polynomial vector field $f_{\boldsymbol{\mu}}$ on the compact set $\Gamma_{(2,2,2)}\equiv[0,1]^3$
whose flow is denoted by $\phi_P^t$,   for ${\boldsymbol{\mu}}\in \RR$.

\begin{rem}
The finding of an explicit expression for $P_{\boldsymbol{\mu}}$ has been motivated by the work of Arneodo~\etal~\cite{arneodo1980occurrence}, and its  finding has been possible due to the numerical experience of the first author in previous works~\cite{alishah2015conservative, peixe2015lotka, peixe2019permanence}. 
\end{rem}

\begin{lem}
\label{Lemma3}
The prism $\Gamma_{(2,2,2)}$ is flow-invariant for system~\eqref{eq:poly_rep_3}.
\end{lem}

\begin{proof}
Concerning system~\eqref{eq:poly_rep_3}, notice that, for each $i\in\{1,3,5\}$, if $x_i\in\{0,1\}$ then $\dot{x}_i=0$ (i.e. initial conditions starting at the faces, stay there for all $t\in \Rr$).
\end{proof}

By compactness of $\Gamma_{(2,2,2)}$, the flow associated to system~\eqref{eq:poly_rep_3} is complete, i.e. all solutions are defined for all $t\in \RR$.
From now on, let $\left( (2,2,2), P_{\boldsymbol{\mu}} \right)$ be the polymatrix game associated to~\eqref{eq:poly_rep_3}. For $P_{\boldsymbol{\mu}}$, system~\eqref{eq:poly_rep_3} becomes

\begin{equation}\label{eq:poly_rep_5}
\left \{ \begin{array}{l}
\dot{x_1}=x_1(1-x_1)(P_{\boldsymbol{\mu}}\, x)_1 \\[2mm]
\dot{x_3}=x_3(1-x_3)(P_{\boldsymbol{\mu}}\, x)_3 \\[2mm]
\dot{x_5}=x_5(1-x_5)(P_{\boldsymbol{\mu}}\, x)_5 
\end{array} \right. .
\end{equation}

Using~\eqref{sum1} and considering $x=x_2$, $y=x_4$, $z=x_6$, the equation \eqref{eq:poly_rep_5} is equivalent to
\begin{equation}\label{eq:poly_rep_8}
\left \{ \begin{array}{l}
\dot{x}=x(1-x)\left(12- \boldsymbol{\mu}+(\boldsymbol{\mu}-14) x-20 y - 4 z \right) \\[2mm]
\dot{y}=y(1-y)\left( -10 +20 x + 4 y - 4 z \right) \\[2mm]
\dot{z}=z(1-z)\left( 27 - 54 x + 11 y - 4 z \right)
\end{array} \right. .
\end{equation}

The one-parameter polynomial vector field associated to \eqref{eq:poly_rep_8} will be denoted by $f_{\boldsymbol{\mu}}$. Let us denote by $\mathcal{X}$ the set of $C^2$-vector fields on $\Rr^3$ that leave the cube $[0,1]^3$ invariant.
A similar model has been used by Accinelli \etal \cite{accinelli2020power} to study the power of voting and corruption cycles in democratic societies.

\begin{rem}
In the transition from \eqref{eq:poly_rep_5} to \eqref{eq:poly_rep_8},
we have identified the point $(1,0,1,0,1,0)\in\Gamma_{(2,2,2)}$ with   $(0,0,0)\in \Rr^3$. 
\end{rem}


\section{Terminology and main result}\label{terminology}
In this section we define the main concepts used throughout the article and  we state the main result.
For $I\subset \Rr$ and $n\in\NN$, we consider  a smooth one-parameter family of vector fields $(f_{\mu})_{\mu\in I}$ on $\RR^{n}$,
with flow given by the unique solution  $u(t)=\phi^t(u_{0})$  of
\begin{equation}
\dot{u}=f_\mu(u),\qquad \phi^0(u_0)=u_{0}, 
\label{sistema geral}
\end{equation}
where $u_{0} \in \RR^{n},$ $t\in \RR$, and $\mu\in I$.
If $A\subseteq \RR^n$, we denote by $\inter\left(A\right)$ and $\overline{A}$
the topological interior and the closure of $A$, respectively.

For a solution of~\eqref{sistema geral} passing through $u_0\in \RR^n$, the set of its accumulation points as $t$ goes to $+\infty$ is the $\omega$-limit set of $u_0$ and will be denoted by $\omega(u_0)$. More formally, 
$$
\omega(u_0)=\bigcap_{T=0}^{+\infty} \overline{\left(\bigcup_{t>T}\phi^t (u_0)\right)}.
$$ 
It is well known that $\omega(u_0)$  is closed and flow-invariant, and if the $\phi$-trajectory of $u_0$ is contained in a compact set, then 
$\omega(u_0)$ is non-empty \cite{guckenheimer2013nonlinear}.

For the sake of completeness, we describe  the main features of the local codimension-one bifurcations studied in this paper.
We say that an equilibrium $O_\mu$ of \eqref{sistema geral} undergoes:
\begin{enumerate}
\item  a \emph{transcritical bifurcation} if $Df_\mu(O_\mu)$ has an eigenvalue whose real part passes through zero and interchanges its stability with another equilibrium as the parameter varies;
\item a \emph{Belyakov transition} if it changes from a node to a focus or vice-versa, i.e. if there is at least a pair of  eigenvalues of $Df_\mu(O_\mu)$ changing from real to complex (conjugate), as long as the sign of the real part is the same;
\item a \emph{supercritical Hopf bifurcation}  if it changes from an attracting focus to an unstable one and generates an attracting periodic solution.
\end{enumerate}

We say that a non-trivial periodic solution of~\eqref{sistema geral} undergoes a \emph{period-doubling bifurcation}
when a small perturbation of the system produces a new periodic solution, doubling the period of the original one.
The linear and nonlinear conditions which guarantee the existence of such bifurcations may be found in  \cite{guckenheimer2013nonlinear}.

For $m\in\NN$, given two hyperbolic saddles
$A$ and $B$ associated to the flow of~\eqref{sistema geral}, an $m$-dimensional \emph{heteroclinic connection } from $A$ to $B$, denoted by $[A\rightarrow B]$, is a  connected and flow-invariant $m$-dimensional manifold contained in $W^{u}(A)\cap W^{s}(B)$. There may be more than one connection from  $A$ to $B$ (see Field~\cite{field2020lectures}). 

Let $\mathcal{S=}\{A_{j}:j\in \{1,\ldots,k\}\}$ be a finite ordered set of hyperbolic equilibria.
We say that there is a {\em heteroclinic cycle }associated to $\mathcal{S}$ if 
\begin{equation*}
\forall j\in \{1,\ldots,k\},W^{u}(A_{j})\cap W^{s}(A_{j+1})\neq
\emptyset \pmod k .
\end{equation*}
If $k=1$, the cycle is called \emph{homoclinic}. In other words, there is a connection whose trajectories tend to $A_1$ in both backward and forward times.

 A \emph{Lyapunov exponent} (LE) associated to a solution of \eqref{sistema geral} is an average exponential rate of divergence or convergence of nearby trajectories in the phase space.
Based on \cite{sandri1996numerical},
to estimate these exponents we consider two nearby points $u_0$ and $u_0+v$ in the phase space, where $v$ is a (small) vector.
Denoting by $\|\cdot\|$ the euclidean norm in $\Rr^n$, the number 
$$
LE(u_0,v)=\lim_{t\rightarrow +\infty}\frac{1}{t}\, \log \,  \| D_{u_0} \phi^t(u_0).v \|,
$$ 
designated as the Lyapunov exponent of $u_0$ in the direction $v$, exists and is finite for almost all points in $\Rr^n$ (\cite{oseledec1968multiplicative}).
For $u_0\in\Rr^n$, if $LE(u_0,v)>0$ for some direction $v$, then one has exponential divergence of nearby orbits.
In this case, we say that there exists an orbit with a \emph{positive Lyapunov exponent}.

Following~\cite{homburg2002periodic}, a  (H\'enon-type) \emph{strange attractor} of a two-dimensional dissipative diffeomorphism $R$ defined in a compact and Riemannian manifold, is a compact invariant set $\Lambda$ with the following properties:
\begin{itemize}
\item $\Lambda$  equals the topological closure of the unstable manifold of a hyperbolic periodic point;
\item the basin of attraction of $\Lambda$   contains an open set ($\Rightarrow$ has positive Lebesgue measure);
\item there is a dense orbit in $\Lambda$ with a positive Lyapunov exponent (i.e. there is exponential growth of the derivative along its orbit). 
\end{itemize}
A vector field possesses a strange attractor if the first return map to a cross section does.

Let $\Psi$ be a property of a dynamical system and $I\subseteq\Rr$ a non degenerate interval.
We say that a one-parameter family $(f_{\boldsymbol{\mu}})_{\boldsymbol{\mu}\in I}$  exhibits \emph{persistently} the property
$\Psi$ if it is observed for $f_{\boldsymbol{\mu}}$ over a set of
parameter values $\boldsymbol{\mu}$ with positive Lebesgue measure.  
The novelty of this article is the following result, where the neighbourhood is defined in the  $C^2$-Whitney topology.

\begin{maintheorem}\label{thrm:main}
For the  family $(f_{\boldsymbol{\mu}})_{\boldsymbol{\mu}\in \Rr}$,
there exists $I\subseteq\Rr$ such that, for all 
neighbourhood $\mathcal{U}\subset \mathcal{X}$ of $f_{\tilde{\boldsymbol{\mu}}}$, $\tilde{\boldsymbol{\mu}}\in I$,  there exists a generic family $\left( g_{ \boldsymbol{\lambda}}\right)_{\boldsymbol{\lambda}\in [-1,1]} \in\mathcal{U}$  exhibiting persistently strange attractors.
\end{maintheorem}
  
In Figure \ref{fig:scheme1} we provide a scheme with the main idea of Theorem \ref{thrm:main} in the space of vector fields under consideration.

\begin{figure}[h]
	\includegraphics[width=6cm]{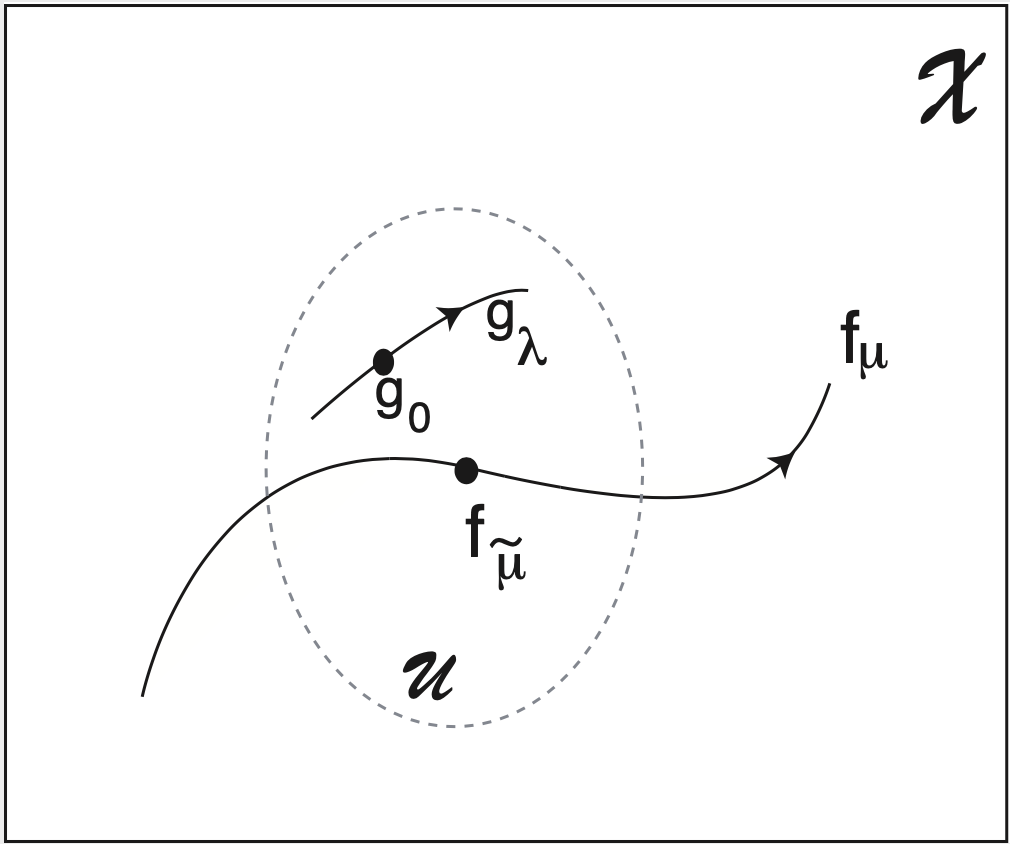}
    \caption{\small{ Illustration of the main idea of Theorem \ref{thrm:main}:   for all 
neighbourhood $\mathcal{U}\subset \mathcal{X}$ of $f_{\tilde{\boldsymbol{\mu}}}$, $\tilde{\boldsymbol{\mu}}\in I$,  there exists a generic family $\left( g_{ \boldsymbol{\lambda}}\right)_{\boldsymbol{\lambda}\in [-1,1]} \in\mathcal{U}$  exhibiting persistently strange attractors.}} \label{fig:scheme1}
\end{figure}

After studying the dynamics of the   differential equation \eqref{eq:poly_rep_8}, we   prove  Theorem \ref{thrm:main} by making use of a homoclinic cycle to the interior equilibrium (of Shilnikov type), whose existence is a criterion for observable and persistent chaotic dynamics.  
The core of the present work goes further: we describe the phenomenological scenario leading to the emergence of a strange attractor.

\section{Bifurcation analysis}
\label{s: bifurcation analysis}

We proceed to the analysis of the one-parameter family of differential equations~\eqref{eq:poly_rep_8}.  We describe the dynamics on the different faces, including the emergence of different equilibria on the edges.
Our analysis will be focused on   $\boldsymbol{\mu}\in \left[-\frac{2938}{95},10\right]$, where a unique equilibrium exists in $[0,1]^3$. This equilibrium will play an important role in the emergence of the strange attractor.

In what follows, we list some assertions that have been analytically found.

\subsection{Boundary}
We describe the equilibria and bifurcations on the boundary of $[0,1]^3$  as function of $\boldsymbol{\mu}$.
The equilibria of system~\eqref{eq:poly_rep_8} do not necessarily belong to the cube.
Throughout this article, the equilibria are those that lie on  $[0,1]^3$ and formal equilibria (as defined in~\cite[Def. 4.1]{alishah2015conservative}) lie outside it.

\begin{figure}[h]
    \centering
\hspace{-10mm}
    \begin{subfigure}[t]{0.3\textwidth}\centering
        \includegraphics[width=5.5cm]{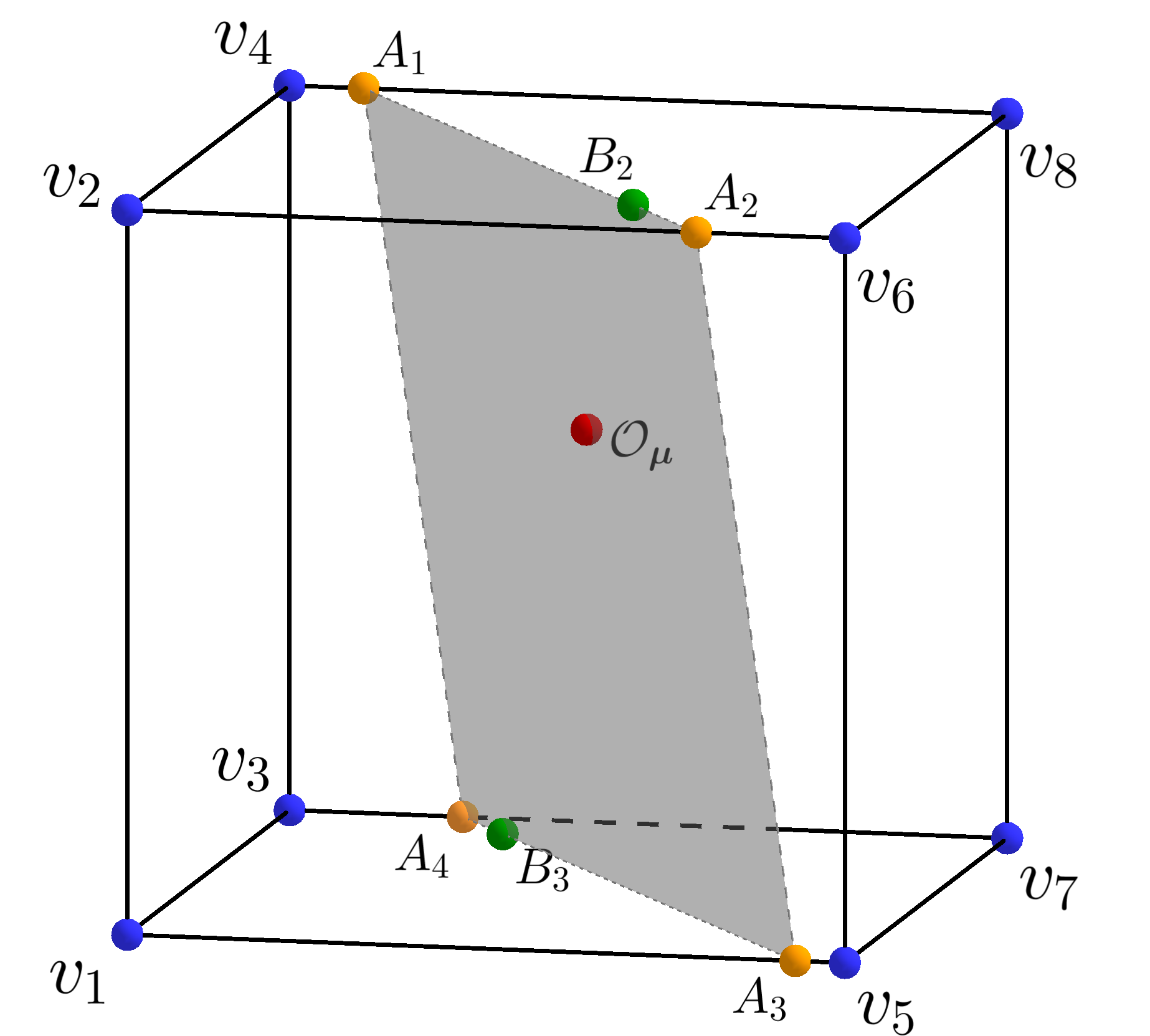}
        \includegraphics[width=5.687cm]{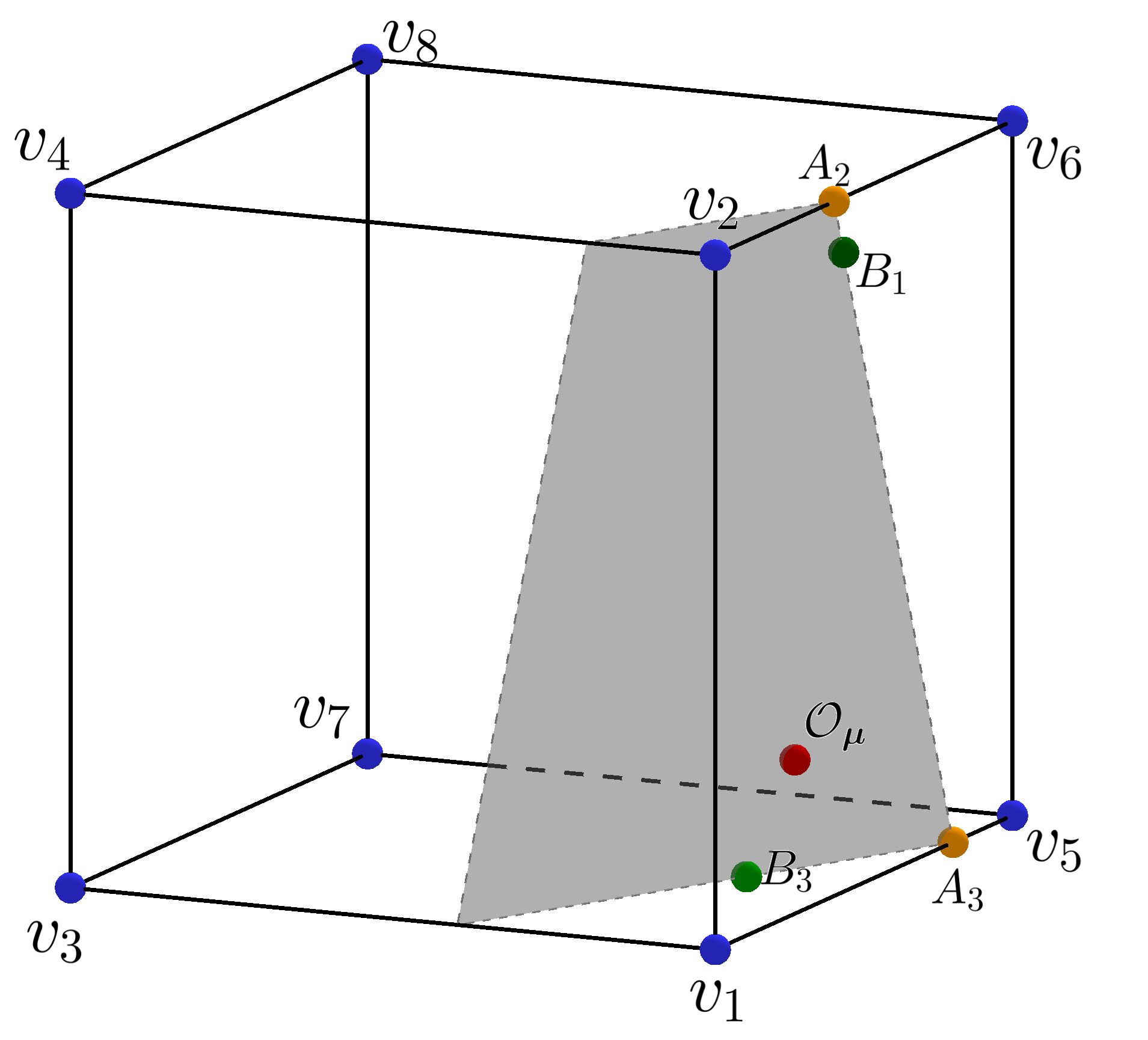}
    \end{subfigure}
    \quad\quad\quad\quad
    \begin{subfigure}[t]{0.3\textwidth}\centering
        \includegraphics[width=5.5cm]{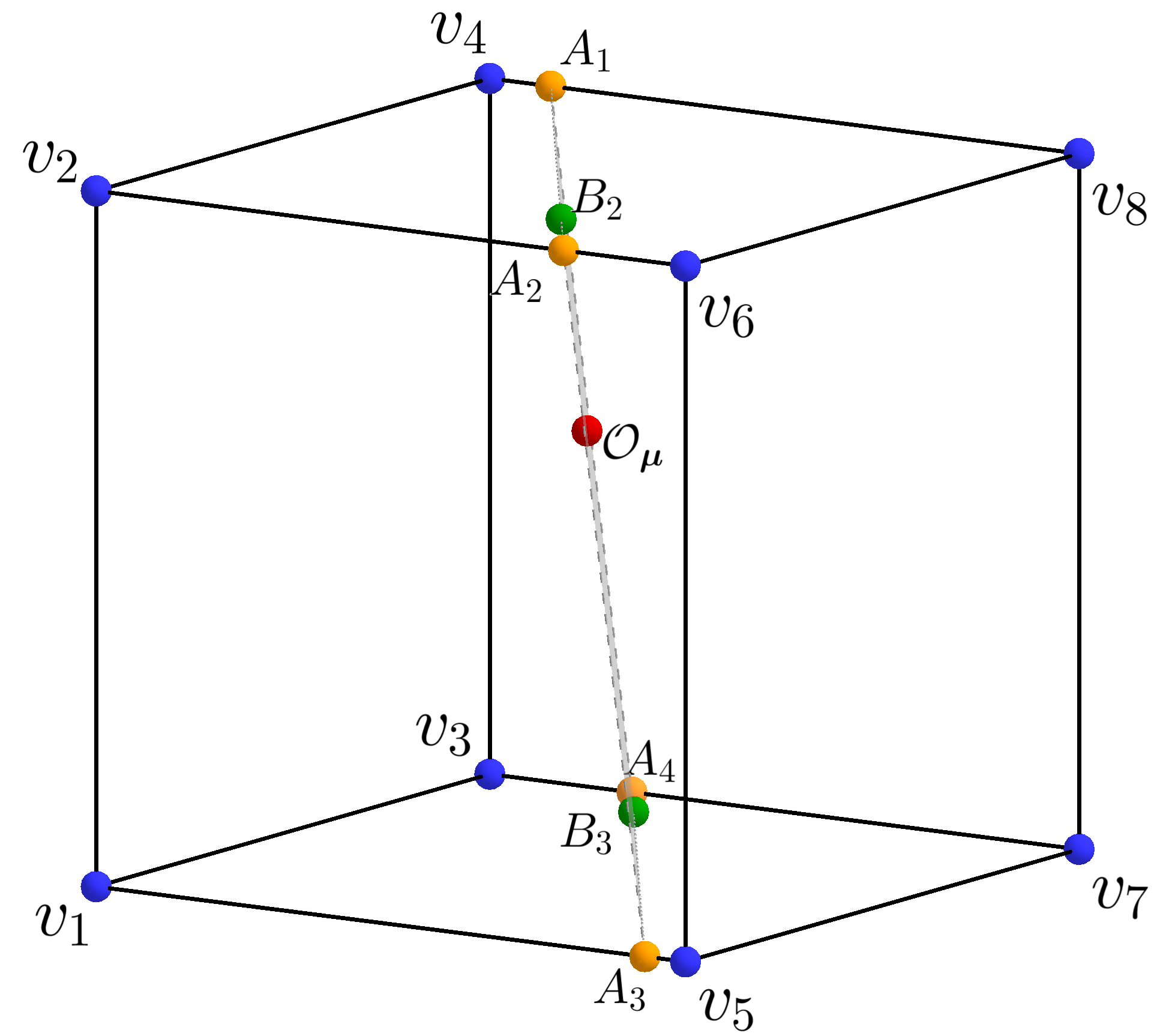}
        \includegraphics[width=5.5cm]{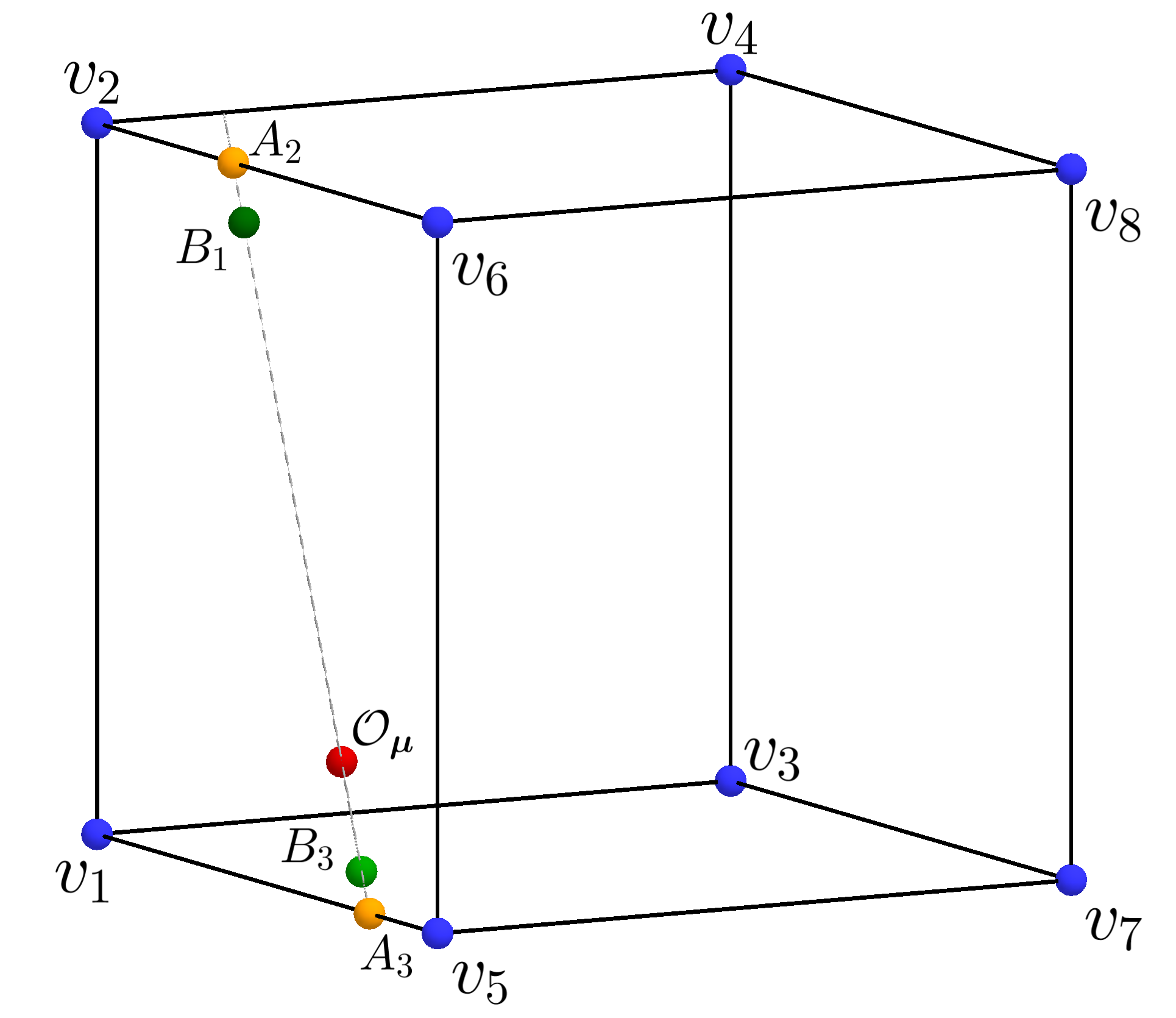}
    \end{subfigure}
    \vspace{-.3cm}
    \caption{\small{Different perspectives of the phase space and the corresponding equilibria of  \eqref{eq:poly_rep_8}: the eight vertices $v_1, \dots , v_8$ (in blue), four equilibria on edges, $A_1, A_2, A_3 , A_4$ (in orange), three equilibria on faces, $B_1, B_2, B_3$ (in green), and the interior equilibrium $\mathcal{O}_{\boldsymbol{\mu}}$ (in red), for $\boldsymbol{\mu}=-15$ (top) and $\boldsymbol{\mu}=4$ (bottom). These equilibria lie on a plane  (se the remark after Lemma~\ref{lem:fact_a}).}}\label{fig:equilibria_on_cube}
\end{figure}

From now on, all the figures with numerical plots of the flow of \eqref{eq:poly_rep_8} on  $[0,1]^3$ are in the same position of Figure~\ref{fig:equilibria_on_cube} (up left) where $v_1=(0,0,0)$ is the vertex   located in the lower left front corner.

\begin{figure}[h]
    \includegraphics[width=14cm]{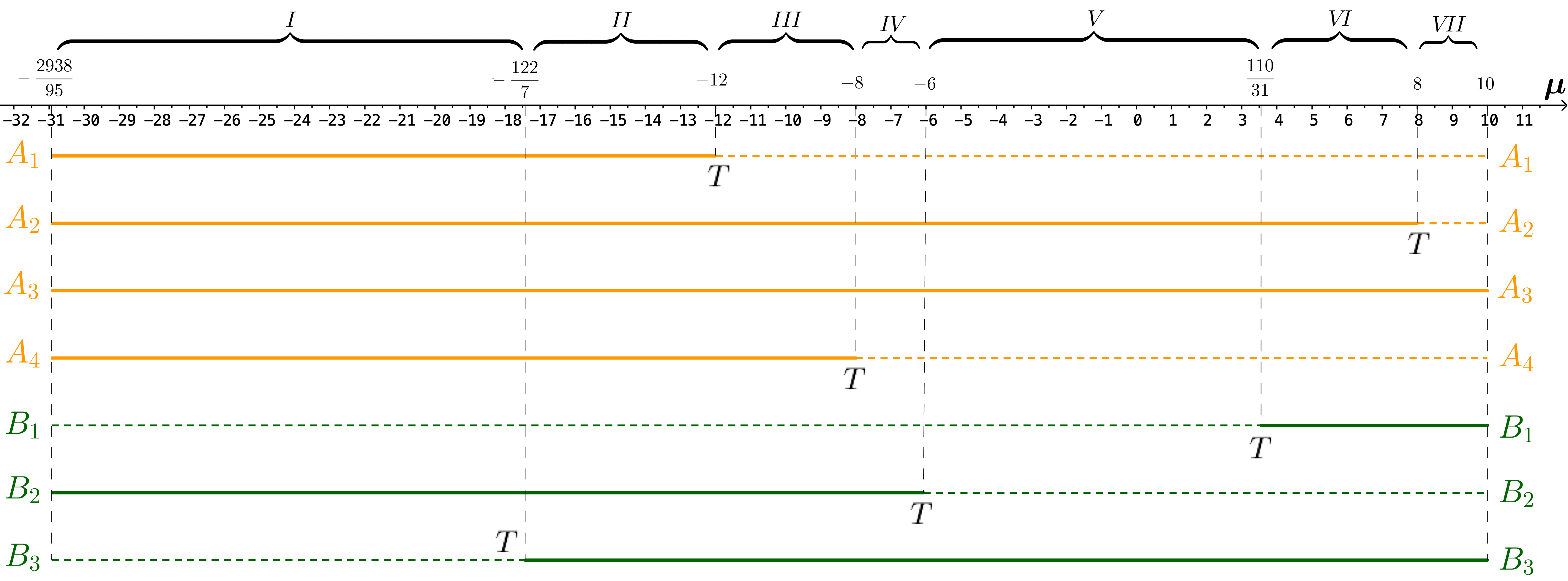}
    \caption{\small{The values of $\boldsymbol{\mu}$ for which the equilibria of system~\eqref{eq:poly_rep_8} exist on $[0,1]^3$ (continuous line) and outside $[0,1]^3$ (dashed line) -- formal equilibria.
The Roman numerals $\RN{1},\dots,\RN{7}$ represents the cases of Table \ref{tbl:cases} and the letter $T$ corresponds to the values of $\boldsymbol{\mu}$ for which a transcritical bifurcation occur.
}}\label{fig:equilibria_on_scale}
\end{figure}

We describe a list of equilibria that appear on the edges and faces of the cube $[0,1]^3$, as function of the parameter $\boldsymbol{\mu}$.
The cube has six faces defined by, for  $i\in\{1, 3, 5 \}$,
\begin{align*}
\sigma_i &: \qquad x_{i+1} = 1\\
\sigma_{i+1} &: \qquad x_{i+1} =0.
\end{align*}
In Table~\ref{tbl:Representation_of_vs} we identify the vertices that belong to each face.
As suggested in Figure \ref{fig:equilibria_on_cube}, we set the notation $A_j$, $j=1, 2, 3, 4$ for equilibria on the edges and $B_j$, $j=1, 2, 3$ for equilibria on the interior of the faces.
Formally, the $A$'s and $B$'s equilibria depend on $\boldsymbol{\mu}$ but, once again, we omit their dependence on the parameter.

\begin{lem}\label{Lemma4}
With respect to system~\eqref{eq:poly_rep_8},  for $\boldsymbol{\mu} \in [-\frac{2938}{95},10]$, the following assertions hold:
\begin{enumerate}
	\vspace{.3cm}\item The eight vertices and $A_3=\left( \frac{12-\boldsymbol{\mu}}{14-\boldsymbol{\mu}},0,0 \right)$ exist  for $\boldsymbol{\mu} \in \left[-\frac{2938}{95},10 \right]$,

	\vspace{.3cm}\item $A_1=\left( \frac{\boldsymbol{\mu}+12}{\boldsymbol{\mu}-14},1,1 \right) \,\textrm{exists in the cube}\quad \Leftrightarrow\quad \boldsymbol{\mu} \in \left[ -\frac{2938}{95},-12 \right[$,

	\vspace{.3cm}\item $A_4=\left( \frac{8+\boldsymbol{\mu}}{\boldsymbol{\mu}-14},1,0 \right) \,\textrm{exists in the cube}\quad \Leftrightarrow\quad \boldsymbol{\mu} \in \left[ -\frac{2938}{95},-8 \right[$,

	\vspace{.3cm}\item $A_2=\left( \frac{8-\boldsymbol{\mu}}{14-\boldsymbol{\mu}},0,1 \right) \,\textrm{exists in the cube}\quad \Leftrightarrow\quad \boldsymbol{\mu} \in \left[ -\frac{2938}{95},8 \right[$,

	\vspace{.3cm}\item $B_1=\left( \frac{15+\boldsymbol{\mu}}{40+\boldsymbol{\mu}},0,\frac{27(10-\boldsymbol{\mu})}{4(40+\boldsymbol{\mu})} \right) \,\textrm{exists in the cube}\quad \Leftrightarrow\quad \boldsymbol{\mu} \in \left] \frac{110}{31}, 10 \right]$,
	
	\vspace{.3cm}\item $B_2=\left( \frac{62+\boldsymbol{\mu}}{86+\boldsymbol{\mu}},-\frac{3(6+\boldsymbol{\mu})}{2(86+\boldsymbol{\mu})},1 \right) \,\textrm{exists in the cube}\quad \Leftrightarrow\quad \boldsymbol{\mu} \in \left[ -\frac{2938}{95},-6 \right[$,
	
	\vspace{.3cm}\item $B_3=\left( \frac{38+\boldsymbol{\mu}}{86+\boldsymbol{\mu}},\frac{5(10-\boldsymbol{\mu})}{2(86+\boldsymbol{\mu})},0 \right) \,\textrm{exists in the cube}\quad \Leftrightarrow\quad \boldsymbol{\mu} \in \left] -\frac{122}{7}, 10 \right]$.
\end{enumerate}
\end{lem}

The proof of Lemma \ref{Lemma4} is   elementary by computing zeros of $f_{\boldsymbol{\mu}}$ and taking into account that they should live in $[0,1]^3$. The evolution (as function of $\boldsymbol{\mu}$) of the equilibria $A_1, A_2, A_3, A_4$ on the edges and $B_1, B_2, B_3$ on the faces is depicted in the scheme of Figure \ref{fig:equilibria_on_scale}.
The eigenvalues and eigendirections are summarised in Tables \ref{tbl:Eigenv_of_As} and \ref{tbl:Eigenv_of_Bs}.
Using the sign of the eigenvalues, as well as their evolution, we are able to locate \emph{transcritical bifurcations},
which are summarised in the following paragraph and pointed out in Figure \ref{fig:equilibria_on_scale}.

We will consider sub-intervals of $[-\frac{2938}{95},10]$ based on the values of $\boldsymbol{\mu}$ for which the bifurcations occur.
Namelly:
\begin{itemize}
	\item at $\boldsymbol{\mu}=-12$ the vertex $v_4$ undergoes a transcritical bifurcation (see the zero eigenvalue in Table~\ref{tbl:Eigenv_of_vs}) responsible for the transition of $A_1$  from $[0,1]^3$ to outside, becoming a formal equilibrium; the analysis of the bifurcation associated to $A_2$ and $A_4$ is similar at $\boldsymbol{\mu}=8$ and $\boldsymbol{\mu}=-8$, respectively;
	\item at $\boldsymbol{\mu}=-6$, the equilibrium $A_2$ undergoes a transcritical bifurcation and  $B_2$ evolves from an equilibrium (inside the cube) to a formal one (outside the cube); the reverse happens to $B_1$ at $\boldsymbol{\mu}=\frac{110}{31}$;
	\item at $\boldsymbol{\mu}=-\frac{122}{7}$, the equilibrium $A_4$ undergoes a transcritical bifurcation and $B_3$ evolves from a formal equilibrium (outside the cube) to an equilibrium (inside the cube);
	\item at $\boldsymbol{\mu}=b_2\approx -21.9$ (see Table~\ref{tbl:Eigenv_of_Bs}) and $\boldsymbol{\mu}=-12$, the equilibrium $B_2$ undergoes a Belyakov transition; and, at $\boldsymbol{\mu}=b_3\approx -14.22$ the  $B_3$ also undergoes the same bifurcation (see Table~\ref{tbl:Eigenv_of_Bs}).
\end{itemize}


\subsection{Interior}
In this section, we focus our attention on the interior equilibrium and its relation to others on the   boundary.

\begin{lem}
\label{lem:fact_a}
For $\boldsymbol{\mu} \in\,]-\frac{2938}{95},10[$, system~\eqref{eq:poly_rep_8} has a unique interior equilibrium,
whose expression is
\begin{equation}\label{int_equil}\nonumber
\mathcal{O}_{\boldsymbol{\mu}}:=\left( \frac{7 \boldsymbol{\mu} -1042}{7 \boldsymbol{\mu} -2014}, \frac{37 (\boldsymbol{\mu} -10)}{7 \boldsymbol{\mu} -2014},
\frac{109 (\boldsymbol{\mu} -10)}{2 (7 \boldsymbol{\mu} -2014)} \right).
\end{equation}
\end{lem}

\begin{proof}
The proof is immediate by computing the non-trivial zeros of the vector field   \eqref{eq:poly_rep_8}:
\begin{equation}\nonumber
\left \{ \begin{array}{l}
12- \boldsymbol{\mu}+(\boldsymbol{\mu}-14) x-20 y - 4 z =0 \\[2mm]
 -10 +20 x + 4 y - 4 z =0 \\[2mm]
27 - 54 x + 11 y - 4 z =0
\end{array} \right. .
\end{equation}

\end{proof}

Taking into account that the equilibria $B_1$, $B_2$ and $B_3$ depend on $\boldsymbol{\mu}$, it is worth to notice that 
$$ \lim_{\boldsymbol{\mu} \rightarrow -\frac{2938}{95}} \mathcal{O}_{\boldsymbol{\mu}}= \lim_{\boldsymbol{\mu} \rightarrow -\frac{2938}{95}} B_2 = \left( \frac{123}{218},\frac{74}{109},1\right) $$ and 
$$\lim_{\boldsymbol{\mu} \rightarrow 10}\mathcal{O}_{\boldsymbol{\mu}}=\lim_{\boldsymbol{\mu} \to 10} B_1=\lim_{\boldsymbol{\mu} \to 10} B_3 = \left(\frac{1}{2}, 0,0\right), 
$$
which means that when $\boldsymbol{\mu} \in\,]-\frac{2938}{95},10]$, the point $O_ {\boldsymbol{\mu}}$ travels from the face $\sigma_5$ to the edge which connects $v_1$ to $v_5$, the intersection of the faces $\sigma_4$ and $\sigma_6$.

\begin{rem}
The points $A_1$, $A_2$, $A_3$, $A_4$, $B_1$, $B_2$,   $B_3$ and $\mathcal{O}_{\boldsymbol{\mu}}$ belong to the plane defined by
$
 (14-\boldsymbol{\mu}) x+20y+ {4} z= {12-\boldsymbol{\mu}}.
$
This follows immediately  from the first equation of system \eqref{eq:poly_rep_8}.
This plane is not flow-invariant.
\end{rem}

\begin{lem}\label{corollary1}
With respect to $Df_{\boldsymbol{\mu}} \left(\mathcal{O}_{\boldsymbol{\mu}} \right)$,  there exist $\mu_1, \mu_2, \mu_4  \in [-\frac{2938}{95},10]$ such that  $\mu_1<\mu_2<\mu_4$ and\footnote{A parameter value $\mu_3\in\,\,]\mu_2, \mu_4[$ appears later in Subsection \ref{ss:LE1}.}:
\begin{enumerate}
\item for $\boldsymbol{\mu} =\mu_1$, the equilibrium $\mathcal{O}_{\boldsymbol{\mu}}$ undergoes a Belyakov transition;
\item  for $\boldsymbol{\mu}= \mu_2 $ and  $\boldsymbol{\mu}= \mu_4 $ the equilibrium $\mathcal{O}_{\boldsymbol{\mu}}$ undergoes a supercritical Hopf bifurcation.
\end{enumerate}
\end{lem}

\begin{proof}
The characteristic polynomial of $Df_{\boldsymbol{\mu}} \left(\mathcal{O}_{\boldsymbol{\mu}} \right)$ has three roots, which depend on $\boldsymbol{\mu}$.
Although these three functions have an intractable analytical expression,
it is possible to show the existence of 
$\mu_1, \mu_2, \mu_4 \in [-\frac{2938}{95},10]$ such that  $\mu_1< \mu_2<\mu_4$  and the following assertions hold (cf. Figures~\ref{fig:Egvls_int_eq_func_mu_1} and~\ref{fig:Egvls_int_eq_func_mu_2}):
\begin{enumerate}
\item for $\boldsymbol{\mu} \in [-\frac{2938}{95}, \mu_1]$, the three eigenvalues are real and negative;
\item for $\boldsymbol{\mu} \in \,\,]\mu_1, \mu_2[ \,\cup\, ]\mu_4, 10]$, there are two complex conjugate eigenvalues and one real, all of them with negative real part;
\item  for $\boldsymbol{\mu} \in \,\,]\mu_2,  {\mu}_4[$, there are two complex conjugate eigenvalues with positive real part  and one real negative.
\end{enumerate}

As suggested by Figure \ref{fig:Egvls_int_eq_func_mu_2}, the complex (non-real) eigenvalues cross the imaginary axis with strictly positive speed as $ \boldsymbol{\mu}$ passes through $  {\mu}_2$ and  $  {\mu}_4$, confirming that:
$$
\frac{ Re \left( D f_{\boldsymbol{\mu}} \left( \mathcal{O}_{\boldsymbol{\mu}} \right) \right)}{d\, \boldsymbol{\mu} }|_{\boldsymbol{\mu} = {\mu} _2} \neq 0 \neq \frac{ Re \left( D f_{\boldsymbol{\mu}} \left( \mathcal{O}_{\boldsymbol{\mu}} \right) \right) }{d\, \boldsymbol{\mu} }|_{\boldsymbol{\mu} ={\mu} _4}.
$$
  This means that at $\boldsymbol{\mu} =\mu_2$ and $\boldsymbol{\mu}=\mu_4$, the equilibrium $\mathcal{O}_{\boldsymbol{\mu}}$ undergoes a supercritical Hopf bifurcation. When  $\boldsymbol{\mu} =\mu_2$, it gives rise to a non-trivial attracting periodic solution, say $\mathcal{C}_{\boldsymbol{\mu}}$, which collapses again into $\mathcal{O}_{\boldsymbol{\mu}}$ for $\boldsymbol{\mu} =\mu_4$. 

\end{proof}

\begin{figure}[h]
	\includegraphics[width=13cm]{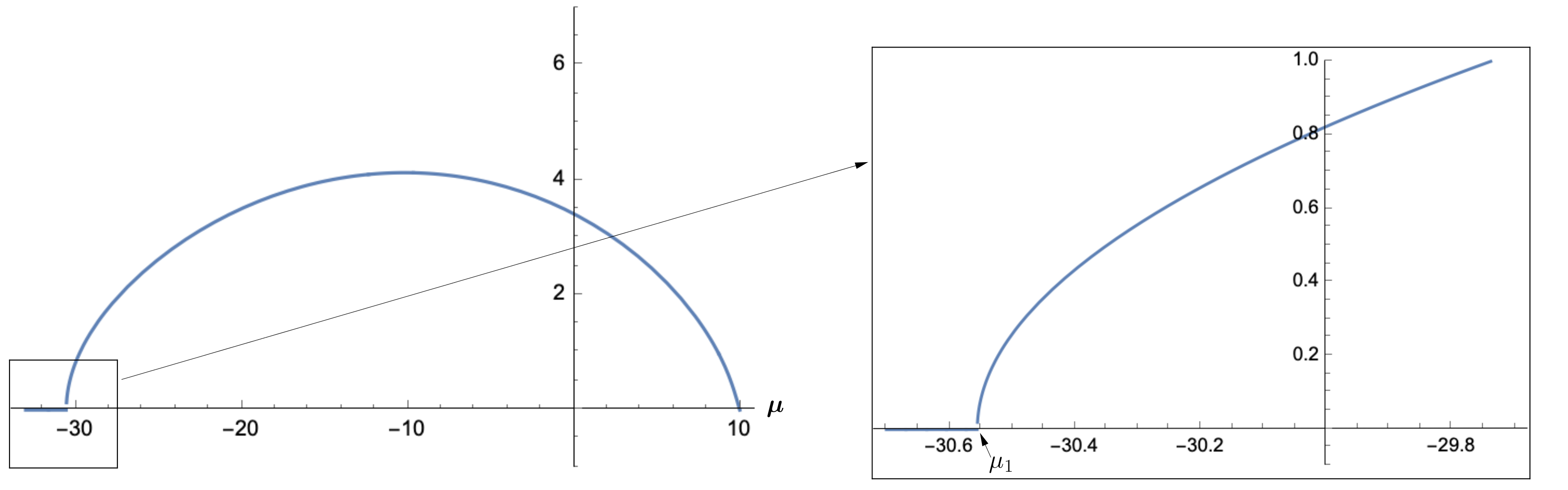}
    \caption{\small{\emph{Belyakov transition}: graph  of the imaginary part of the complex eigenvalues of $Df_{\boldsymbol{\mu}} \left(\mathcal{O}_{\boldsymbol{\mu}} \right)$ for $\boldsymbol{\mu} \in \left[ -\frac{2938}{95}, 10 \right]$ (left) and its zoom around $\mu_1$, $\boldsymbol{\mu} \in \left[ -30.7, -29.7 \right]$ (right), for system~\eqref{eq:poly_rep_8}.}} \label{fig:Egvls_int_eq_func_mu_1}
\end{figure}

\begin{figure}[h]
	\includegraphics[width=13cm]{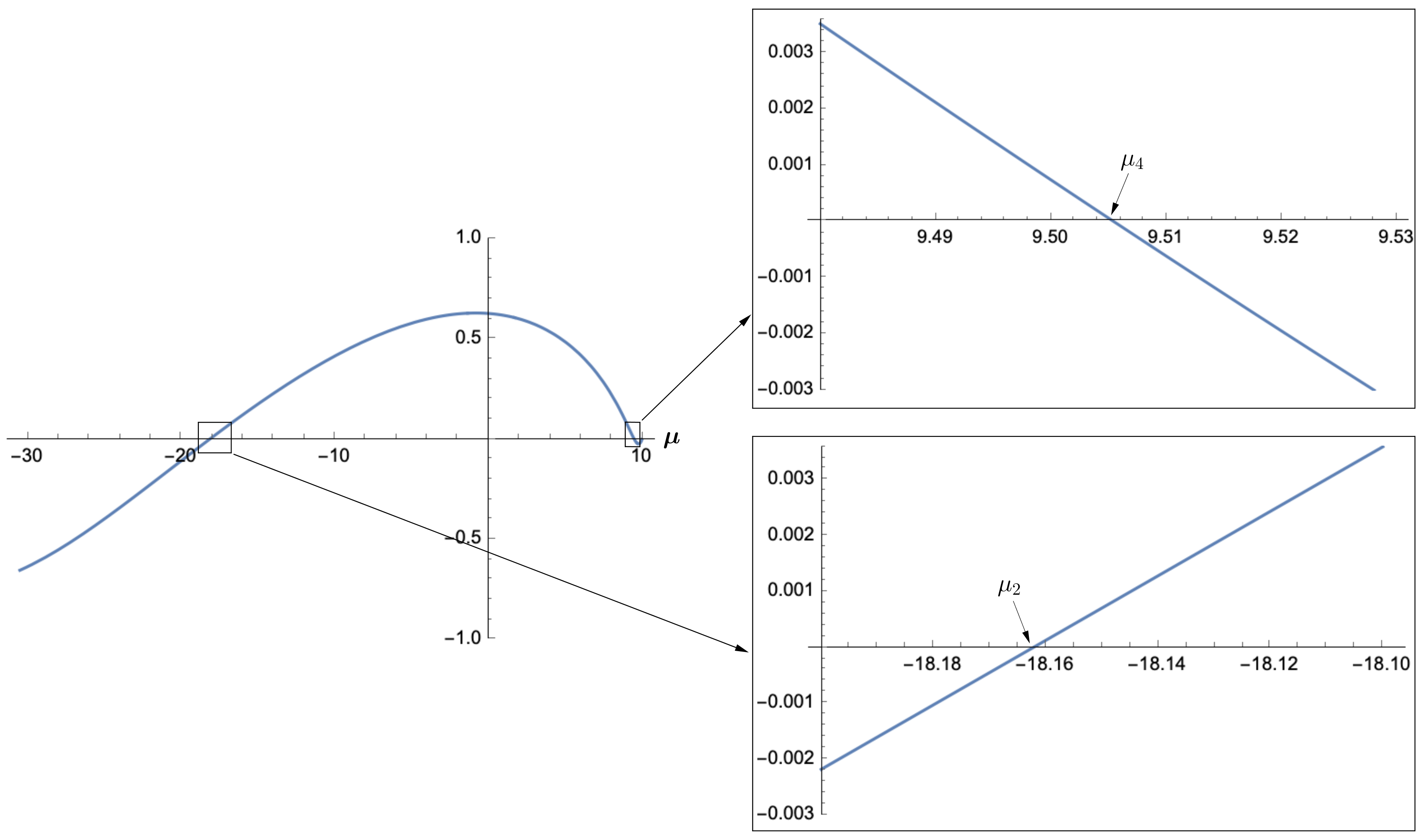}
    \caption{\small{\emph{Hopf bifurcation}: graph of  the real part of the eigenvalues of $Df_{\boldsymbol{\mu}} \left(\mathcal{O}_{\boldsymbol{\mu}} \right)$ for $\boldsymbol{\mu} \in \left[ -\frac{2938}{95}, 10 \right]$ (left) and its zoom around $\mu_4$, $\boldsymbol{\mu} \in \left[ 9.48, 9.53 \right]$ (right up) and around $\mu_2$, $\boldsymbol{\mu} \in \left[ -18.2, -18.1 \right]$ (right down), for system~\eqref{eq:poly_rep_8}.}} \label{fig:Egvls_int_eq_func_mu_2}
\end{figure}

From now on, we set\footnote{These values  coincide with those obtained via \textit{MatCont} \cite{Matcont2008} for $\textrm{MATLAB}^{\circledR}$.}:
 \begin{eqnarray*}
\mu_1 &\mapsto &  \boldsymbol{\mu}_{\textrm{Belyakov}}\approx-30.5550;  \\ 
 \mu_2 &\mapsto & \boldsymbol{\mu}^1_{\textrm{Hopf}}\approx-18.1623;  \\ 
  \mu_4 &\mapsto & \boldsymbol{\mu}^2_{\textrm{Hopf}}\approx 9.5055.
\end{eqnarray*}

For $\boldsymbol{\mu} \in [-\frac{2938}{95}, \mu_2]$,  the equilibrium $\mathcal{O}_{\boldsymbol{\mu}}$ is globally attracting as depicted in Figure~\ref{fig:04_mu=-20_v1_v2}. 
In the context of Game Theory,  it is called a \emph{global attracting mixed Nash equilibrium}, in the sense that it is associated to non-pure strategies.

\begin{figure}[h]
	\includegraphics[width=6.5cm]{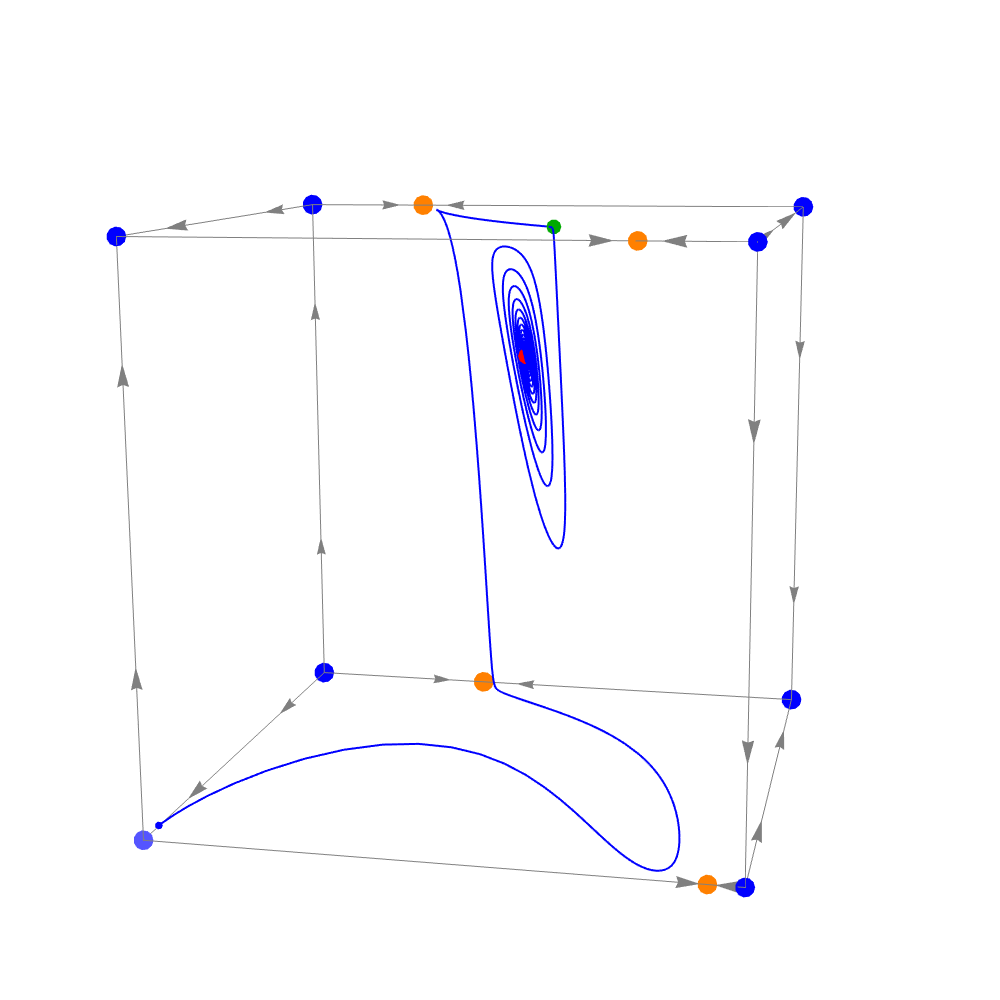}
\hspace{-1.2cm}
	\includegraphics[width=6.5cm]{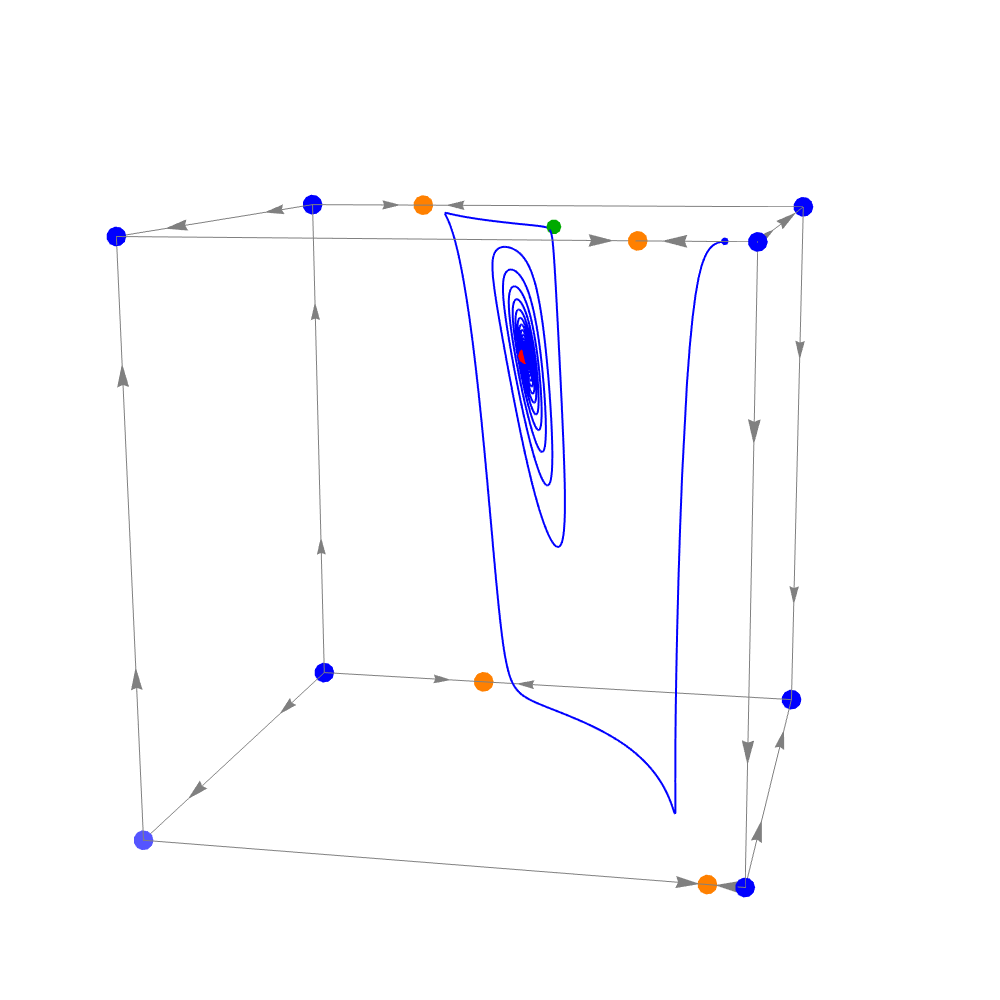}
	\vspace{-.5cm}
    \caption{\small{\emph{Global attractiveness of $\mathcal{O}_{\boldsymbol{\mu}}$ (cf. Proposition~\ref{prop:central_manfld_I} $(1)$):} plot of two orbits (in blue) with initial condition near  $v_1$ (left), initial condition near $v_6$ (right), the interior equilibrium (in red) and the boundary equilibria of system~\eqref{eq:poly_rep_8} for $\boldsymbol{\mu}=-20$  and $t\in [0,50]$. }}\label{fig:04_mu=-20_v1_v2}
\end{figure}

\section{Numerical Analysis}\label{s: numerical analysis}
Using $\textit{Mathematica Wolfram}^{\circledR}$, we present checkable numerical evidences about the vector field $f_{\boldsymbol{\mu}}$, for $\boldsymbol{\mu} \in [-\frac{2938}{95}, 10]$, that will constitute the foundations for the persistence of strange attractors. At the end of this section we discuss the validity of these  numerical results.

\subsection{Lyapunov exponents}
\label{ss:LE1}
Using the method in Sandri ~\cite{sandri1996numerical}, in Figure~\ref{fig:LE} we computed numerically the LE of system~\eqref{eq:poly_rep_8} for an initial condition of the form:
\begin{equation}\nonumber
(x_0, y_0, z_0)=\mathcal{O}_{\boldsymbol{\mu}}+\left \{ \begin{array}{ll}
(\varepsilon,0,-\varepsilon), &  \textrm{if } \boldsymbol{\mu}\in [-\frac{2938}{95},-30[  \\[2mm]
(0,0,\varepsilon), &  \textrm{if } \boldsymbol{\mu}\in [-30,9]  \\[2mm]
(\varepsilon,\varepsilon,0), &  \textrm{if } \boldsymbol{\mu}\in \,\,]9,10]
\end{array} \right.,
\end{equation}
with $\varepsilon=0.001$, thus ensuring that $(x_0, y_0, z_0)\in \inter\left([0,1]^3\right) \backslash W^s(\mathcal{O}_{\boldsymbol{\mu}})$. As suggested in \cite{aguirre2014global}, since  $(x_0, y_0, z_0) \notin W^s(\mathcal{O}_{\boldsymbol{\mu}})$, its trajectory
is strongly governed by the invariant manifold $W^u(\mathcal{O}_{\boldsymbol{\mu}})$,  which plays an essential role in the construction of the H\'enon-type strange attractor of Theorem  \ref{thrm:main}. From a close analysis of Figure~\ref{fig:LE},  we deduce that:
\begin{enumerate}
\item for $\boldsymbol{\mu}<\boldsymbol{\mu}^1_{\textrm{Hopf}}$ and $\boldsymbol{\mu}>\boldsymbol{\mu}^2_{\textrm{Hopf}}$, the three LE are negative;
\item there exists $\mu_3 \in\, \,  ]\boldsymbol{\mu}^1_{\textrm{Hopf}},\boldsymbol{\mu}^2_{\textrm{Hopf}}[$ such that:
\begin{enumerate}
\item  for $\boldsymbol{\mu}\in\, ]\boldsymbol{\mu}^1_{\textrm{Hopf}},\mu_3[$, there are two negative LE and one zero;
\item there are non-trivial subintervals of  $  ]\mu_3, \boldsymbol{\mu}^2_{\textrm{Hopf}} [$ where  there is  one positive LE.
\end{enumerate}
\end{enumerate}

From this analysis, according to~\cite{sandri1996numerical, wolf1985determining}, we  infer that the  attracting set of  system~\eqref{eq:poly_rep_8}, when restricted to the cube's  interior, contains:
\begin{enumerate}
\item   a single  equilibrium, for $\boldsymbol{\mu}<\boldsymbol{\mu}^1_{\textrm{Hopf}}$ (Figure~\ref{fig:04_mu=-20_v1_v2}) and $\boldsymbol{\mu}>\boldsymbol{\mu}^2_{\textrm{Hopf}}$;
\item  a non-trivial periodic solution, for $\boldsymbol{\mu}\in \, ]\boldsymbol{\mu}^1_{\textrm{Hopf}},\mu_3[$ (Figure~\ref{fig:01_limit_cycle});
\item  a strange attractor for some intervals of  $  ]\mu_3, \boldsymbol{\mu}^2_{\textrm{Hopf}} [$ (it follows from $(2)\textrm{(b)}$ and the classification of \cite{wolf1985determining}).
\end{enumerate}

Since $\mu_3$ is the threshold above which we find numerically strange attractors (Figure~\ref{fig:LE}),  we set
\begin{align*}
\mu_3 & \mapsto \boldsymbol{\mu}_{\textrm{SA}}\approx  {1.4645}.
\end{align*}

 As referred in Section \ref{terminology}, a LE is a limit over the variable $t$ and numerical computations  require its truncation.
Since for $\boldsymbol{\mu}> \boldsymbol{\mu}^1_{\textrm{Hopf}}$ there is at least one LE oscillating around the zero value, we decided to consider \emph{positive LE} those greater than $5\times 10^{-3}$.
This will allow to discard uncertain positive Lyapunov exponents due to numerical precision issues. 
Lower precision would complicate the simulations bringing no better results.

\vspace{.3cm}

\begin{figure}[h]
	\includegraphics[width=13cm]{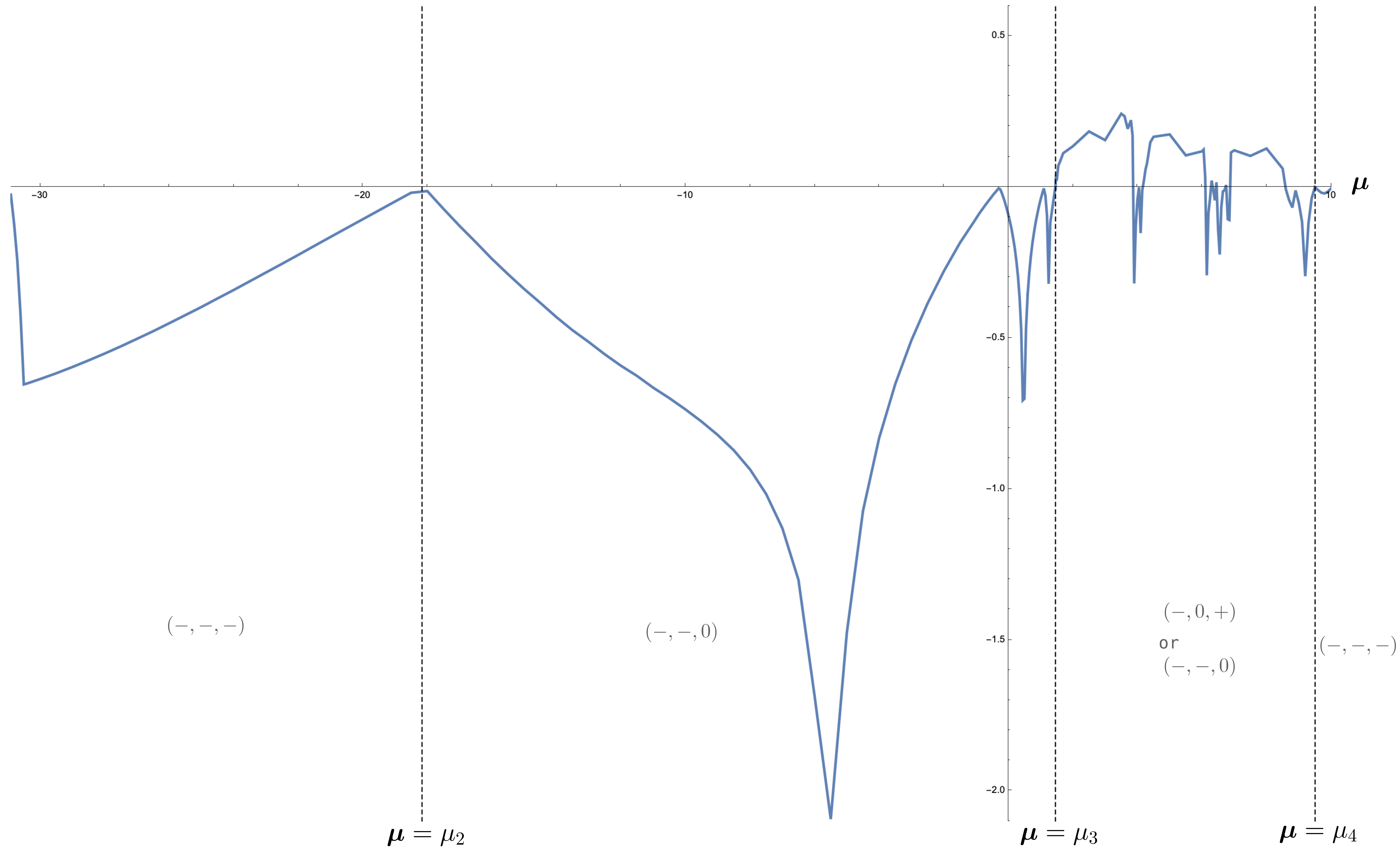}
    \caption{\small{\emph{Sign of the Lyapunov exponents:}  linear interpolation of the largest non-zero Lyapunov exponents of system~\eqref{eq:poly_rep_8} with initial condition near the interior equilibrium ($\notin W^s(\mathcal{O}_{\boldsymbol{\mu}})$).
$(-,-,-)$: all LE are negative. $(-,-,0)$: one LE is zero and the other are negative. $(-,0,+)$: one LE is negative, one is positive and the other is zero.}} \label{fig:LE}
\end{figure}

\subsection{Numerical facts}
\label{ss: numerical facts}
In this subsection, based on numerics, we list some evidences (hereafter called by \emph{Facts}) about system~\eqref{eq:poly_rep_8}. They are essential to characterise the \emph{route to chaos} in Section  \ref{s:route} and  
 may be numerically checked.

\begin{fact}\label{fact_e}
In the parameter interval $\boldsymbol{\mu} \in \,\,]-\frac{2938}{95}, 10[$: 
\begin{enumerate}
\item for  $\boldsymbol{\mu} < -8$ there exist two heteroclinic connections from $v_3$ and $v_6$ to $\mathcal{O}_{\boldsymbol{\mu}}$ (Figure~\ref{fig:heteroclinic_connections_v3_v6_O} (left));
\item for $\boldsymbol{\mu} \in\,]-8,  \boldsymbol{\mu}_{\textrm{SA}}[ \,\cup\, ]\boldsymbol{\mu}^2_{\textrm{Hopf}}, 10[$, there are two heteroclinic connections from the source $v_6$ to $\mathcal{O}_{\boldsymbol{\mu}}$, along the two branches of $W^s ( \mathcal{O}_{\boldsymbol{\mu}})$  (Figure~\ref{fig:heteroclinic_connections_v3_v6_O} (right));
\item for $\boldsymbol{\mu} \in\,] \boldsymbol{\mu}_{\textrm{SA}},  \boldsymbol{\mu}^2_{\textrm{Hopf}}[$, one branch of $W^s ( \mathcal{O}_{\boldsymbol{\mu}})$  winds around the non-wandering set associated to $W^u(\mathcal{O}_{\boldsymbol{\mu}})$.

 \end{enumerate}
\end{fact}

\begin{figure}[h]
	\includegraphics[width=6cm]{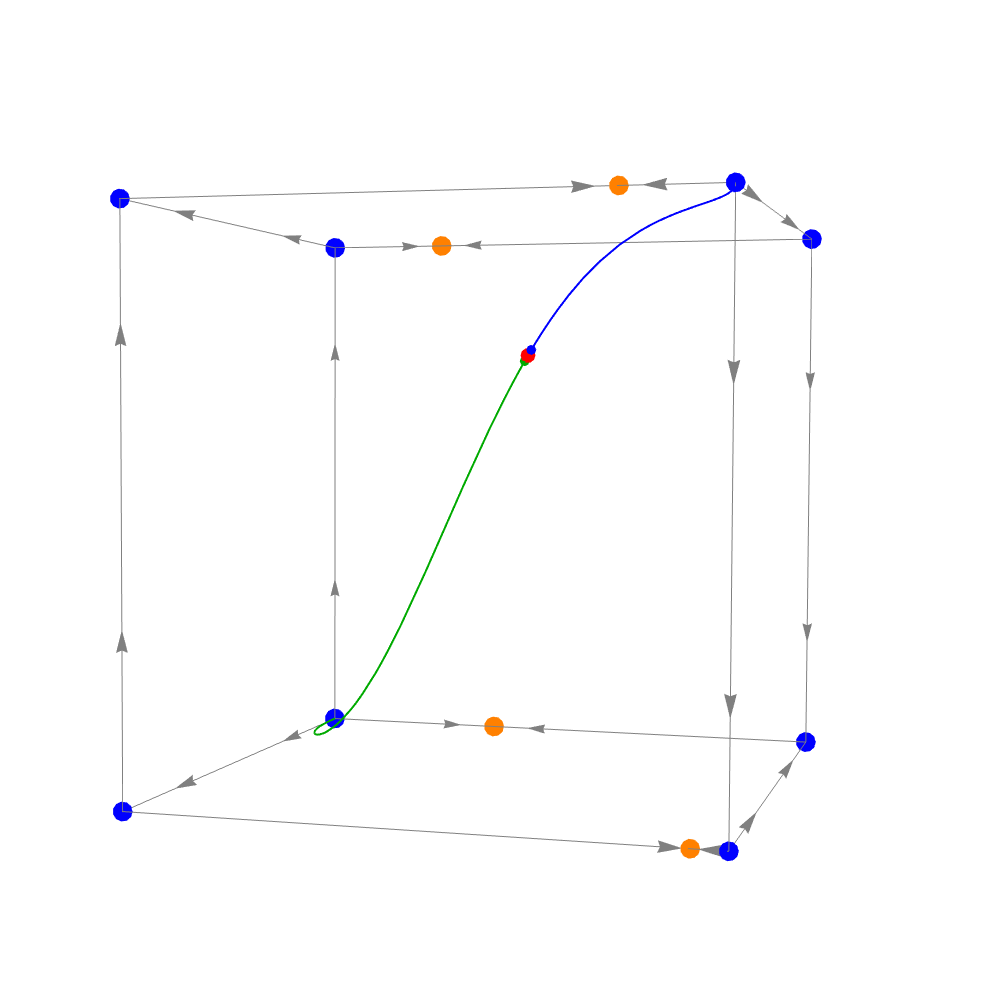}
\hspace{-1.2cm}
	\includegraphics[width=6.1cm]{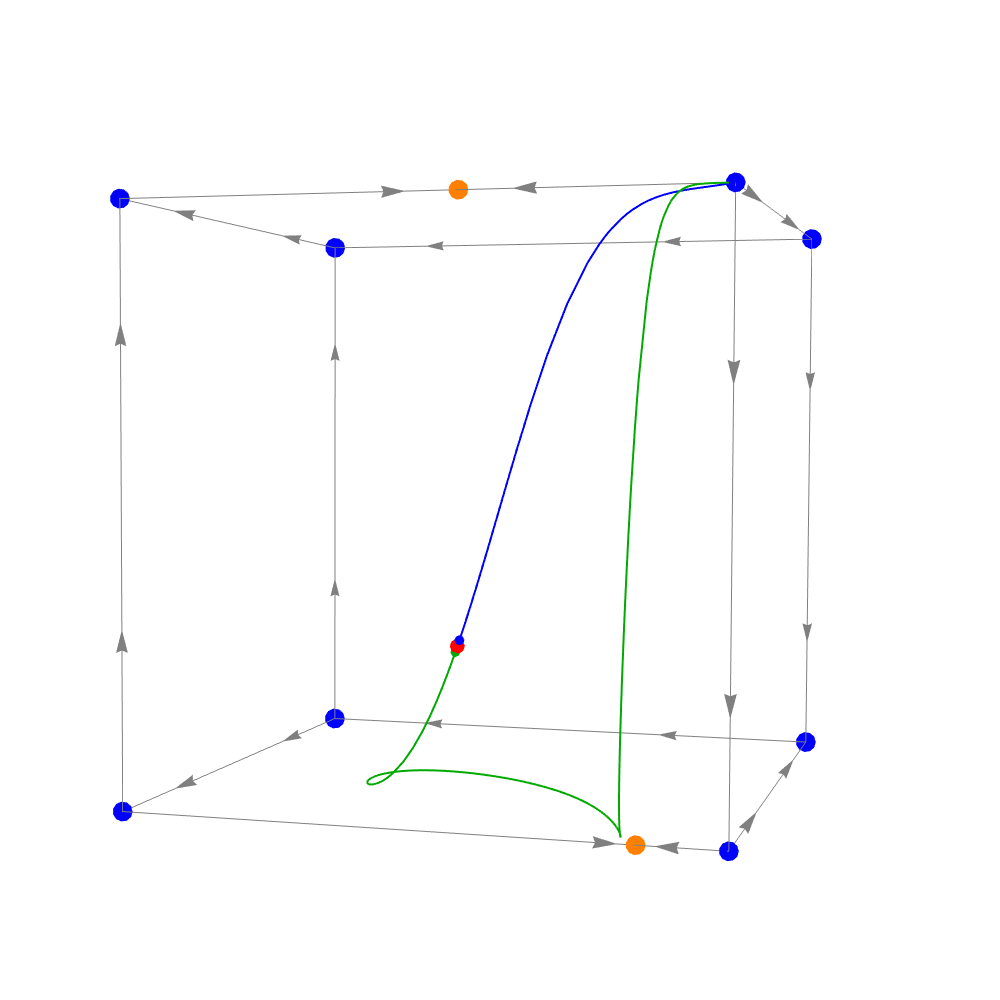}
	\vspace{-.5cm}
    \caption{\small{\textit{Heteroclinic connections (cf. Fact~\ref{fact_e}):} representation of the connections  from the sources  $v_3$ (in green) and $v_6$ (in blue) to $\mathcal{O}_{\boldsymbol{\mu}}$, for system~\eqref{eq:poly_rep_8} with $\boldsymbol{\mu}=-20$ (left) and $\boldsymbol{\mu}=0$ (right). At $\boldsymbol{\mu}=-8$ the equilibrium $A_4$ collapses with $v_3$, which makes the lower branch of $W^u(\mathcal{O}_{\boldsymbol{\mu}})$ to connect $v_6$, a phenomenon that persists for $\boldsymbol{\mu}>-8$.}} \label{fig:heteroclinic_connections_v3_v6_O}
\end{figure}

\begin{fact}\label{fact_g}
For  $\boldsymbol{\mu} \in \left] \boldsymbol{\mu}^1_{\textrm{Hopf}}, \boldsymbol{\mu}^2_{\textrm{Hopf}} \right[$ the eigenvalues of $Df_{\boldsymbol{\mu}} \left(\mathcal{O}_{\boldsymbol{\mu}} \right)$ are of the form
$$
\lambda_{\textrm{u}}(\boldsymbol{\mu} ) \pm i \omega(\boldsymbol{\mu} ) \quad \textrm{and}\quad -\lambda_{\textrm{s}}(\boldsymbol{\mu} ),
$$
where
$$
\lambda_{\textrm{u}}(\boldsymbol{\mu} ), \omega(\boldsymbol{\mu} ), \lambda_{\textrm{s}}(\boldsymbol{\mu} )>0,
\quad
2\lambda_{\textrm{u}} (\boldsymbol{\mu} )< \lambda_{\textrm{s}} (\boldsymbol{\mu} )
\quad
\textrm{ and }
\quad
\frac{d}{d\boldsymbol{\mu}}\left( \frac{\lambda_{u}(\boldsymbol{\mu})}{\lambda_{s}(\boldsymbol{\mu})} \right)\neq 0.
$$
\end{fact}

\begin{fact}\label{hom}
There are non-trivial subintervals of  $  \left]\mu_3, \boldsymbol{\mu}^2_{\textrm{Hopf}} \right[$ where $f_{\boldsymbol{\mu}}$ has one positive LE 
(Figure \ref{fig:LE}).
\end{fact}

\begin{figure}[h]
\includegraphics[width=12cm]{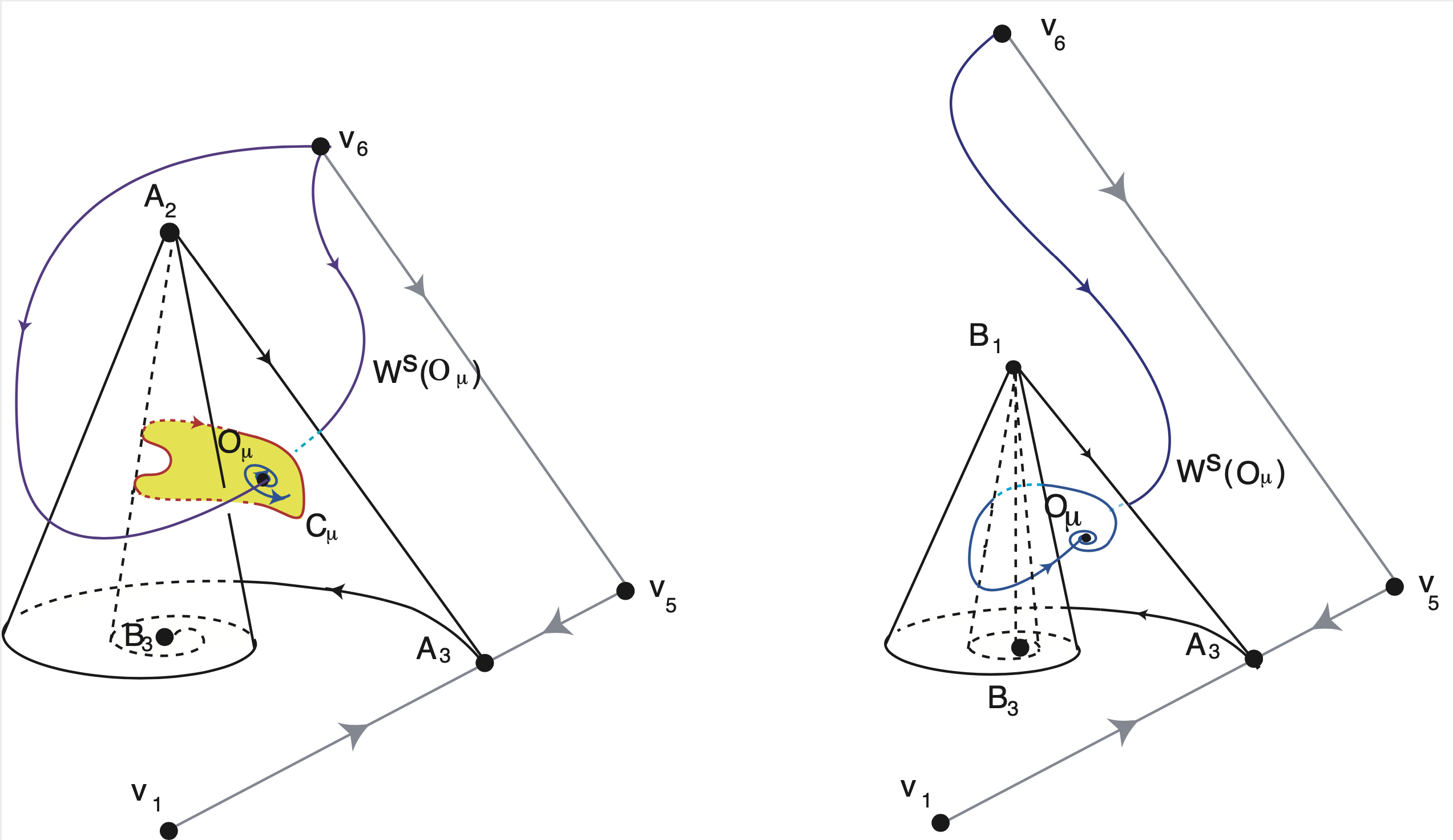}
\caption{\small{\emph{The attracting whirl\-pool
(see Proposition \ref{prop:central_manfld_VI})
:} illustration of Fact~\ref{fact_g} (left) and Fact~\ref{hom} (right).
For the values of $\boldsymbol{\mu}  $ for which the greatest LE of $f_{\boldsymbol{\mu}}$ is positive, there exists a vector field   close to $f_{\boldsymbol{\mu}}$ whose flow exhibits a homoclinic orbit to $\mathcal{O}_{\boldsymbol{\mu}}$.
A  description of the attracting whirl\-pool is given in detail in Subsection \ref{subsec:att_whirlpool}.}} \label{fig:shilnikov_cycle}
\end{figure}

\begin{fact}\label{Ws(B_1)}
For    $\boldsymbol{\mu} \in \,\,]\frac{110}{31}, 8[$ , we have $[v_5 \rightarrow v_6]\subset \overline{W^s(B_1)}$.
\end{fact}

Let $I\subset \left]\boldsymbol{\mu}_{\textrm{SA}}, \boldsymbol{\mu}^2_{\textrm{Hopf}} \right[$ be a non-degenerate interval of $\boldsymbol{\mu}$-values 
 for which the greatest LE of $f_{\boldsymbol{\mu}}$ is positive.
  
\begin{fact}\label{hom2}
For $\boldsymbol{\mu} \in I \subset \left]\boldsymbol{\mu}_{\textrm{SA}}, \boldsymbol{\mu}^2_{\textrm{Hopf}} \right[$, there exists a $C^2$--vector field $g\in \mathcal{X}$ arbitrarily close to $f_{\boldsymbol{\mu}}$ (in the $C^2$-topology), whose flow exhibits a homoclinic orbit to the hyperbolic continuation of $\mathcal{O}_{\boldsymbol{\mu}}$ (Figure \ref{fig:stbl_unstbl_O_mu=3.6}).\\
\end{fact}

 \subsection{Digestive remark}
 \label{ss: digestive}

Facts \ref{fact_e}--\ref{hom2} have been numerically checked  both in   $\textit{Mathematica Wolfram}^{\circledR}$ and $\textrm{MATLAB}^{\circledR}$.

We give a heuristic justification why Fact \ref{hom2} may be considered a consequence of Fact \ref{hom}.  The existence of a positive Lyapunov exponent for $f_ {\boldsymbol{\mu}}$ (Fact \ref{hom}) implies the existence of an invariant subset of $\RR^3$ with positive topological entropy  (Corollary 4.3 of \cite{katok1980lyapunov}), leading to the existence of a transverse homoclinic point and horseshoes.  
Besides the equilibrium $\mathcal{O}_{\boldsymbol{\mu}}$ and the attracting periodic solution $\mathcal{C}_{\boldsymbol{\mu}}$, the flow of  \eqref{eq:poly_rep_8} does not have more nontrivial compact invariant sets in $\inter([0,1]^3)$ for $\mu<\boldsymbol{\mu}_{SA}$. Therefore, the only way  to realize these horseshoes is through a homoclinic cycle to the hyperbolic continuation of $\mathcal{O}_{\boldsymbol{\mu}}$ for $\mu>\boldsymbol{\mu}_{SA}$ (Fact \ref{hom2}). This agrees with numerics.

\begin{figure}[h]
    \includegraphics[width=10cm]{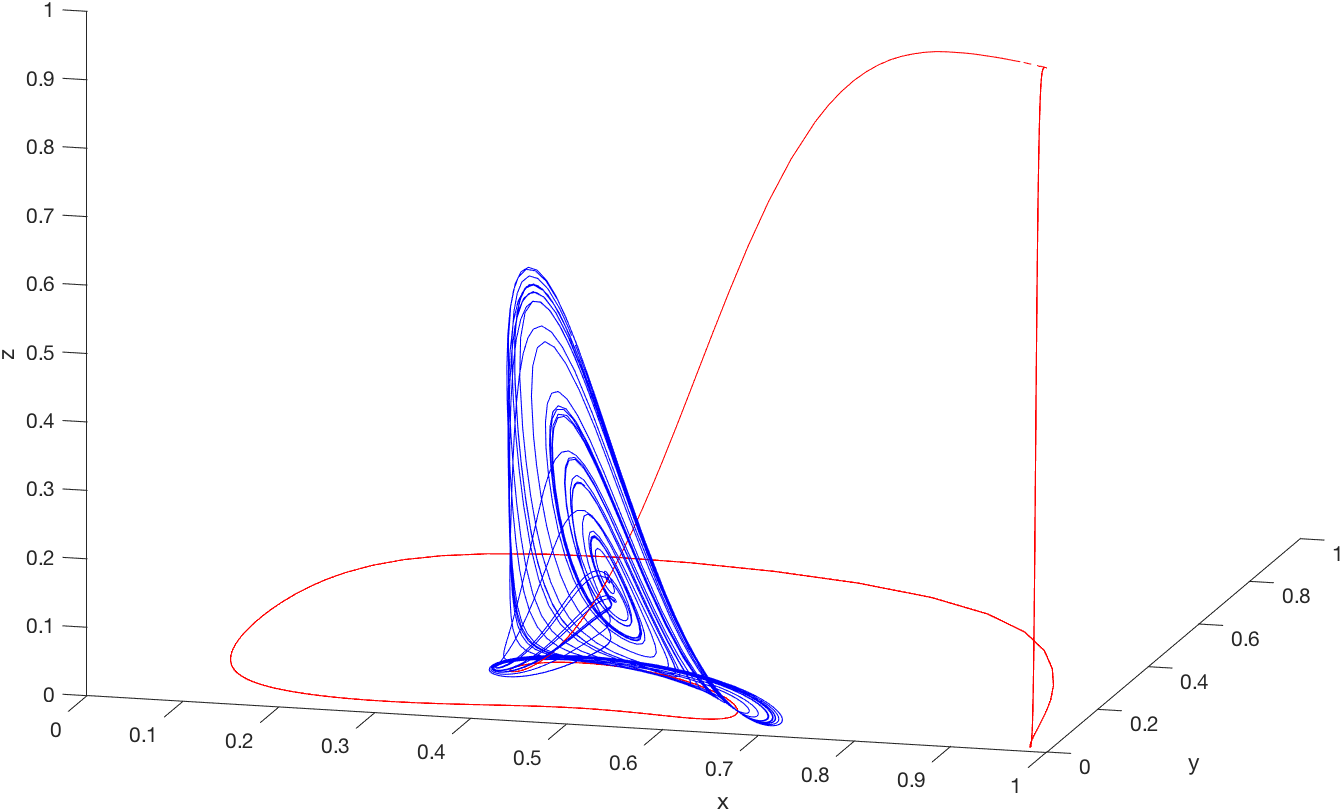}
    \caption{\small{The plot of the stable (in red) and the unstable (in blue) manifolds of $\mathcal{O}_{\boldsymbol{\mu}}$, for system~\eqref{eq:poly_rep_8} with $\boldsymbol{\mu}=3.6$.
}}\label{fig:stbl_unstbl_O_mu=3.6}
\end{figure}

\section{Route to strange attractors}
\label{s:route}

Using the same type of arguments of \cite{hofbauer1998evolutionary, gaunersdorfer1995fictitious}, we explain the global dynamics of system \eqref{eq:poly_rep_8} in $\inter\left( [0,1]^3 \right)$ according to the local bifurcations  studied in Section \ref{s: bifurcation analysis}. 
Based on the transcritical bifurcations of the equilibria on the boundary and the bifurcations of $\mathcal{O}_{\boldsymbol{\mu}}$, we   distinguish the  seven cases described in Table~\ref{tbl:cases} and we make use of the \emph{Facts} stated in Section \ref{s: numerical analysis}, to prove the existence of chaos. 
We also denote by $\FF$ the union of all faces, i.e. $\FF =\{\sigma_1,\dots ,\sigma_6\}$.

\begin{table}[h]
\begin{small}
\begin{tabular}{|c|c|c|c|c|}\toprule
\multicolumn{2}{|c|} {Case}& \multicolumn{3}{c|}{Interval of $\boldsymbol{\mu}$} \\ \toprule
\multirow{3}{*}{$\RN{1}$}  	& $\RN{1}.1$  & \multicolumn{3}{c|}{$]-\frac{2938}{95}, \boldsymbol{\mu}_{\textrm{Belyakov}}[$} \\[1mm] \cmidrule{2-5}
									  		& $\RN{1}.2$  & \multicolumn{3}{c|}{$]\boldsymbol{\mu}_{\textrm{Belyakov}}, \boldsymbol{\mu}^1_{\textrm{Hopf}}[$} \\[1mm] \cmidrule{2-5}
                  			 			  	& $\RN{1}.3$  & \multicolumn{3}{c|}{$]\boldsymbol{\mu}^1_{\textrm{Hopf}}, -\frac{122}{7}[$} \\[1mm] \midrule
\multicolumn{2}{|c|}{$\RN{2}$} & \multicolumn{3}{c|}{$]-\frac{122}{7}, -12[$}   \\[1mm] \midrule
\multicolumn{2}{|c|}{$\RN{3}$}  & \multicolumn{3}{c|}{$]-12, -8[$}   \\[1mm] \bottomrule
\end{tabular}
\qquad 
\begin{tabular}{|c|c|c|c|c|}\toprule
\multicolumn{2}{|c|} {Case}& \multicolumn{3}{c|}{Interval of $\boldsymbol{\mu}$} \\ \toprule
\multicolumn{2}{|c|}{$\RN{4}$}  & \multicolumn{3}{c|}{$]-8, -6[$}   \\[1mm] \midrule
\multicolumn{2}{|c|}{$\RN{5}$}  & \multicolumn{3}{c|}{$]-6, \frac{110}{31}[$}   \\[1mm] \midrule
\multicolumn{2}{|c|}{$\RN{6}$}  & \multicolumn{3}{c|}{$]\frac{110}{31}, 8[$}   \\[1mm] \midrule
\multirow{2}{*}{$\RN{7}$}   & $\RN{7}.1$  & \multicolumn{3}{c|}{$]8, \boldsymbol{\mu}^2_{\textrm{Hopf}}[$} \\[1mm] \cmidrule{2-5}
										& $\RN{7}.2$  & \multicolumn{3}{c|}{$]\boldsymbol{\mu}^2_{\textrm{Hopf}}, 10[$} \\[1mm] \bottomrule
\end{tabular}
\end{small}
\vspace{.3cm}
  \captionof{table}{\small{The sub-intervals (that we designate as Cases) of $]-\frac{2938}{95},10[$ based on the values of $\boldsymbol{\mu}$ for which the bifurcations occur on the boundary (Cases $\RN{1}, \dots, \RN{7}$) and on the interior equilibrium (Cases $\RN{1}.1, \RN{1}.2, \RN{1}.3$ and $\RN{7}.1, \RN{7}.2$).}}\label{tbl:cases}
\end{table}

We illustrate our results with numerical simulations for values of the parameter in each case of Table \ref{tbl:cases}.
In Appendices \ref{bd_appendix} and ~\ref{int_appendix}, all figures are collected to allow a global understanding of the route to chaos. We divide the pictures in two cases: dynamics on the boundary (Table~\ref{tbl:Bd_dynamics_on_mu}) and on the cube's interior (Table~\ref{tbl:Int_dynamics_on_mu}).

\begin{prop}\label{prop:central_manfld_I}
In Case $\RN{1}$,  there exists an invariant and attracting two-dimen\-sional set $\Sigma_{\boldsymbol{\mu}}$ containing the points $A_1, A_2, A_3, A_4, B_2$ and $\mathcal{O}_{\boldsymbol{\mu}}$, and the heteroclinic connections $[A_2\to A_3]$, $[A_3\to A_4]$, $[A_4\to A_1]$, $[A_1\to B_2]$, and $[A_2\to B_2]$. Moreover,
\begin{enumerate}
	\item In Cases $\RN{1}.1$   and $\RN{1}.2$, if $p\in\Sigma_{\boldsymbol{\mu}}\backslash \FF$, then its $\omega$-limit is $\mathcal{O}_{\boldsymbol{\mu}}$. 
	\item In Case $\RN{1}.3$, if $p\in\Sigma_{\boldsymbol{\mu}}\backslash \left( \FF \cup W^s(\mathcal{O}_{\boldsymbol{\mu}}) \right)$, then its $\omega$-limit is $\mathcal{C}_{\boldsymbol{\mu}}$ (\footnote{The set $\mathcal{C}_{\boldsymbol{\mu}}$ is the periodic solution which emerges from the Hopf bifurcation described in the proof of Lemma \ref{corollary1}.}).
\end{enumerate}
In the three Cases,  $\inter\left( [0,1]^3 \right)$ is divided by $\Sigma_{\boldsymbol{\mu}}$ in two connected components.
\end{prop}

\begin{proof} By Lemma \ref{Lemma3}, we know that the faces are invariant.
In Cases $\RN{1}.1$ and $\RN{1}.2$, besides the attracting interior equilibrium, there are no more invariant compact sets in $\inter\left( [0,1]^3 \right)$.
Therefore, analysing the direction of the flow, the $\omega$-limit of all points in the cube's interior is the two-dimensional set bounded by the heteroclinic connections $[A_2\to A_3]$, $[A_3\to A_4]$, $[A_4\to A_1]$, $[A_1\to B_2]$, and $[A_2\to B_2]$ as depicted in Figure \ref{fig:02_mu=1,1, extra1} (left).
This defines a two-dimensional set $\Sigma_{\boldsymbol{\mu}}$ containing $\mathcal{O}_{\boldsymbol{\mu}}$ which is attracting and invariant (see Figure \ref{fig:04_mu=-20_v1_v2}).

In Case $\RN{1}.3$, besides the interior equilibrium $\mathcal{O}_{\boldsymbol{\mu}}$,
system~\eqref{eq:poly_rep_8} exhibits an attracting periodic solution, $\mathcal{C}_{\boldsymbol{\mu}}$, lying on the attracting two-dimensional set $\Sigma_{\boldsymbol{\mu}}$ (observe that this plane is attracting)
(see Figure~\ref{fig:01_limit_cycle} (left)), which emerge from a transcritical Hopf bifurcation by Lemma \ref{corollary1}. This two-dimensional set contains 
$\overline{W^s(\mathcal{C}_{\boldsymbol{\mu}})}$.
In all cases, since the boundary of $\Sigma_{\boldsymbol{\mu}}$ belongs to the opposite faces $\sigma_3,\sigma_4$ and $\sigma_5,\sigma_6$ of the phase space, it divides its interior in two connected components.

\end{proof}

\begin{figure}[h]
\hspace{-10mm}
    \begin{subfigure}[t]{0.4\textwidth}\centering
        \includegraphics[width=6.5cm]{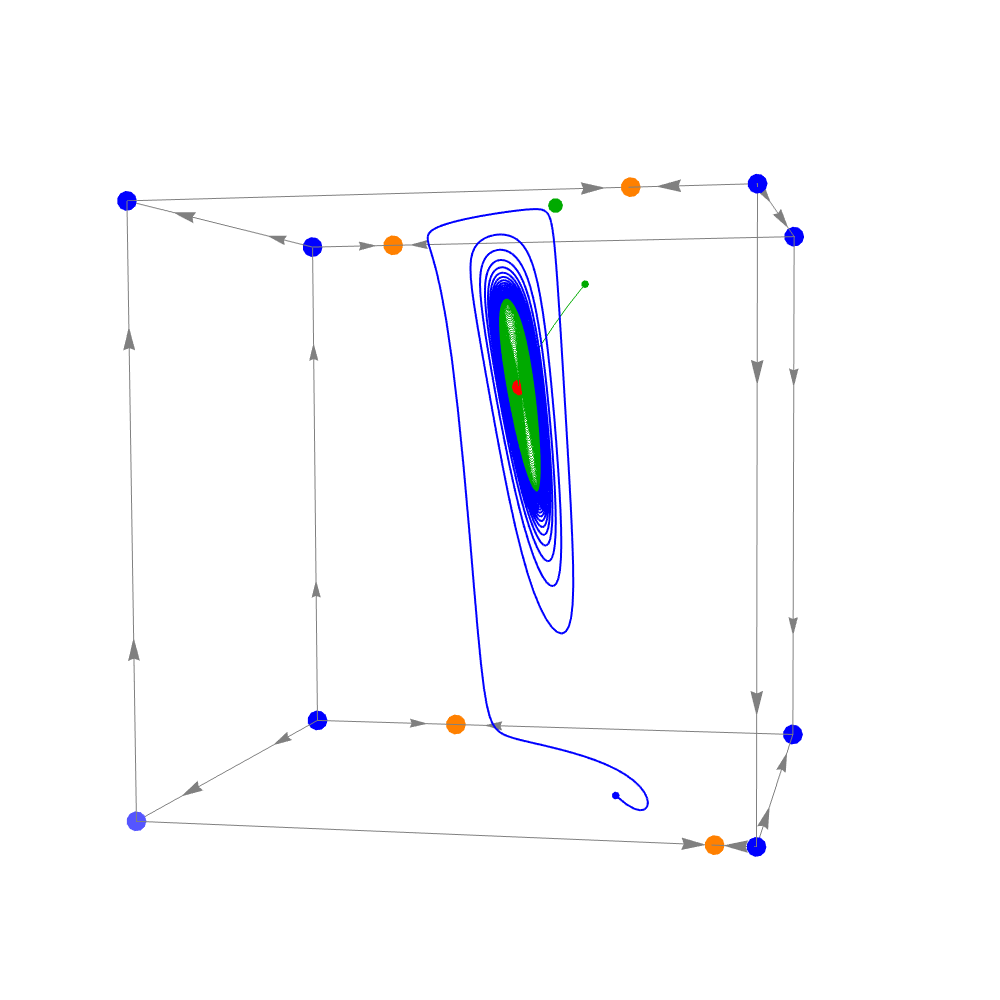}
    \end{subfigure}
    \quad \quad
    ~ 
    \begin{subfigure}[t]{0.4\textwidth}\centering
        \includegraphics[width=6.5cm]{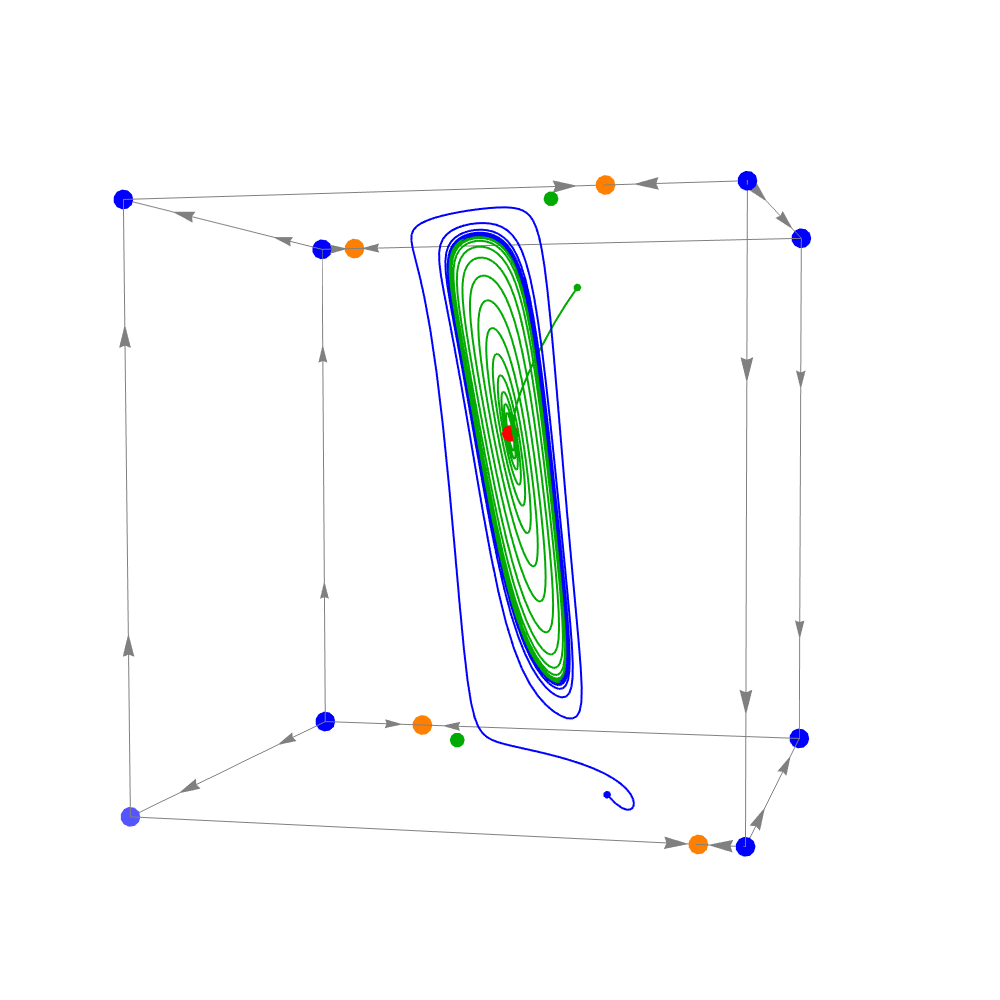}
    \end{subfigure}
    \vspace{-.7cm}
    \caption{\small{\textit{The limit cycle in Cases $\RN{1}$ and $\RN{2}$:} plot of two orbits (one in blue and one in green), the interior equilibrium, and all the boundary equilibria of system~\eqref{eq:poly_rep_8}, for $\boldsymbol{\mu}=-17.5$ (left) and $\boldsymbol{\mu}=-14$ (right) with $t\in [0,100]$.}}\label{fig:01_limit_cycle}
\end{figure}

\begin{figure}[h]
\hspace{-12mm}
    \begin{subfigure}[t]{0.4\textwidth}
        \includegraphics[width=6cm]{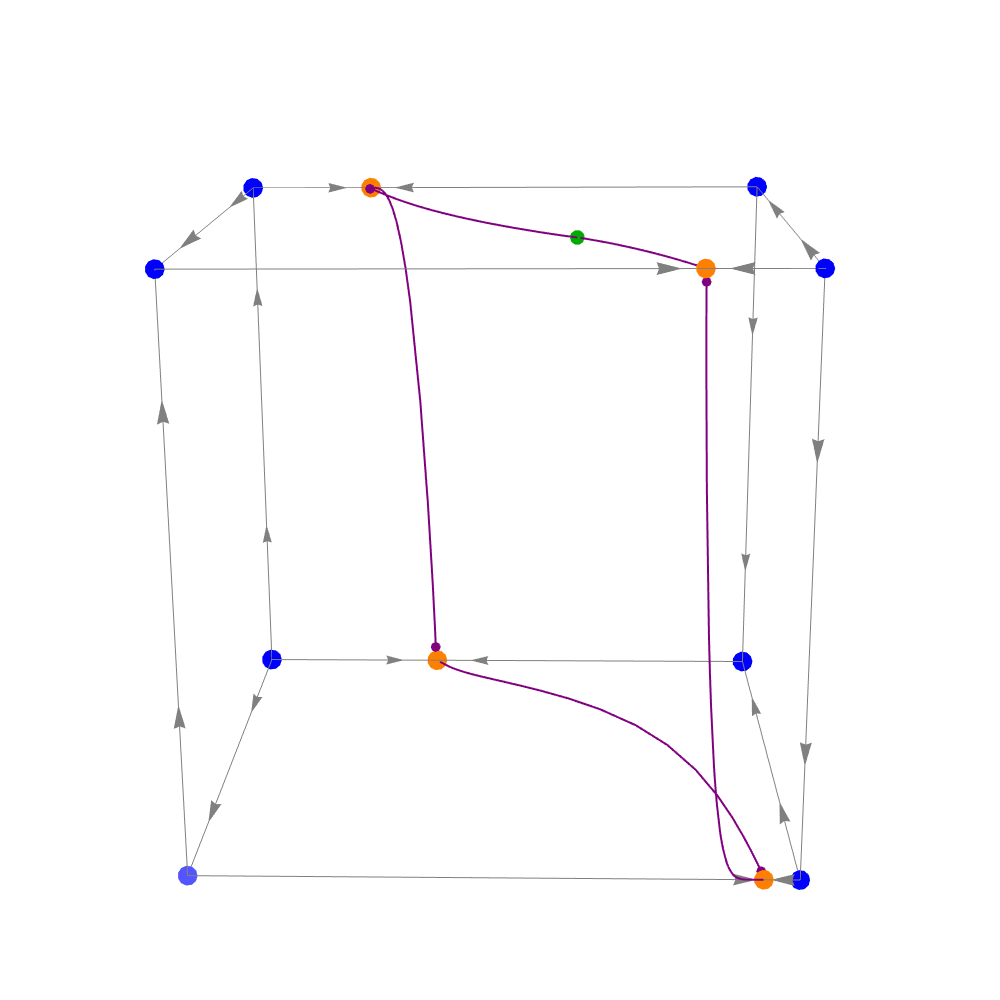}
    \end{subfigure}
    \quad
    \begin{subfigure}[t]{0.4\textwidth}
        \includegraphics[width=6cm]{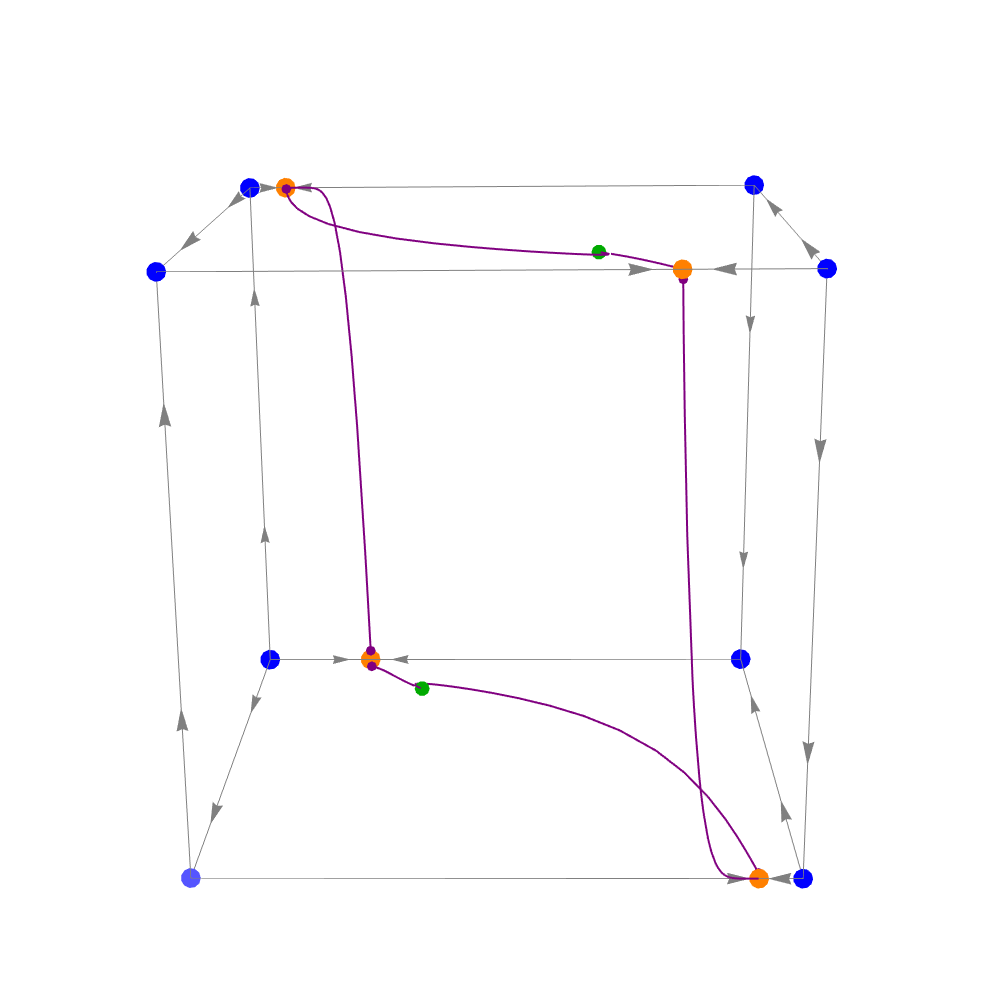}
    \end{subfigure}
    \vspace{-.5cm}
     \caption{\small{\textit{Boundary of $\Sigma_{\boldsymbol{\mu}}$ in Cases $\RN{1}$ and $\RN{2}$:} illustration of the dynamics on the boundary in Cases I (left, $\boldsymbol{\mu}=-20$) and II (right, $\boldsymbol{\mu}=-14$). }}\label{fig:02_mu=1,1, extra1}
\end{figure}

\begin{prop}\label{prop:central_manfld_II}
In Case $\RN{2}$, there exists an invariant and attracting $2$-di\-men\-sio\-nal set $\Sigma_{\boldsymbol{\mu}}$ containing the points $A_1, A_2, A_3, A_4, B_2, B_3$, $\mathcal{O}_{\boldsymbol{\mu}}$, and $\mathcal{C}_{\boldsymbol{\mu}}$.
If $p\in\Sigma_{\boldsymbol{\mu}}\backslash \left( \FF \cup \{ \mathcal{O}_{\boldsymbol{\mu}} \} \right)$, then its $\omega$-limit is $\mathcal{C}_{\boldsymbol{\mu}}$.
The set  $\inter\left( [0,1]^3 \right)$ is divided by $\Sigma_{\boldsymbol{\mu}}$ in two connected components.
\end{prop}

\begin{proof}
The global dynamics in Case $\RN{2}$ is the same as in Case $\RN{1}.3$, with exception that  $A_4$ has undergone a transcritical bifurcation from where the saddle $B_3$ has evolved from a formal equilibrium to an equilibrium in the cube $[0,1]^3$ (see Figure \ref{fig:02_mu=1,1, extra1} (right)).
Using Table \ref{tbl:Eigenv_of_Bs}, we know that  $W^{\textrm{u}}(B_3)$ points out to the interior of the cube.
Thus the periodic solution $\mathcal{C}_{\boldsymbol{\mu}}$ is still the $\omega$-limit set of all points in $\inter\left( [0,1]^3 \right)\backslash\left( \{\mathcal{O}_{\boldsymbol{\mu}}\} \cup W^s(\mathcal{O}_{\boldsymbol{\mu}})\right)$ (see Figure~\ref{fig:01_limit_cycle} (right)).
Notice also that the $\inter\left( [0,1]^3 \right)$ is still divided by $\Sigma_{\boldsymbol{\mu}}$ in two connected components.

\end{proof}

\begin{rem}
The difference between Cases $\RN{1}$ and $\RN{2}$ is that $B_3$ appears as an equilibrium on the cube  in the second scenario.
However, the ``interior dynamics'' does not change qualitatively.
\end{rem}

\begin{prop}\label{prop:central_manfld_III}
In Case $\RN{3}$, there exists an invariant and attracting two-dimensional set $\Sigma_{\boldsymbol{\mu}}$ containing the points $v_2, v_4, A_2, A_3, A_4, B_2, B_3$, $\mathcal{O}_{\boldsymbol{\mu}}$ and $\mathcal{C}_{\boldsymbol{\mu}}$.
The set  $\inter\left( [0,1]^3 \right)$ is divided by $\Sigma_{\boldsymbol{\mu}}$ in two connected components.
\end{prop}

\begin{proof}
The $\omega$-limit of all points in $\inter\left(\sigma_5 \right)$ and $\inter\left(\sigma_6 \right)$ is $B_2$ and $B_3$, respectively (see Figure \ref{fig:02_mu=1,1, extra2} (left)).
The structure of the two-dimensional set $\Sigma_{\boldsymbol{\mu}}$ on the face $\sigma_5$ comes from Proposition~\ref{prop:central_manfld_II} and the fact that $A_1$ collapses with the vertex $v_4$ through a transcritical bifurcation (see   Table~\ref{tbl:Eigenv_of_As}).
Notice also that, by continuity from Case $\RN{2}$, the $\inter\left( [0,1]^3 \right)$ is still divided by $\Sigma_{\boldsymbol{\mu}}$ in two connected components (see Figure \ref{fig:02_limit_cycle} (left)).

\end{proof}

\begin{figure}[h]
\hspace{-10mm}
    \begin{subfigure}[t]{0.4\textwidth}\centering
        \includegraphics[width=6.5cm]{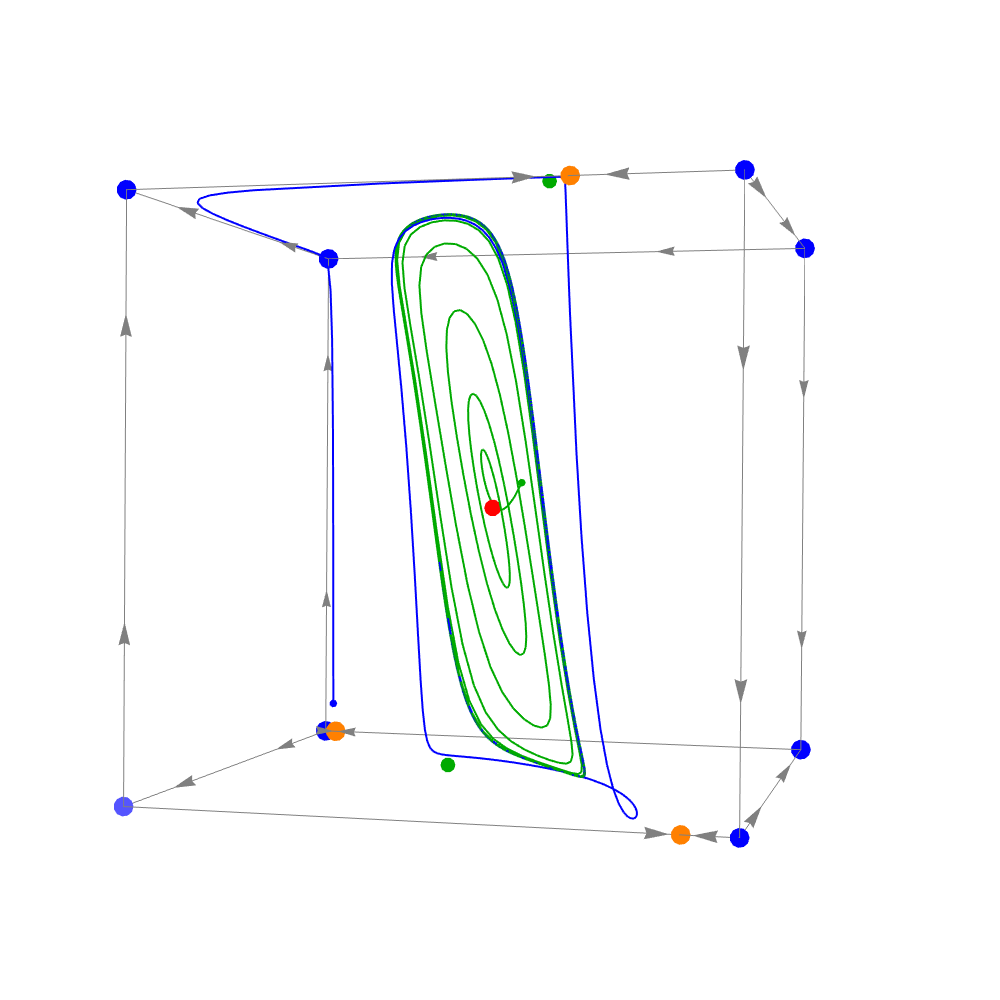}
    \end{subfigure}
    \quad \quad
    ~ 
    \begin{subfigure}[t]{0.4\textwidth}\centering
        \includegraphics[width=6.5cm]{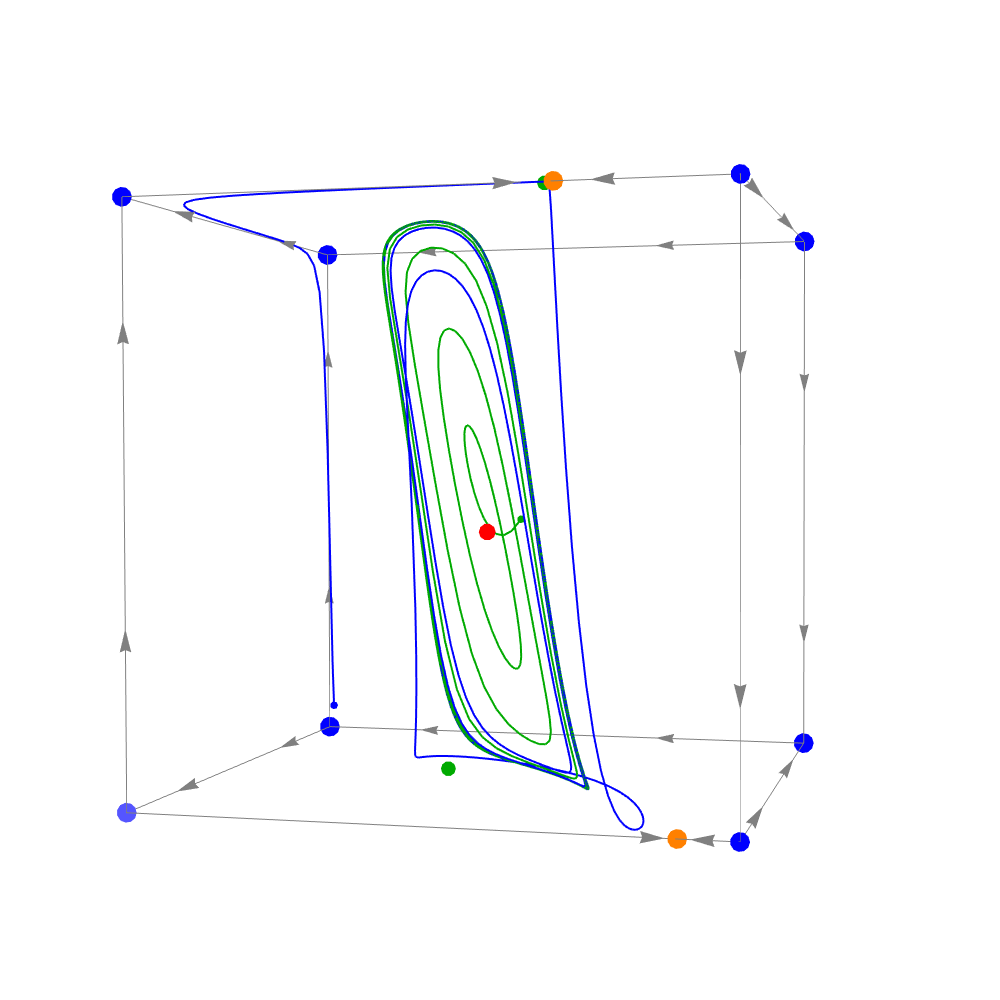}
    \end{subfigure}
    \vspace{-.7cm}
    \caption{\small{\textit{The boundary of $\Sigma_{\boldsymbol{\mu}}$ and the limit cycle in Cases $\RN{3}$ and $\RN{4}$:} plot of two orbits (one in blue and one in green), the interior equilibrium, and all the boundary equilibria of system~\eqref{eq:poly_rep_8}, for $\boldsymbol{\mu}=-8.5$ (left) and $\boldsymbol{\mu}=-7$ (right) with $t\in [0,100]$.}}\label{fig:02_limit_cycle}
\end{figure}

\begin{prop}\label{prop:central_manfld_IV}
In Case $\RN{4}$,  there exists an invariant and attracting two-dimensional set $\Sigma_{\boldsymbol{\mu}}$ containing the points $v_2, v_3, v_4, A_2, A_3, B_2, B_3, \mathcal{O}_{\boldsymbol{\mu}}$, and $\mathcal{C}_{\boldsymbol{\mu}}$. The set  $\inter\left( [0,1]^3 \right)$ is divided by $\Sigma_{\boldsymbol{\mu}}$ in two connected components.
\end{prop}

\begin{proof}
From Case $\RN{3}$ to Case $\RN{4}$, the equilibrium $A_4$ disappears through a transcritical bifurcation (Table~\ref{tbl:Eigenv_of_As}).
Since the $\omega$-limit of all points in $\inter\left(\sigma_5 \right)$ and $\inter\left(\sigma_6 \right)$ is $B_2$ and $B_3$, respectively (see Figure \ref{fig:02_mu=1,1, extra2} (right)), the two-dimensional set $\Sigma_{\boldsymbol{\mu}}$ of Case $\RN{3}$ gives rise to a  two-dimensional set containing $\mathcal{C}_{\boldsymbol{\mu}}$ and $\mathcal{O}_{\boldsymbol{\mu}}$ (see Figure \ref{fig:02_limit_cycle} (right)).
Notice that $\inter\left( [0,1]^3 \right)$ is divided by $\Sigma_{\boldsymbol{\mu}}$ in two connected components, and the latter set is still attracting.

\end{proof}

\begin{figure}[h]
\hspace{-12mm}
    \centering
    \begin{subfigure}[t]{0.4\textwidth}
        \includegraphics[width=6cm]{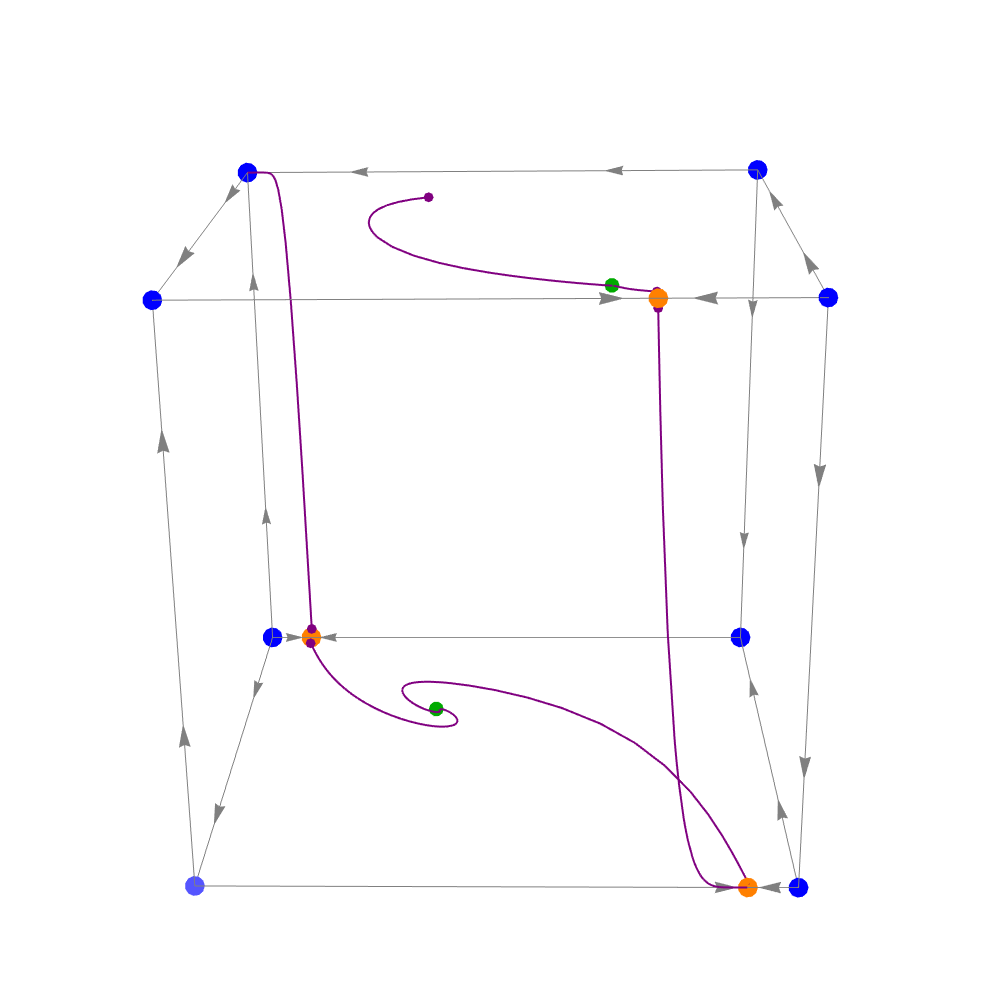}
    \end{subfigure}
    \quad
    \begin{subfigure}[t]{0.4\textwidth}
        \includegraphics[width=6cm]{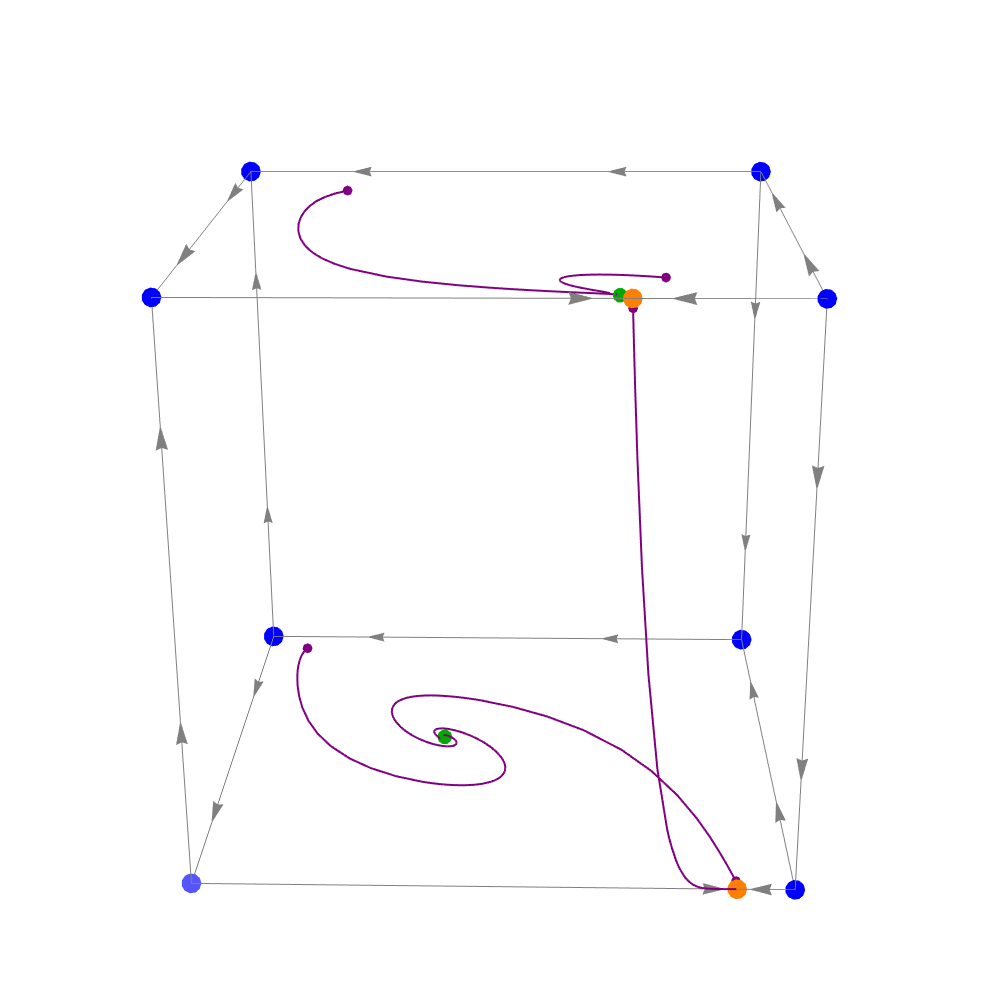}
    \end{subfigure}
    \vspace{-.5cm}
     \caption{\small{\textit{The boundary of $\Sigma_{\boldsymbol{\mu}}$ in Cases $\RN{3}$ and $\RN{4}$:} illustration of the dynamics on the boundary in Cases III (left, $\boldsymbol{\mu}=-10$) and IV (right, $\boldsymbol{\mu}=-7$).  }}\label{fig:02_mu=1,1, extra2}
\end{figure}

By Fact \ref{fact_e}, for $\boldsymbol{\mu} \in\,]    \boldsymbol{\mu}^1_{\textrm{Hopf}}, -8[$ (Cases $\RN{1}$--$\RN{3}$),  the two branches of $W^s(\mathcal{O}_{\boldsymbol{\mu}})$ are connected with the sources $v_3$ and $v_6$. From Case $\RN{4}$ on, the equilibrium $v_3$ changes stability and the two branches of  $W^s(\mathcal{O}_{\boldsymbol{\mu}})$  are connected with $v_6$ in different eigenspaces.
Dramatic changes occur in Case $\RN{5}$.

\begin{prop}\label{prop:central_manfld_V}
In Case $\RN{5}$,  there exists an invariant and attracting two-dimensional set $\Sigma_{\boldsymbol{\mu}}$ containing the points $A_2, A_3, B_3$, $\mathcal{O}_{\boldsymbol{\mu}}$, and $\mathcal{C}_{\boldsymbol{\mu}}$. This manifold is  $\overline{W^s(\mathcal{C}_{\boldsymbol{\mu}})}$ and the set  $\Sigma_{\boldsymbol{\mu}}$ does not divide $\inter\left( [0,1]^3 \right)$ in two connected components.
Moreover, there exists $\tilde{\boldsymbol{\mu}}\gtrsim\boldsymbol{\mu}_{\textrm{SA}}$ such that $f_{\tilde{\boldsymbol{\mu}}}$ is $C^2$-close to a vector field exhibiting strange attractors.

\end{prop}

\begin{proof}
From Case $\RN{4}$ to Case $\RN{5}$, the equilibrium $B_2$ disappears through a transcritical bifurcation (see Figure~\ref{fig:02_mu=1,1, extra3} (left)) and a screwed attracting two-dimensional set $\Sigma_{\boldsymbol{\mu}}$ with a singular point at $A_2$ emerges.  The manifold $W^s(\mathcal{C}_{\boldsymbol{\mu}})$ is spreading along $\inter\left( [0,1]^3 \right)$ (see Figure~\ref{fig:03_limit_cycle} (left)).
The last assertion is a consequence of  Fact~\ref{hom},  for the  $\boldsymbol{\mu}$-values for which  $f_{\boldsymbol{\mu}}$ has a positive LE, and Theorem  \ref{thrm:main}.

\end{proof}

\begin{figure}[h]
\hspace{-10mm}
    \begin{subfigure}[t]{0.4\textwidth}\centering
         \includegraphics[width=6.5cm]{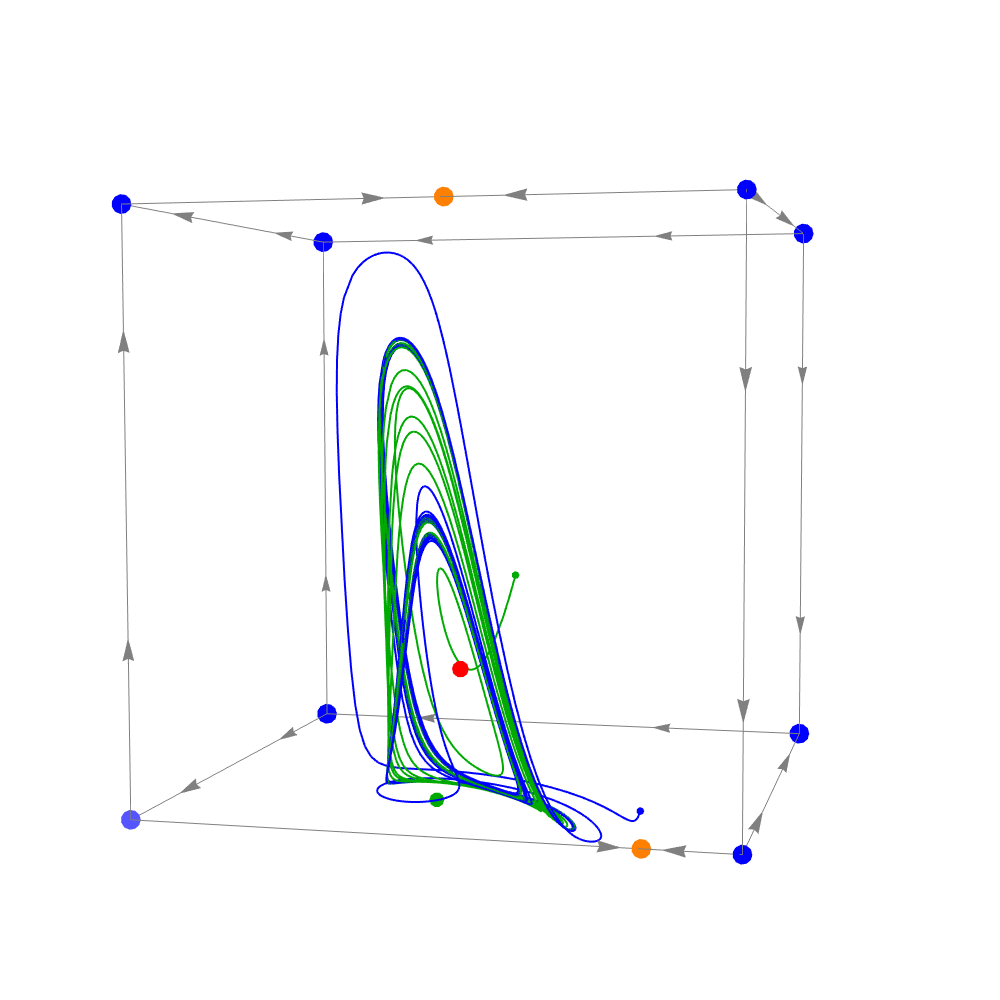}
    \end{subfigure}
    \quad \quad
    ~ 
    \begin{subfigure}[t]{0.4\textwidth}\centering
        \includegraphics[width=6.5cm]{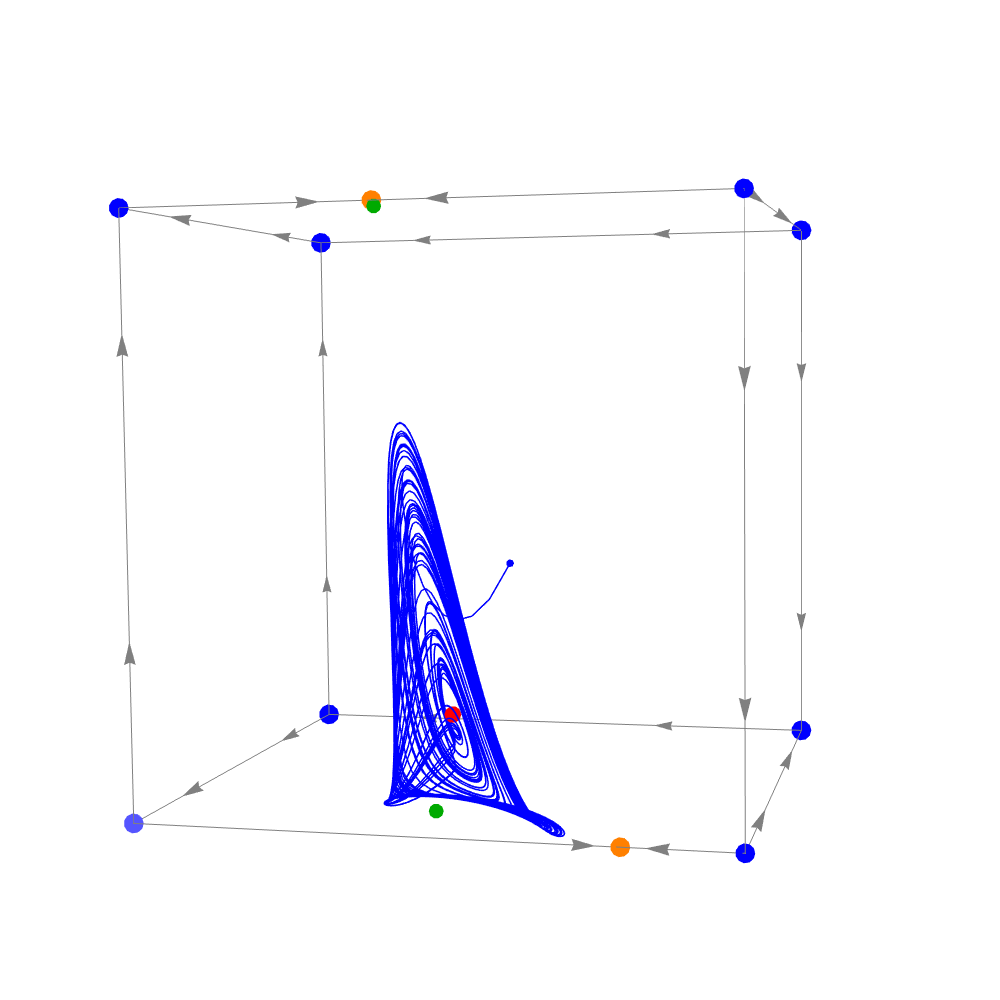}
    \end{subfigure}
    \vspace{-.7cm}
    \caption{\small{\textit{Switching from regular (Case $\RN{5}$) to complex (Case $\RN{6}$) dynamics:} plot of two orbits (left) and one orbit (right), the interior equilibrium, and all the boundary equilibria of system~\eqref{eq:poly_rep_8}, for $\boldsymbol{\mu}=1.1$ (left) and $\boldsymbol{\mu}=3.6$ (right) with $t\in [0,100]$.}}\label{fig:03_limit_cycle}
\end{figure}

The value $\boldsymbol{\mu} = \boldsymbol{\mu}_{\textrm{SA}}$ seems to be the  parameter which separates  regular (zero topological entropy) from chaotic dynamics. Before going into Case VI, notice that $W^s(A_2)$ is contained in face $\sigma_5$.

\begin{figure}[h]
\hspace{-15mm}
    \begin{subfigure}[t]{0.3\textwidth}\centering
        \includegraphics[width=5.5cm]{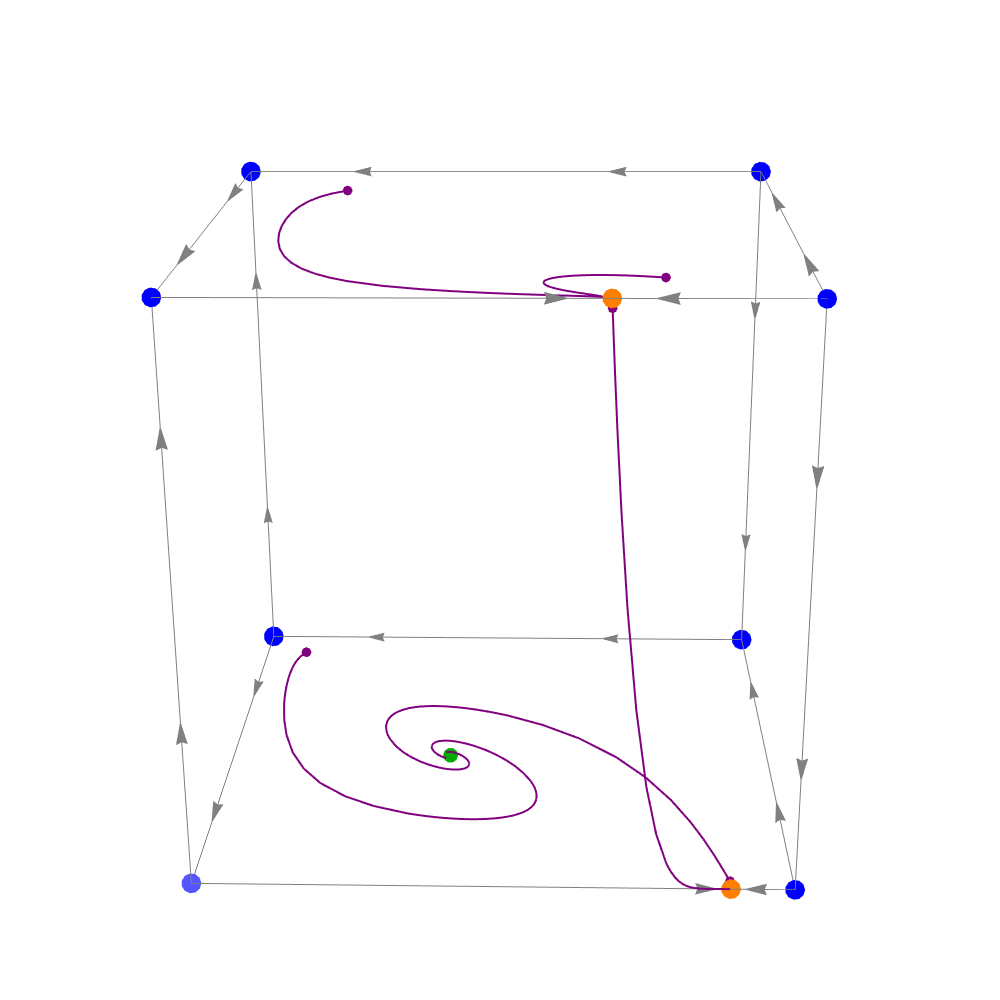}
    \end{subfigure}
    \qquad
    \begin{subfigure}[t]{0.3\textwidth}\centering
        \includegraphics[width=5.5cm]{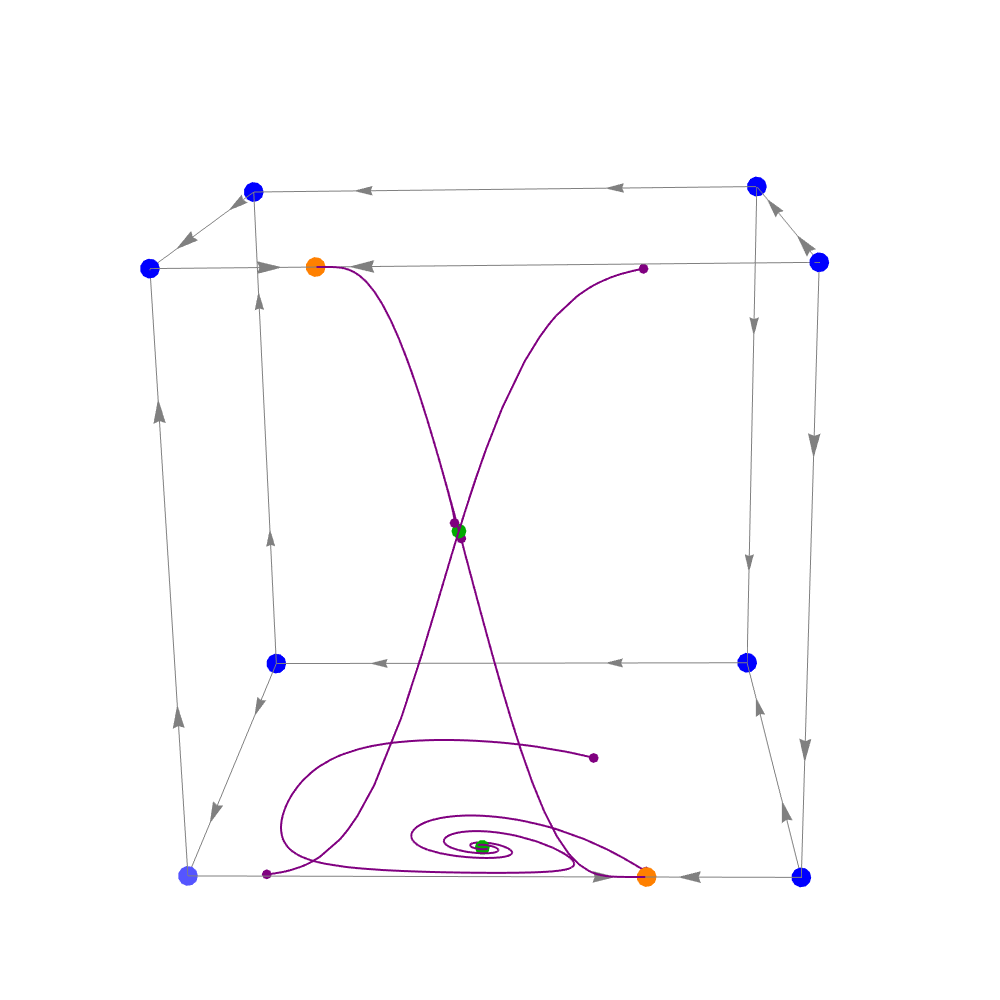}
    \end{subfigure}
    \qquad
    \begin{subfigure}[t]{0.3\textwidth}\centering
        \includegraphics[width=5.5cm]{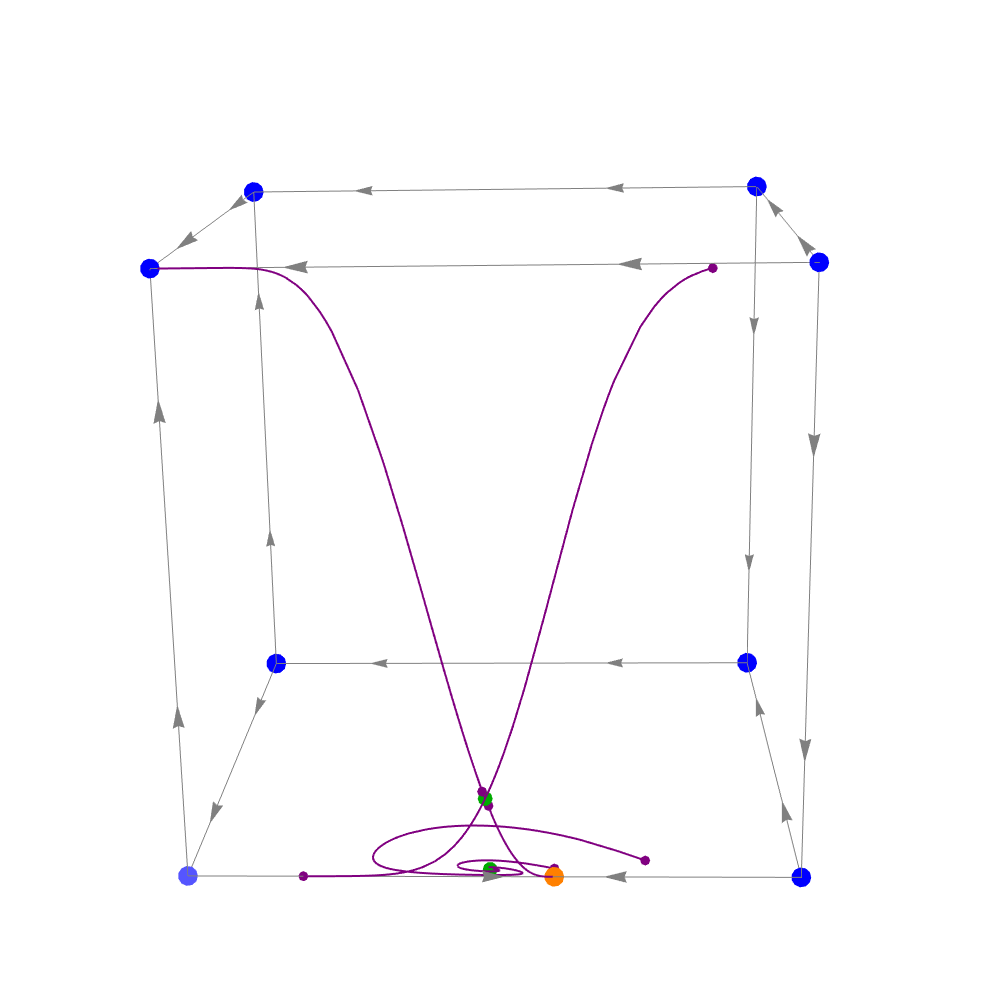}
    \end{subfigure}
    \vspace{-.5cm}
    \caption{\small{\textit{Dynamics on the boundary in Cases $\RN{5}$, $\RN{6}$ and $\RN{7}$:} illustration of the dynamics on the boundary in Cases V (left, $\boldsymbol{\mu}=-5$), VI (center, $\boldsymbol{\mu}=6$) and VII (right, $\boldsymbol{\mu}=9$). }}\label{fig:02_mu=1,1, extra3}
\end{figure}

\begin{prop}\label{prop:central_manfld_VI}
In Case $\RN{6}$, the set $\overline{W^s(B_1)}$ contains the points $v_1$, $v_3$, $v_5$, $v_6$, $v_7$  and $B_1$. The set $[0,1]^3\backslash \overline{W^s(B_1)}$ has two connected components: in one, the equilibrium $A_2$ is a sink; the other is dominated by the non-wandering set associated to a homoclinic orbit to $\mathcal{O}_{\boldsymbol{\mu}}$ for the values of $\boldsymbol{\mu}$ such that the greatest LE of $f_{\boldsymbol{\mu}}$ is positive.
\end{prop}

\begin{proof}
The proof of this result comes from a continuity analysis of Case $\RN{5}$  and Facts~~\ref{hom} and~\ref{Ws(B_1)}.
In $\inter\left( \sigma_4 \right)$, $B_1$ is a hyperbolic saddle and Lebesgue almost all points are attracted to either $A_2$ or
$A_3$ as depicted in Figure~\ref{fig:02_mu=1,1, extra3} (center).
Since, in Case $\RN{5}$, $W^{\textrm{s}}(A_2)\cap \inter\left(\sigma_5 \right)=\inter\left(\sigma_5 \right)$ and
the equilibrium $B_1$ comes from $A_2$ through a transcritical bifurcation (see Figure~\ref{fig:03_limit_cycle} (right)), it turns out that, in Case $\RN{6}$, $W^{\textrm{s}}(B_1)\,\cap \,\inter\left( [0,1]^3 \right)$ is a two-dimensional invariant manifold whose shape is governed by the internal dynamics that divides the $\inter\left( [0,1]^3 \right)$ in two connected components: the one whose solutions have $\omega$-limit equal to $A_2$, and the one that contains the interior equilibrium $\mathcal{O}_{\boldsymbol{\mu}}$.
By the same argument as in the proof of Proposition \ref{prop:central_manfld_V}, the last assertion follows.

\end{proof}

\begin{prop}\label{prop:central_manfld_VII}
In Case $\RN{7}$, the set $\overline{W^s(B_1)}$ contains the points $v_1$, $v_3$, $v_6$, $v_7$ and $B_1$. The set $\Gamma_{(2,2,2)}\backslash 
\overline{W^s(B_1)}$ has two connected components: in one the  equilibrium $v_2$ is a global sink; in the other, we have two sub-Cases:
\begin{enumerate}
\item in Case $\RN{7}.1$, for the values of $\boldsymbol{\mu}$ such that the greatest LE of $f_{\boldsymbol{\mu}}$ is positive, $f_{\boldsymbol{\mu}}$ is $C^2$-close to a vector field exhibiting strange attractors.

\item in Case $\RN{7}.2$, $\mathcal{O}_{\boldsymbol{\mu}}$ is a global sink.
\end{enumerate}
\end{prop}

\begin{proof}
The proof of this result replicates that of Proposition \ref{prop:central_manfld_V}, except that equilibrium $A_2$ no longer exists (see Figure~\ref{fig:02_mu=1,1, extra3} (right)). Notice also that the value of separation between Cases $\RN{7}.1$ and $\RN{7}.2$ is $\boldsymbol{\mu} =\boldsymbol{\mu}^2_{\textrm{Hopf}}$, responsible for the disappearance of $\mathcal{C}_{\boldsymbol{\mu}}$. From this parameter value on, the proximity of the homoclinic cycle referred in Fact \ref{hom} is no longer valid (see Figure \ref{fig:LE}), and $\mathcal{O}_{\boldsymbol{\mu}}$ becomes  stable.

At  $\boldsymbol{\mu} =10$, the point $B_1$ collapses to $A_3$, meaning that volume of the connected component containing $\mathcal{O}_{\boldsymbol{\mu}}$ is shrinking and collapses to a point. 
In fact, at  $\boldsymbol{\mu} =10$, the points $B_1, B_2$ and $\mathcal{O}_{\boldsymbol{\mu}}$ collapse to $A_3$ (see $\boldsymbol{\mu}=6.5$ and $\boldsymbol{\mu}=8$ in Table~\ref{tbl:Int_dynamics_on_mu}).

\end{proof}

In Cases VI and VII.1, $\overline{W^s(B_1)}$ plays the role of  \emph{separatrix}: in one component, the $\omega$-limit is either $A_2$ (if it exists) or $v_2$ (if $A_2$ does not exist); in the other, the $\omega$-limit is a strange attractor or a limit cycle (see $\boldsymbol{\mu}=3.6$ and $\boldsymbol{\mu}=6.5$ in Table~\ref{tbl:Int_dynamics_on_mu}).


\section{Proof of Theorem \ref{thrm:main}}
\label{s: second part}

Recall, from Subsection \ref{ss: numerical facts}, that $I\subset \left]\boldsymbol{\mu}_{\textrm{SA}}, \boldsymbol{\mu}^2_{\textrm{Hopf}} \right[$ is a non-degenerate interval of $\boldsymbol{\mu}$-values 
 for which the greatest LE of $f_{\boldsymbol{\mu}}$ is positive (cf.  Figure~\ref{fig:LE}).  

 By Fact~\ref{hom2}, there exists an arbitrarily small  neighbourhood $\mathcal{U}\subset \mathcal{X}$ of  $f_{{\boldsymbol{\mu}}}, {\boldsymbol{\mu}}\in I$ and $g\in \mathcal{U}$ such that the flow of $g$ exhibits a homoclinic orbit to the hyperbolic continuation of   $\mathcal{O}_{{\boldsymbol{\mu}}}$ (in the $C^2$-topology). The eigenvalues of $Df_{{\boldsymbol{\mu}}} \left(\mathcal{O}_{{\boldsymbol{\mu}}} \right)$ have the form described in  Fact~\ref{fact_g}.
Reversing the time,  the previous configuration gives rise to a flow exhibiting a homoclinic cycle associated to the hyperbolic continuation of $\mathcal{O}_{{\boldsymbol{\mu}}}$, say $\gamma$, whose eigenvalues   satisfy the conditions stated in~\cite[Theorem 1.4]{homburg2002periodic}.

Define now a   one-parameter family of vector fields $(g_{\boldsymbol{\lambda}})_{\boldsymbol{\lambda}\in [-1,1]}\in \mathcal{U}$ unfolding $g$ generically\footnote{An equivalent definition of ``generic unfolding'' may be found on page 38 of Palis and Takens \cite{palis1995hyperbolicity}.} such that:
\begin{itemize}
\item $g_0\equiv g$ and
\item the invariant manifolds associated to the interior equilibrium split with non-zero speed with respect to $\boldsymbol{\lambda} $.
\end{itemize}

Let $\mathcal{T}$ be a small tubular neighbourhood of the cycle $\gamma$ and $\Sigma$ a cross section to $\gamma$. 
As illustrated in Figure \ref{fig:escolhe_1}, for $\boldsymbol{\lambda} \in [-1,1]$, let us denote by $\Pi_{\boldsymbol{\lambda}}$ the  first return map  to the compact  cross section $\Sigma\cap \mathcal{T}$, associated to the flow of $g_{\boldsymbol{\lambda}}$.

\begin{figure}[h]
	\includegraphics[width=13cm]{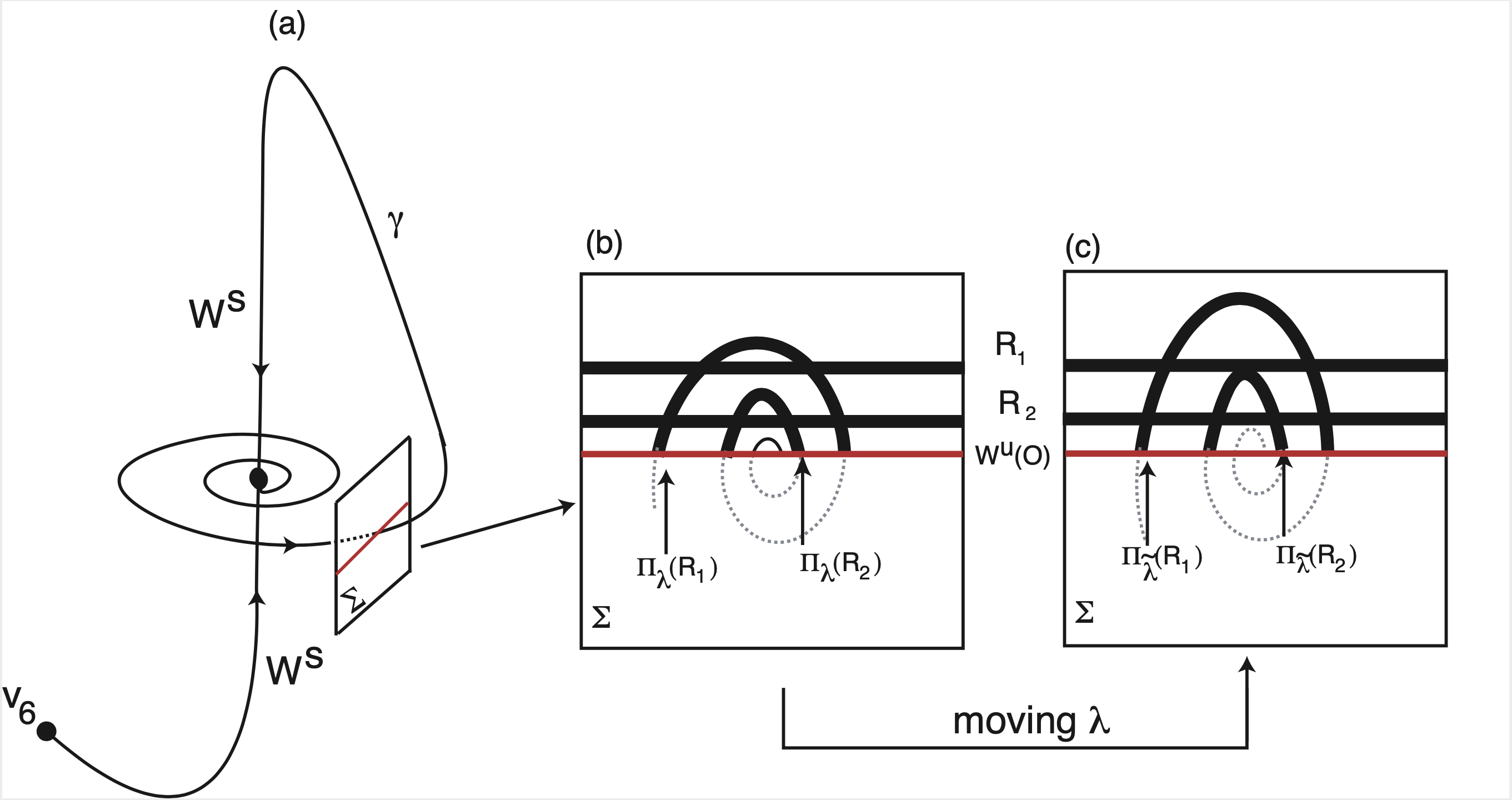}
    \caption{\small{ \emph{The homoclinic cycle $\gamma$:} scheme and shape of the first return map to the cross section $\Sigma$ for the associated to the flow of $g_{\boldsymbol{\lambda}}$, for different values of $\boldsymbol{\lambda}\in [-1,1]$ (b and c). In (b) and (c), the image, under the first return map, of the two horizontal rectangles $R_1$ and $R_2$ overlap the original rectangles.}} \label{fig:escolhe_1}
\end{figure}

Following Shilnikov \emph{et al} \cite{shilnikov1981bifurcation}, there exists a $\Pi_{0}$-invariant set of initial conditions $\Lambda_{0} \subset  \Sigma$ on which the map $\Pi_{0}$ is topologically conjugate to a full shift over an infinite number of symbols.  
By Gonchenko \emph{et al} \cite{gonchenko1996dynamical}, the set $\Lambda_{0} $ contains a sequence of hyperbolic horseshoes $(\mathcal{H}_n)_{n\in \NN}$ that are \emph{heteroclinically related}:  the unstable manifolds of the periodic orbits in $\mathcal{H}_n$ are long enough to intersect the stable manifolds of the periodic points of $\mathcal{H}_m$ (cf Figure \ref{fig:escolhe_2}), for $n, m \in \NN$. In other words,  there exist periodic solutions jumping from a strip of $\mathcal{H}_n$ to another strip of $\mathcal{H}_m$.  For $\boldsymbol{\lambda} \neq 0$, the homoclinic orbit $\gamma$ is broken, but finitely many horseshoes survive. We now use the following result:

\begin{prop}[\cite{ovsyannikov1986systems, newhouse1974diffeomorphisms}(adapted)]
With respect to the family of maps $(\Pi_{\boldsymbol{\lambda}})_{\boldsymbol{\lambda}\in [-1,1]}$, homoclinic tangencies associated to a dissipative periodic point are dense in $[-1,1]$.
\end{prop}

Therefore, for infinitely many parameters $\boldsymbol{\lambda} \in [-1,1]$ we may find a dissipative periodic point $c_1\in \mathcal{H}_n$, $n\in \NN$,   so that its stable and unstable manifolds 
have a homoclinic tangency.  This tangency is quadratic and breaks generically. 
Although the original tangencies are destroyed, when the
parameter $\boldsymbol{\lambda}$ varies,  new tangencies arise nearby. The family $(\Pi_{\boldsymbol{\lambda}})_{{\boldsymbol{\lambda}}}$ may be seen an unfolding
of a map exhibiting a quadratic homoclinic tangency; thus one can apply the results by Mora and Viana \cite{mora1993abundance}, which guarantee the existence of a
positive Lebesgue measure set $E\subset[-1,1]$ of parameter values  such that for $\boldsymbol{\lambda} \in E$ the diffeomorphism $\Pi_{\boldsymbol{\lambda}}$
exhibits a H\'enon-type strange attractor near the orbit of tangency.
Theorem \ref{thrm:main} is proved.

\begin{figure}[h]
	\includegraphics[width=13cm]{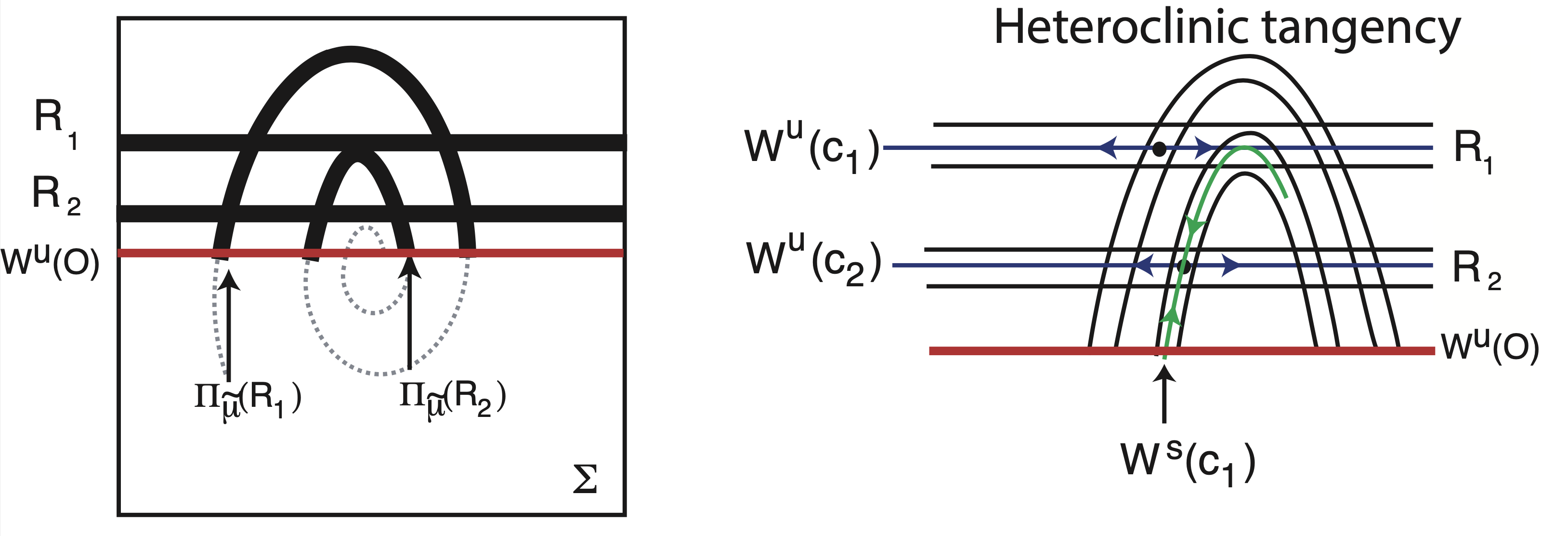}
    \caption{\small{\emph{Whiskers of the horseshoes:} sketch of the heteroclinic tangencies associated to two  saddles associated to the horseshoe $\Lambda_0$. The homoclinic classes associated to the horseshoes are not disjoint. }} \label{fig:escolhe_2}
\end{figure}

\begin{figure}[h]
\hspace{-10mm}
    \begin{subfigure}[t]{0.4\textwidth}\centering
        \includegraphics[width=6.5cm]{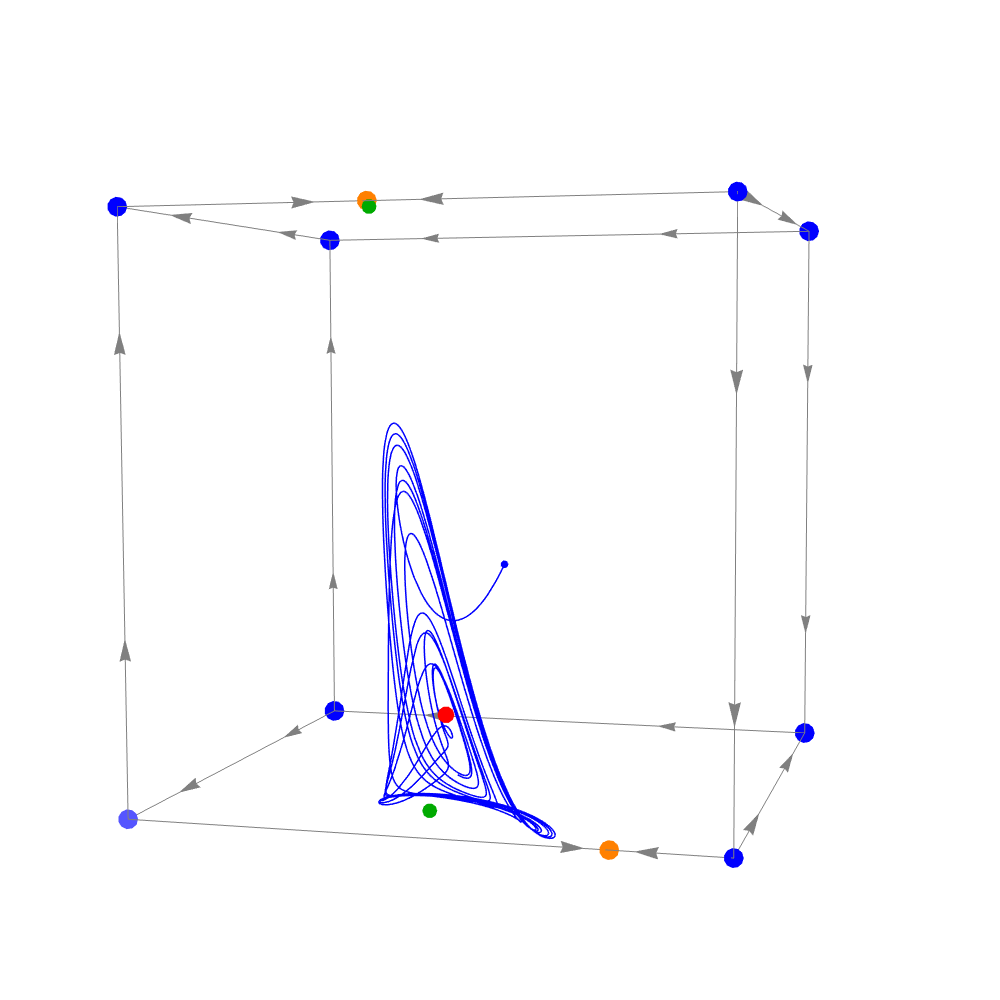}
        \vspace{-1.1cm}
        \label{fig:mu=3,6_t=40}
    \end{subfigure}
    \quad \quad
    \begin{subfigure}[t]{0.4\textwidth}\centering
        \includegraphics[width=6.5cm]{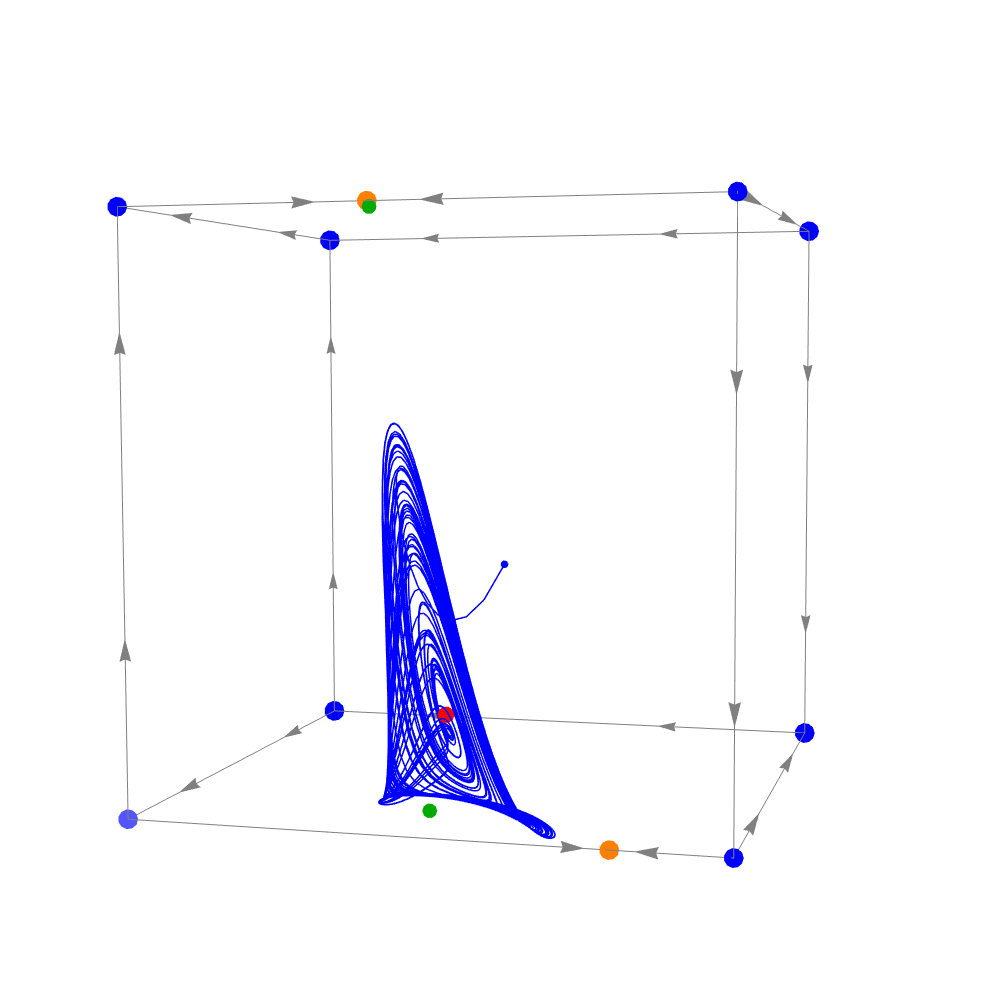}
        \vspace{-1.1cm}
        \label{fig:mu=3,6_t=220}
    \end{subfigure}
    \vspace{-.2cm}
    \caption{\small{\textit{Strange attractor:} plot of one orbit (in blue), the interior equilibrium (in red), and all the boundary equilibria (in the corresponding colors) of system~\eqref{eq:poly_rep_8}, for $\boldsymbol{\mu}=3.6$ with  $t\in [0,40]$ (left), and $t\in [0,170]$ (right).}}
    \label{fig:03_mu=3,6}
\end{figure}

\begin{rem}
Besides the existence of strange attractors in the family  $(g_{\boldsymbol{\lambda}})_{\boldsymbol{\lambda}\in[-1,1]}$, we may  apply Newhouse's results in the version   for one-parameter families \cite[Appendix 4]{palis1995hyperbolicity}, and conclude the existence of infinitely many values of  $\boldsymbol{\lambda} \in [-1,1]$ for which the associated flow exhibits a sink. 
\end{rem}


\section{Discussion}
\label{s: discussion}
In this section, we describe the phenomenological scenario behind the formation of strange attractors for the differential equation~\eqref{eq:poly_rep_8}, relating it with others in the literature.

\subsection{Attracting whirl\-pool: a phenomenological description}\label{subsec:att_whirlpool}

We describe the phenomenological scenario responsible for the  appearance of the strange attractor of Theorem \ref{thrm:main}.  We go back to the works \cite{shilnikov1981bifurcation, shil1995bifurcation} where a similar scenario was proposed for one-parameter families of three-dimensional flows in the context of an atmospheric model. 

 At  $\mu=\boldsymbol{\mu}_{\textrm{Belyakov}}$, the stable  interior  equilibrium $\mathcal{O}_{\boldsymbol{\mu}}\in \Sigma_\mu$   becomes focal. At $\mu=\boldsymbol{\mu}_{\textrm{Hopf}}>\boldsymbol{\mu}_{\textrm{Belyakov}}$, it  undergoes a supercritical Andronov-Hopf bifurcation, becoming an unstable saddle-focus, and a stable invariant curve $\mathcal{C}_{\boldsymbol{\mu}}$ is born in its neighborhood. The two-dimensional unstable invariant manifold of $\mathcal{O}_{\boldsymbol{\mu}}$, $W^u(\mathcal{O}_{\boldsymbol{\mu}})\subset \Sigma_{\boldsymbol{\mu}}$, is a topological disc limited by $\mathcal{C}_{\boldsymbol{\mu}}$.

After the emergence of the saddle-focus $B_3$ ($\boldsymbol{\mu}>b_3$ -- see Table \ref{tbl:Eigenv_of_Bs}), the set $\Sigma_{\boldsymbol{\mu}}$ starts rotating around $W^u(B_3)$ due to the complex (non-real) eigenvalues of $Df_{\boldsymbol{\mu}} \left( B_3 \right)$. When $\boldsymbol{\mu}$ increases further, the periodic solution  $\mathcal{C}_{\boldsymbol{\mu}}\subset \Sigma_{\boldsymbol{\mu}}$ approaches the cube's boundary and  winds  around $W^u(B_3)$, forming a structure similar to the so-called \emph{Shilnikov whirlpool} \cite{shilnikov1981bifurcation, shil1995bifurcation}.
For the $\boldsymbol{\mu}$ values associated to Cases IV and V, the equilibria $B_2$ and $B_3$  of this whirlpool are pulled to the face $\sigma_4$ and the orbits lying on the connected component of $\inter\left(\Gamma_{(2,2,2)}\right) \backslash \overline{W^s(B_1)}$ containing the interior equilibrium are tightened by this whirlpool. Further increasing $\boldsymbol{\mu}$, the  size  of the whirlpool decreases, and finally, at $\boldsymbol{\mu}\gtrsim\boldsymbol{\mu}_{\textrm{SA}}$, the set $W^s(\mathcal{O}_{\boldsymbol{\mu}})$ is arbitrarily close to the screw manifold $W^u(\mathcal{O}_{\boldsymbol{\mu}})$.
Fact \ref{hom2} claims the existence of a  homoclinic cycle of Shilnikov type for  a vector field $C^2$-close to $f_{\boldsymbol{\mu}_{\textrm{SA}}}$. The argument that supports this approximation, explained in Subsection \ref{ss: digestive}, is valid for the parameter values in the interval $] \boldsymbol{\mu}_{\textrm{SA}},\boldsymbol{\mu}^2_{\textrm{Hopf}} [$ for which the greatest LE of $f_{\boldsymbol{\mu}}$ is positive.

As $ \boldsymbol{\mu}$ evolves,  $\overline{W^s(B_1)}$ divides the cube in two connected components. The one which contains the strange attractor is shrinking, meaning that the volume of initial conditions which realize chaos is vanishing. In terms of the EGT, this would mean that, although system \eqref{eq:poly_rep_8} may exhibit chaos, the initial strategies that realize it are very close and cannot go too far (in an appropriate metric).

Our findings are different from those of \cite{shilnikov1981bifurcation} in the sense that, in  their case, the periodic solution coming from the Hopf Bifurcation becomes focal, playing the role of ``our'' saddle-focus $B_3$.

This type of mechanisms, the so called \emph{whirl\-pool attractor} may be seen as the universal scenario for the formation of Shilnikov cycles in a typical system \cite{chua2001methods} -- see for example the Rossler model and the ``new'' Lorenz two-parameter model \cite[Section 5]{shil1995bifurcation}.

\subsection{Open problem}
\label{ss: open}
Simulations with $\textit{Mathematica Wolfram}^{\circledR}$ suggest the existence of a homoclinic orbit to $\mathcal{O}_{\boldsymbol{\mu}}$, for $\boldsymbol{\mu}\approx \boldsymbol{\mu}_{SA}$ and $\boldsymbol{\mu}\approx 3.6$.
Nevertheless, for the same parameter values,
the homotopy method used by \textit{MatCont} indicates that the invariant manifolds of $\mathcal{O}_{\boldsymbol{\mu}}$ are close but do not touch.
This is why,  we have assumed in Fact \ref{hom2} the existence of a vector field exhibiting a homoclinic nearby.
It is an open problem to prove whether the Shilnikov homoclinic orbit exists in the family $\left( f_{\boldsymbol{\mu}}\right)_{\boldsymbol{\mu}}$, however  its arbitrary proximity to an element of the family $\left( f_{\boldsymbol{\mu}}\right)_{\boldsymbol{\mu}}$ is enough to trap the observable chaotic dynamics. We conjecture that, up to a change of coordinates,  the families $(g_{\boldsymbol{\lambda}})_{\boldsymbol{\lambda} }$ and $(f_{\boldsymbol{\mu}})_{\boldsymbol{\mu}}$ coincide.
 We are trying hard to  find out an answer for this problem.

\subsection{Final Remark}
 \label{ss: final remark}
We  introduced a one-parameter family of polymatrix replicators defined on $\Gamma_{(2,2,2)}$ and study its bifurcations in detail.
 In an open interval of the  parameter space, we prove the existence of a vector field, $C^2$-close to elements of the family  $(f_{\boldsymbol{\mu}})_{\boldsymbol{\mu}}$,  exhibiting a homoclinic cycle to a saddle-focus,  responsible for the emergence of suspended horseshoes and persistent observable chaos  (H\'enon-type in the sense of \cite{mora1993abundance}).
This family is structurally stable in the sense that small perturbations of the non-zero entries of $P_{\boldsymbol{\mu}}$ do not qualitatively change the dynamics.

The mechanism responsible for the emergence of chaos seems to be the same for a large class of examples: we obtain an attracting limit cycle (from a supercritical Hopf bifurcation) limiting the unstable manifold of an unstable focus. The stable manifold of the  limit cycle starts winding around a focus and  accumulates on the stable manifold of the interior equilibrium, undergoing successive saddle-node and period-doubling bifurcations.
This criterion relies on Shilnikov's results \cite{shilnikov1965case}. It creates strange attractors that may be seen as a suspension of H\'enon-type diffeomorphisms. In particular, when the parameter varies, on a typical cross section,   topological horseshoes emerge linked with sinks \cite{newhouse1974diffeomorphisms}.

The reduction of a polymatrix replicator to dimension three may be carried out just in two cases, 
$\Gamma_{(2,2,2)}$ (our model) and $\Gamma_{(3,2)}$ (the population is divided in two groups, one with three available strategies and the other with two).
The search of strange attractors in the second case may be performed in the same way as we have done in the present article.

The existence of strange attractors in polymatrix replicators has profound implications in the setting of the EGT.
Observable chaos is the result of a strategy evolution in which individuals are constantly changing their plans of action in an unforeseeable way.
The existence of chaos for model \eqref{eq:poly_rep} is relevant  to maintain the complexity and diversity of strategies, in particular their high unpredictability \cite{kaneko1993chaos}. 

Although a complete understanding of the   bifurcation diagram associated to \eqref{eq:poly_rep_8} and the mechanisms underlying the dynamical changes is out of reach, we uncover complex patterns for the one-parameter family under analysis, using a combination of theoretical tools and computer simulations. 

The dynamics on the interior of the cube is strongly determined by the dynamics on the boundary. A lot more needs to be done before the dynamics  of polymatrix replicators  is completely understood.

\section*{Acknowledgements}
We are grateful to Hil Meijer (University of Twente, Netherlands) and Willy Govaerts (University of Ghent, Belgium) for the numerical simulations in \textit{MatCont}.

The authors are indebted to J. P. Gaiv\~ao for his remarks and to Pedro Duarte for the development of our Mathematica code. 
They are also grateful to the two reviewers for the constructive comments, corrections and suggestions which helped to improve the readability of the article.
The first author was supported by
the Project CEMAPRE/REM -- UIDB /05069/2020 financed by FCT/MCTES through national funds.

The second author was partially supported by CMUP (Project reference: UID/MAT/ 00144/2019), which is funded by FCT with national (MCTES) and European structural funds through the programs FEDER, under the partnership agreement PT2020. He also acknowledges financial support from Program INVESTIGADOR FCT (IF/ 0107/ 2015).


\bibliographystyle{unsrt}
\bibliography{Bibfile}

\pagebreak

\appendix

\section{Tables}
\label{tables_appendix}

\begin{table}[h]
\begin{small}
\begin{tabular}{|c|c|c|c|c|c|} \toprule
Eq.       & Eigenvalues  	& $\boldsymbol{\mu}$		&  On edge        & On $\sigma_3$	& On $\sigma_5$  \\ \toprule 
$A_1$  & \small{$\left\{ \frac{10 (\boldsymbol{\mu} +38)}{14-\boldsymbol{\mu}},\frac{4 (5 \boldsymbol{\mu} +281)}{\boldsymbol{\mu} -14},\frac{26 \left(\boldsymbol{\mu} ^2-2 \boldsymbol{\mu} -168\right)}{-(\boldsymbol{\mu} -14)^2} \right\} $}
		   & $\left[ -\frac{2938}{95}, -12 \right[$ & $-$ & $-$ & $+$ \\ \bottomrule
           & 						 &     					  			    &  On edge        & On $\sigma_4$	 & On $\sigma_5$  \\ \toprule
\multirow{5}{*}{$A_2$} & \multirow{5}{*}{\small{$\left\{ \frac{6 (\boldsymbol{\mu} +6)}{\boldsymbol{\mu} -14},\frac{31 \boldsymbol{\mu} -110}{\boldsymbol{\mu} -14},\frac{6 \left(\boldsymbol{\mu} ^2-22 \boldsymbol{\mu} +112\right)}{-(\boldsymbol{\mu} -14)^2} \right\}$}}
				& $\left[ -\frac{2938}{95}, -6 \right[$ 	& $-$ & $+$ & $+$ \\ \cmidrule{3-6}
			&	& $\{ -6 \}$ 											& $-$ & $+$ & $0$ \\ \cmidrule{3-6}
			&	& $\left] -6, \frac{110}{31} \right[$ 		& $-$ & $+$ & $-$ \\ \cmidrule{3-6}
			&	& $\{ \frac{110}{31} \}$ 							& $-$ & $0$ & $-$ \\ \cmidrule{3-6}
			&	& $\left] \frac{110}{31}, 8 \right[$  		& $-$ & $-$ 	& $-$ \\ \bottomrule
	        & 						 & 								    &  On edge        & On $\sigma_4$ 	& On $\sigma_6$  \\ \toprule
$A_3$  & \small{$\left\{ \frac{27 (\boldsymbol{\mu} -10)}{14-\boldsymbol{\mu}},\frac{10 (\boldsymbol{\mu} -10)}{\boldsymbol{\mu} -14},\frac{2 \left(\boldsymbol{\mu} ^2-26 \boldsymbol{\mu} +168\right)}{-(\boldsymbol{\mu} -14)^2} \right\}$}
		   & $\left[ -\frac{2938}{95}, 10 \right]$ 	& $-$ & $-$ & $+$ \\ \bottomrule
            & 					  	& 								  	    &  On edge        & On $\sigma_3$ 	& On $\sigma_6$  \\ \toprule
\multirow{3}{*}{$A_4$} 	& \multirow{3}{*}{\small{$\left\{ \frac{4 (4 \boldsymbol{\mu} +241)}{14-\boldsymbol{\mu}},\frac{2 (7 \boldsymbol{\mu} +122)}{14-\boldsymbol{\mu}},\frac{22 \left(\boldsymbol{\mu} ^2-6 \boldsymbol{\mu} -112\right)}{-(\boldsymbol{\mu} -14)^2} \right\}$}}
				& $\left[ -\frac{2938}{95}, -\frac{122}{7} \right[$ 	& $-$	& $+$	& $-$	\\ \cmidrule{3-6}
			&	& $\{ -\frac{122}{7} \}$ 										 	& $-$	& $+$	& $0$	\\ \cmidrule{3-6}
			&	& $\left] -\frac{122}{7}, -8 \right[$ 					 	& $-$	& $+$	& $+$	\\ \bottomrule
\end{tabular}
\end{small}
\vspace{.3cm}
        \captionof{table}{\small{Eigenvalues of equilibria $A_1$, $A_2$, $A_3$, and $A_4$ for system~\eqref{eq:poly_rep_8}, on the corresponding edges and faces (pointing to the interior), where the signs $(-)$, $(0)$, and $(+)$ mean that the eigenvalues are real negative, zero, or positive, respectively. }}
        \label{tbl:Eigenv_of_As}
\end{table}

\newpage

\begin{table}[h]
\begin{small}
\begin{tabular}{|c|c|c|c|c|} \toprule
Eq.               & Eigenvalues & $\boldsymbol{\mu}$      			    &  On face        & On the interior  \\ \toprule
\multirow{2}{*}{$B_1$} & \multirow{2}{*}{$\left\{ \frac{37 (\boldsymbol{\mu} -10)}{\boldsymbol{\mu} +40}, z_1, \bar{z}_1 \right\}$}
									& $\{ \frac{110}{31} \}$ \,\,\,$(B_1=A_2)$ 				& $(-,0)$ 	& $(-)$ \\ \cmidrule{3-5}
                                  &  & $\left] \frac{110}{31}, 10 \right] $    						& $(-,+)$  & $(-)$ \\ \midrule

\multirow{3}{*}{$B_2$} & \multirow{3}{*}{$\left\{ \frac{95 \boldsymbol{\mu} +2938}{2 (\boldsymbol{\mu} +86)}, z_2, \bar{z}_2 \right\}$}
									& $\left[ -\frac{2938}{95}, b_2 \right[ \cup \left[ -12,-6 \right[ $ & $(-,-)$ 		& $(+)$ \\ \cmidrule{3-5}
								&	& $\left[ b_2, -12 \right[ $  														 & $(-,-)_\textbf{C}$ & $(+)$ \\ \cmidrule{3-5}
								&	& $\{ -6 \}$ \,\,\,$(B_2=A_2)$ 													 & $(-,0)$ 		& $(+)$ \\ \midrule  

\multirow{3}{*}{$B_3$} & \multirow{3}{*}{$\left\{ \frac{109 (10-\boldsymbol{\mu})}{2 (\boldsymbol{\mu} +86)}, z_3, \bar{z}_3 \right\}$}
									& $\{ -\frac{122}{7} \}$ \,\,\,$(B_3=A_4)$ 				& $(-,0)$ 		& $(+)$ \\ \cmidrule{3-5}
								&  & $\left] -\frac{122}{7}, b_3 \right]$ 						& $(-,-)$ 		& $(+)$ \\ \cmidrule{3-5}								
                                  &  & $\left[ b_3, 10 \right] $    										& $(-,-)_\textbf{C}$ 	& $(+)$  \\ \bottomrule
\end{tabular}
\end{small}

\begin{tiny}
\begin{align*}
z_1 &= \frac{ 8700-11240 \boldsymbol{\mu} +937 \boldsymbol{\mu}^2+\sqrt{ -7052310000+1872624000 \boldsymbol{\mu} +179361400 \boldsymbol{\mu}^2 -34941760 \boldsymbol{\mu}^3 + 543169 \boldsymbol{\mu}^4 }}{8 (\boldsymbol{\mu} +40)^2} \\
z_2 &= \frac{3 \left( -8084+164 \boldsymbol{\mu} + 3 \boldsymbol{\mu}^2+\sqrt{ 259862416+40284768 \boldsymbol{\mu} +1909912 \boldsymbol{\mu}^2 +31704 \boldsymbol{\mu}^3 + 169 \boldsymbol{\mu}^4}\right)}{2 (\boldsymbol{\mu} +86)^2} \\
z_3 &= \frac{ -19436 +892 \boldsymbol{\mu} + 13 \boldsymbol{\mu}^2+\sqrt{ -3449723504 -16764064 \boldsymbol{\mu} +27270168 \boldsymbol{\mu}^2 +906392 \boldsymbol{\mu}^3
+ 6889 \boldsymbol{\mu}^4 }}{2 (\boldsymbol{\mu} +86)^2}
\end{align*}
\end{tiny}

        \captionof{table}{\small{Eigenvalues of equilibria $B_1$, $B_2$, and $B_3$ for system~\eqref{eq:poly_rep_8}, on the corresponding faces and pointing to the interior, where $b_2\approx-21.9$ and $b_3\approx-14.22$. The signs   $(-)$, $(0)$, and $(+)$ mean that the eigenvalues are real negative, zero, or positive, respectively, and $(-,-)_\textbf{C}$ means that the eigenvalues are conjugate (non-real) with negative real part. 
}}
        \label{tbl:Eigenv_of_Bs}
\end{table}

\newpage

\begin{table}[h]
\begin{small}
\begin{tabular}{|c|c|c|c|} \toprule
Eq.  & Eigenvalues & $\boldsymbol{\mu}$      			    &  Analysis  \\ \toprule
$v_1$          & $\left\{ -10, 27, 12-\boldsymbol{\mu} \right\}$ & $\left[-\frac{2938}{95}, 10 \right]$ & $(-,+,+)$  \\ \midrule

\multirow{3}{*}{$v_2$} & \multirow{3}{*}{$\left\{ -23, -14, 8-\boldsymbol{\mu} \right\}$} & $\left[-\frac{2938}{95}, 8 \right[$ & $(-,-,+)$  \\ \cmidrule{3-4}
	& & $\{ 8\} $  & $(-,-,0)$  \\ \cmidrule{3-4}
	& & $\left] 8, 10 \right] $  & $(-,-,-)$ \\ \midrule

\multirow{3}{*}{$v_3$} & \multirow{3}{*}{$\left\{ 6, 38, -8-\boldsymbol{\mu} \right\}$} & $\left[-\frac{2938}{95}, -8 \right[$ & $(+,+,+)$ \\ \cmidrule{3-4}
	& & $\{ -8\} $  & $(0,+,+)$ \\ \cmidrule{3-4}
	& & $\left] -8, 10 \right] $  & $(-,+,+)$ \\ \midrule
								
\multirow{3}{*}{$v_4$} & \multirow{3}{*}{$\left\{ -34, 10, -12-\boldsymbol{\mu} \right\}$} & $\left[-\frac{2938}{95}, -8 \right[$ & $(-,+,+)$ \\ \cmidrule{3-4}
	& & $\{ -12 \} $  & $(-,0,+)$  \\ \cmidrule{3-4}
	& & $\left] -12, 10 \right] $  & $(-,-,+)$ \\ \midrule

$v_5$  & $\left\{ -27, 2, 10 \right\}$     & $\left[-\frac{2938}{95}, 10 \right]$ & $(-,+,+)$  \\ \midrule

$v_6$  & $\left\{ 6, 6, 31 \right\}$         & $\left[-\frac{2938}{95}, 10 \right]$ & $(+,+,+)$  \\ \midrule

$v_7$  & $\left\{ -14, -16, 22 \right\}$ & $\left[-\frac{2938}{95}, 10 \right]$ & $(-,-,+)$  \\ \midrule

$v_8$  & $\left\{ -10, 20, 26 \right\}$   & $\left[-\frac{2938}{95}, 10 \right]$ & $(-,+,+)$  \\ \bottomrule																
\end{tabular}
\end{small}
\vspace{.3cm}
        \captionof{table}{\small{Real eigenvalues of the vertices $v_1, \dots, v_8$ for system~\eqref{eq:poly_rep_8}, where  $(-)$, $(0)$, and $(+)$ mean that they are negative, zero, or positive, respectively.}}
        \label{tbl:Eigenv_of_vs}
\end{table}

\newpage

\section{Summarizing movie of the boundary dynamics}
\label{bd_appendix}

\begin{table}[h]
\begin{tabular}{|c|c|c|} \toprule
\includegraphics[width=4cm]{./figures/Int_1.png} & \includegraphics[width=4cm]{./figures/Int_2.png}  & \includegraphics[width=4cm]{./figures/Int_3.png} \\[-4mm]
\small{$(\RN{1})$} & \small{$(\RN{2})$} & \small{$(\RN{3})$} \\[1mm]
\small{$\boldsymbol{\mu} \in \left[-\frac{2938}{95}, -\frac{122}{7} \right]$} & \small{$\boldsymbol{\mu} \in \left] -\frac{122}{7}, -12 \right]$} &
\small{$\boldsymbol{\mu} \in \left] -12, -8 \right]$} \\ \midrule

\includegraphics[width=4cm]{./figures/Int_4.png} & \includegraphics[width=4cm]{./figures/Int_5.png} & \includegraphics[width=4cm]{./figures/Int_6.png} \\[-4mm]
\small{$(\RN{4})$} & \small{$(\RN{5})$} & \small{$(\RN{6})$} \\[1mm]
\small{$\boldsymbol{\mu} \in \left] -8, -6 \right]$} & \small{$\boldsymbol{\mu} \in \left] -6, \frac{110}{31} \right]$} & \small{$\boldsymbol{\mu} \in \left] \frac{110}{31}, 8 \right]$} \\ \midrule

\includegraphics[width=4cm]{./figures/Int_7.png} &  & \\[-4mm]
\small{$(\RN{7})$} &  & \\[1mm]
\small{$\boldsymbol{\mu} \in \left] 8, 10 \right]$} &  & \\ \bottomrule
\end{tabular}
	\vspace{.3cm}
        \captionof{table}{\small{Illustration of the dynamics on the cube's boundary on each interval of $\boldsymbol{\mu}$ in different cases of Table \ref{tbl:cases}.}}
        \label{tbl:Bd_dynamics_on_mu}
\end{table}

\newpage

\section{Summarizing movie of the interior dynamics}
\label{int_appendix}

\begin{table}[h]
\begin{tabular}{|c|c|c|} \toprule
\includegraphics[width=4cm]{./figures/04_mu=-20_01.png} & \includegraphics[width=4cm]{./figures/01_mu=-17,5.png}  & \includegraphics[width=4cm]{./figures/01_mu=-14_01.png} \\[-4mm]
\small{$\boldsymbol{\mu}=-20$} & \small{$\boldsymbol{\mu}=-17.5$} &
\small{$\boldsymbol{\mu}=-14$} \\ \midrule

\includegraphics[width=4cm]{./figures/01_mu=-8,5.png} & \includegraphics[width=4cm]{./figures/01_mu=-7.png} & \includegraphics[width=4cm]{./figures/01_mu=1,1_01.png} \\[-4mm]
\small{$\boldsymbol{\mu}=-8.5$} & \small{$\boldsymbol{\mu}=-7$} & \small{$\boldsymbol{\mu}=1.1$} \\ \midrule

\includegraphics[width=4cm]{./figures/01_mu=3,6.png} & \includegraphics[width=4cm]{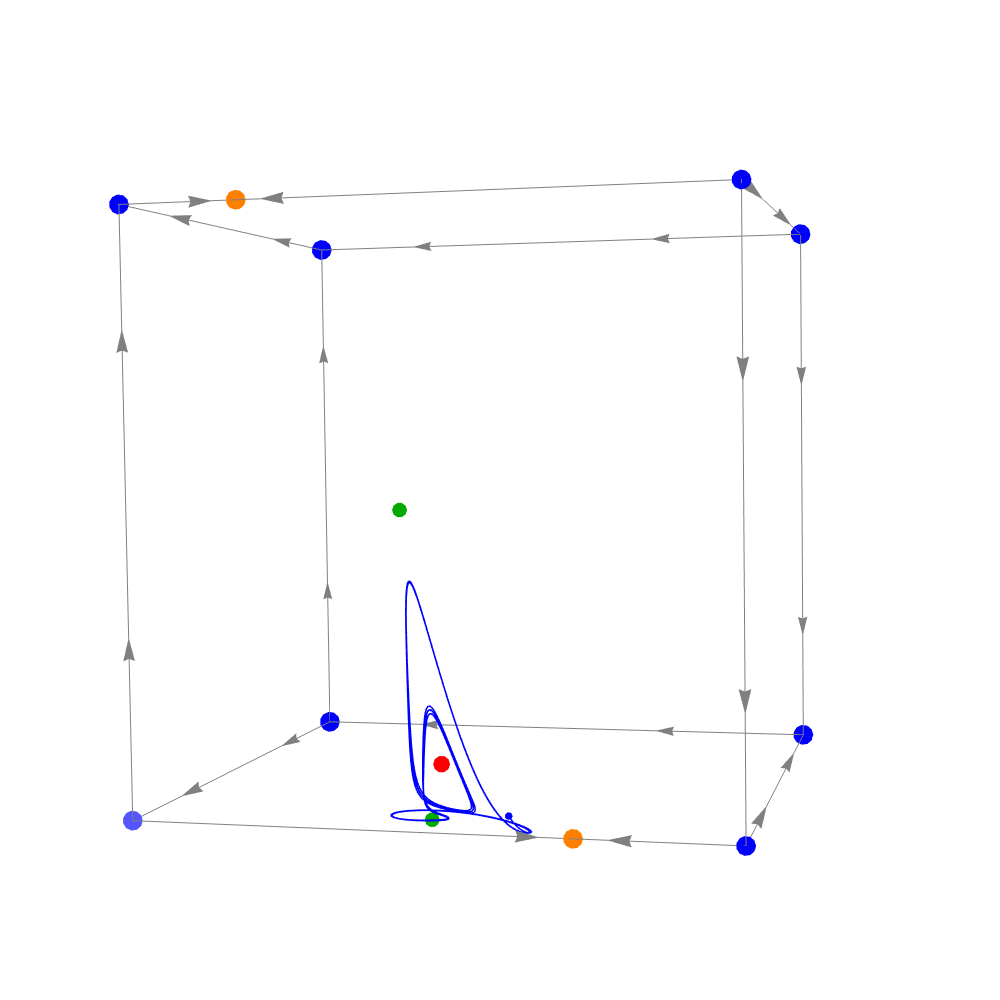} & \includegraphics[width=4cm]{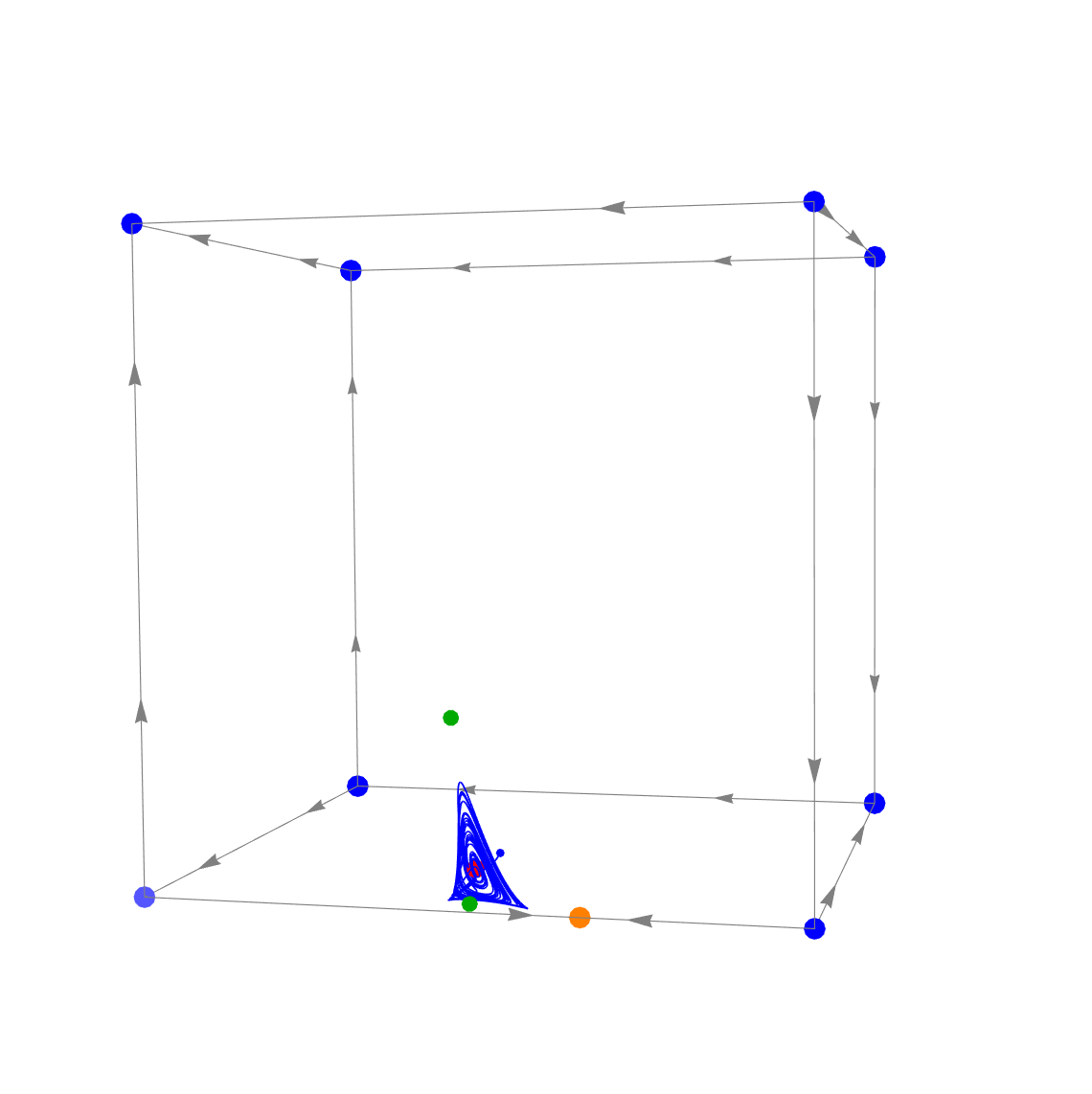} \\[-4mm]
\small{$\boldsymbol{\mu}=3.6$} & \small{$\boldsymbol{\mu}=6.5$} & \small{$\boldsymbol{\mu}=8$} \\ \bottomrule
\end{tabular}
	\vspace{.3cm}
        \captionof{table}{\small{Illustration of the dynamics on the cube's interior for different values of $\boldsymbol{\mu}$.}}
        \label{tbl:Int_dynamics_on_mu}
\end{table}

\end{document}